\documentclass{amsproc}
\usepackage{euscript,graphicx} 
\usepackage{cases}
\usepackage{mathrsfs}
\usepackage{bbm}
\usepackage{amssymb}
\usepackage{txfonts}
\usepackage{amsfonts,amsmath,amsxtra,mathdots,mathabx}
\usepackage{color}
\usepackage{hyperref}
\usepackage{tikz}
\usepackage{appendix}

\allowdisplaybreaks

\DeclareFontFamily{U}{matha}{\hyphenchar\font45}
\DeclareFontShape{U}{matha}{m}{n}{
	<5> <6> <7> <8> <9> <10> gen * matha
	<10.95> matha10 <12> <14.4> <17.28> <20.74> <24.88> matha12
}{}
\DeclareSymbolFont{matha}{U}{matha}{m}{n}

\DeclareMathSymbol{\Lt}{3}{matha}{"CE}
\DeclareMathSymbol{\Gt}{3}{matha}{"CF} 

\DeclareSymbolFont{mathc}{OML}{txmi}{m}{it}
\DeclareMathSymbol{\txv}{\mathord}{mathc}{118}
\DeclareMathSymbol{\txw}{\mathord}{mathc}{119}

\DeclareSymbolFont{mathd}{OML}{ztmcm}{m}{it}
\DeclareMathSymbol{\varalpha}{\mathord}{mathd}{11} 
\DeclareMathSymbol{\vartau}{\mathord}{mathd}{28} 
\DeclareMathSymbol{\varlambda}{\mathord}{mathd}{21}
\def\valpha{\text{\scalebox{0.84}{$\varalpha$}}}
\def\vtau{\text{\scalebox{0.94}{$\vartau$}}}
\def\vlambda{\text{\scalebox{0.86}[0.92]{$\varlambda$}}} 

\def\vepsilon{\text{\scalebox{0.88}[1.04]{$\varepsilon$}}}

\newcommand{\BA}{{\mathbb {A}}} 
\newcommand{\BC}{{\mathbb {C}}} \newcommand{\BH}{{\mathbb {H}}}
\newcommand{\BQ}{{\mathbb {Q}}} \newcommand{\BR}{{\mathbb {R}}} \newcommand{\BZ}{{\mathbb {Z}}}

\newcommand{\RF}{{\mathrm {F}}}
\newcommand{\RM}{{\mathrm {M}}}

\def\EuH{\text{\usefont{U}{eus}{m}{n}H}}
\def\EuM{\text{\usefont{U}{eus}{m}{n}M}}

\def\CaloO {\text{\raisebox{- 2 \depth}{\scalebox{1.1}{$ \text{\usefont{U}{BOONDOX-calo}{m}{n}O}  $}}}} 
 \def\SC{\text{$\text{\usefont{U}{BOONDOX-calo}{m}{n}C}$}}
\def\ScrL{\text{$\text{\usefont{U}{BOONDOX-calo}{m}{n}L}$}}
\def\SE{\text{$\text{\usefont{U}{BOONDOX-calo}{m}{n}E}$}}
\def\SD{\text{$\text{\usefont{U}{BOONDOX-calo}{m}{n}D}$}}
\def\ScrO{\text{$\text{\usefont{U}{BOONDOX-calo}{m}{n}O}$}}
\def\ScrS{\text{$\text{\usefont{U}{BOONDOX-calo}{m}{n}S}$}}
\def\SM{\text{$\text{\usefont{U}{BOONDOX-calo}{m}{n}M}$}}
\def\SR{\text{$\text{\usefont{U}{BOONDOX-calo}{m}{n}R}$}}
\def\SP{\text{$\text{\usefont{U}{BOONDOX-calo}{m}{n}P}$}}
\def\SZ{\text{$\text{\usefont{U}{BOONDOX-calo}{m}{n}Z}$}}

\newcommand{\GL}{{\mathrm {GL}}}
\newcommand{\SL}{{\mathrm {SL}}}

\newcommand{\ra}{\rightarrow}

\def\lp {\left (}
\def\rp {\right )}
\def\Voronoi{Vorono\"{i}   } 
\def\boldJ {\boldsymbol{B}}
\def\boldF{\boldsymbol{H}} 
\renewcommand{\Im}{{\mathrm{Im} }}
\renewcommand{\Re}{{\mathrm{Re} }} 
\def\shskip{\hskip 0.5pt}

\def\vwedge{\hskip -1pt \wedge \hskip -1pt}
\def\viint{\int \hskip -5pt \int}

\newcommand{\delete}[1]{}

\theoremstyle{plain}

\newtheorem{thm}{Theorem}[section] \newtheorem{cor}[thm]{Corollary}
\newtheorem{lem}[thm]{Lemma}  \newtheorem{prop}[thm]{Proposition}

\newtheorem {conj}[thm]{Conjecture} \newtheorem{defn}[thm]{Definition}

\newtheorem {rem}[thm]{Remark}

\newtheorem*{acknowledgement}{Acknowledgements}

\numberwithin{equation}{section}

\begin{document}

	\title[On the Hankel Transform of Bessel Functions]{{On the Hankel Transform of Bessel Functions on Complex Numbers and Explicit Spectral Formulae over the Gaussian Field}}

\begin{abstract}
	
	In this paper, on the complex field $\mathbb{C}$, we prove two integral formulae for the Hankel--Mellin transform and the double Fourier--Mellin transform of Bessel functions, both resulting the hypergeometric function. As two applications, we use the former integral formula to make explicit the spectral formula of Bruggeman and Motohashi for the fourth moment of Dedekind zeta function over the Gaussian number field $\mathbb{Q}(i)$ and to establish a spectral formula for the Hecke-eigenvalue twisted second moment of central $L$-values for the Picard group $\mathrm{PGL}_2 (\mathbb{Z}[i])$. Moreover, we develop the theory of distributional Hankel transform on $\mathbb{C} \smallsetminus \{0\}$. 
	 
\end{abstract}

\author{Zhi Qi}
\address{School of Mathematical Sciences\\ Zhejiang University\\Hangzhou, 310027\\China}
\email{zhi.qi@zju.edu.cn}

	\thanks{The author was supported by National Key R\&D Program of China No. 2022YFA1005300 and National Natural Science Foundation of China No. 12071420.}

\dedicatory{\normalsize On the occasion of James W. Cogdell's 70th birthday}

\subjclass[2010]{11F12, 11F66, 33C05, 33C10, 42B10}
\keywords{Bessel functions, Hankel transform, hypergeometric function, $L$-functions}

\maketitle

 {\small \tableofcontents}

\section{Introduction}

The theory of  Bessel functions contains abundant elegant  integral formulae, a long list of which may be found in \cite[\S \S 6.5--6.7]{G-R} (95 pages). From the perspective of number theory and representation theory, certain classical integral formulae of Weber, Hardy, and Schafheitlin for the Fourier transform of Bessel functions may now be interpreted over $\BR$ as the local  Shimura--Waldspurger correspondence in the Waldspurger formula or the local ingredient for the Beyond Endoscopy for $\mathrm{Sym}^2$ on $\mathrm{PGL}_2$. The reader is referred to \cite{BaruchMao-Real}, \cite{BaruchMao-Global}, \cite[\S 6.7]{Venkatesh-BeyondEndoscopy}, and \cite[\S 1.1]{Qi-BE} for the details. 

For instance, the formulae of Weber and Hardy read: 
\begin{equation}\label{0eq: Weber's formula}
	\int_0^\infty \frac 1 {\sqrt x}  J_{\nu}   (4 \pi \sqrt x)  e  ( {\pm x y}  )   {d      x}  = \frac 1 {\sqrt {2 y} } e \bigg( {\mp  \bigg( \frac 1 {2 y} - \frac 1 8  \nu - \frac 1 8 \bigg)} \hskip -1pt \bigg) J_{\frac 1 2 \nu} \bigg( \frac {\pi} y \bigg),
\end{equation}
for $\Re (\nu) > -1$ (see \cite[6.725, 1, 2]{G-R} or \cite[1.13 (25), 2.13 (27)]{ET-I}), and 
\begin{equation}\label{0eq: Hardy's formula}
	\begin{split}
		\int_0^\infty \frac 1 {\sqrt x}  K_{\nu}   (4 \pi \sqrt x)  & e   ({\pm x y} )    {d      x}   =   - \frac {\pi} {2 \sin ( \pi \nu)}  \frac 1 {\sqrt {2 y} } e \bigg( {\pm  \bigg( \frac 1 {2 y} + \frac 1 8 \bigg)} \bigg) \\
		& \quad \cdot \bigg\{  e \bigg( {\pm \frac 1 8 \nu} \bigg) J_{\frac 1 2 \nu} \bigg( \frac {\pi} y \bigg) - 
		e \bigg( {\mp \frac 1 8 \nu} \bigg) J_{- \frac 1 2 \nu} \bigg( \frac {\pi} y \bigg)   \bigg\},
	\end{split}
\end{equation}
for  $|\Re (\nu) | < 1$ (see \cite[1.13 (28), 2.13 (30)]{ET-I}), while the formulae of  Weber and Schafheitlin read:
\begin{equation}\label{0eq: Weber-Sch for J}
	\begin{split}
		\int_0^\infty  J_{\nu} (4 \pi x)  e  (\pm x y)  & x^{\shskip \rho - 1 }  d      x = \frac {\Gamma (\rho+\nu) e^{\pm \frac 1 2 \pi i (\rho+\nu)}  } {(2\pi)^{ \rho} \Gamma (1+\nu) y^{\shskip \rho+\nu } } {_2F_1} \hskip -1pt \bigg( \frac {\rho + \nu} 2, \frac {1+\rho+\nu} 2  ; 1+ \nu ;  \frac 4 {y^2}  \bigg) ,
	\end{split}
\end{equation}
in the case $y> 2$, for $ - \Re (\nu) < \Re (\rho) < 3/2 $ (see \cite[6.699 1, 2]{G-R}), and 
\begin{equation}\label{0eq: Weber-Sch for K}
	\begin{split}
		\int_0^\infty  K_{\nu} (4 \pi x)  e  (\pm x y)  & x^{\shskip \rho - 1 }  d      x =    \\
		- \frac {\pi} {2 (2 \pi)^{  \rho} \sin (\pi \nu)}  \bigg\{ &  \frac {       \Gamma (\rho+\nu) e^{\pm \frac 1 2 \pi i (\rho+\nu)}  } {  \Gamma (1+\nu) y^{\shskip \rho + \nu}    }  {_2F_1} \hskip -1pt \bigg(  \frac {\rho + \nu} 2, \frac {1+\rho+\nu} 2  ; 1+ \nu ; - \frac 4 {y^2}  \bigg)  \\
		- & \frac {       \Gamma (\rho-\nu) e^{\pm \frac 1 2 \pi i (\rho-\nu)} } {  \Gamma (1-\nu) y^{\shskip \rho - \nu}    }  {_2F_1} \hskip -1pt  \bigg(  \frac {\rho-\nu} 2 , \frac {1+\rho - \nu} 2 ; 1-\nu  ; - \frac 4 {y^2}  \bigg)    \bigg\}, 
	\end{split}
\end{equation}
for $ |\Re (\nu)| < \Re (\rho)$ (see \cite[6.699 3, 4]{G-R} or \cite[(1.12)]{Qi-BE}).  As usual, $e (x) = e^{2\pi i x}$, $J_{\nu} (z)$ and $K_{\nu} (z)$ are the Bessel functions of order $\nu$ as in \cite{Watson}, and $ {_2F_1} (a, b; c; z)$ is the hypergeometric function of Gauss. 

The works \cite{BaruchMao-Global} and \cite{Venkatesh-BeyondEndoscopy} were restricted to {\it totally real} number fields, since, at that time,  the knowledge of Bessel functions for $\GL_2 (\BC)$ was very limited, not to speak of any Bessel integral formula  on $\BC$ (all the integrals in \cite[\S \S 6.5--6.7]{G-R} are essentially on $\BR_+$ or a sub-interval of $\BR_+$). 

Recently, by a rather indirect method that combines asymptotic analysis and differential equations, the formulae in  \eqref{0eq: Weber's formula}--\eqref{0eq: Weber-Sch for K} have been  extended successfully over $\BC$  in \cite{Qi-Sph,Qi-II-G,Qi-BE}, so that the Waldspurger formula in \cite{BaruchMao-Global} and the Beyond Endoscopy in \cite{Venkatesh-BeyondEndoscopy} are now valid over {\it arbitrary} number fields as in \cite{Chai-Qi-Wald} and \cite{Qi-BE}.\footnote{The Waldspurger formula requires all the local constituents of the relative trace formula of Jacquet \cite{Jacquet-RTF}  over $p$-adic fields, $\BR$, and $\BC$ established in \cite{BaruchMao-NA}, \cite{BaruchMao-Real}, and \cite{Chai-Qi-Bessel} respectively.}

\subsection{Hankel transform in the formulae of Motohashi and Kuznetsov} \label{sec: spectral}

The present paper is in a similar spirit---the motive is to extend other useful formulae from $\BQ$ to number fields. This time, however, we shall be modest and restrict to the Gaussian number field $\BQ(i)$. 
The formulae considered here are 
\begin{itemize}
	\item [(1)] the spectral formula of Motohashi for the fourth moment of the Riemann $\zeta$ function \cite{Motohashi-Riemann-4th,Motohashi-Riemann}, and 
	\item [(2)] the spectral formula of Kuznetsov and Motohashi for the twisted second moment of $\SL_2(\BZ)$ Maass-form $L$-functions \cite{Kuznetsov-Motohashi-formula,Motohashi-JNT-Mean,Motohashi-Riemann}. 
\end{itemize}
 
To describe our objectives, instead of stating these formulae in full, we shall only review and discuss their   ingredients that are related to the hypergeometric function or the Hankel transform of Bessel functions. This, however, would still occupy several pages, so the readers who are familiar with these materials over $\BQ$ may skip safely to \S \ref{sec: purpose}. 
 
 Also, there are no holomorphic forms over $\SL_2 (\BZ(i))$ as there is no discrete series on $\SL_2 (\BC)$.  As such, for the Motohashi formula, we shall not focus on the contribution from holomorphic forms. Note that, for the Kuznetsov--Motohashi formula,   there is also a counterpart in the holomorphic case due to Kuznetsov (see \cite[\S 4.2]{BF-Moments}). 

To start with, we introduce the space of test functions. 
\begin{defn}\label{1defn: H}
	Let $\mathscr{H} $ be the space of even entire functions 
	that  decays rapidly on any given  horizontal strip. 
\end{defn}

\subsubsection{The formula of Motohashi}

Recall that the formula of Motohashi \cite[Theorem 4.2]{Motohashi-Riemann} expresses the weighted fourth moment of the Riemann $\zeta (s)$ on the critical line 
\begin{align*}
\SZ_2 (  \varg)  = \int_{-\infty}^{\infty} \left|\zeta  \big(\tfrac 12+it \big)\right|^4 \varg (t) d      \shskip t , 
\end{align*}
in terms of the cubic moment of $ L \big(\frac 1 2, f \big)$ for $f$ in the full spectrum of $ \mathrm{PSL}_2 (\BZ) \backslash \BH $, including Hecke--Maass cusp forms (even or odd), Hecke holomorphic cusp forms, and the Eisenstein series. There arise weights in the form of the integral transform:
\begin{equation}\label{0eq: Lambda}
	\Lambda (r; \varg) = \int_0^{\infty}  \frac {{\varg}_c (\log (1+1/y)) } {\sqrt{y (1+y)}} \Xi (y; i r) d      y, 
\end{equation}
where   $r  = r_f $ is the spectral parameter of $f$, with $i r_f$ imaginary or a positive half-integer, 
\begin{align}
	\varg_c (x) = \int_{-\infty}^{\infty} \varg (t) \cos (t x) d      \shskip t 
\end{align}
is the cosine Fourier transform, and 
\begin{align}\label{0eq: Xi}
	\Xi (y; \nu) = \Re \bigg\{  \bigg(1 - \frac 1 {\sin (\pi\nu)} \bigg) \frac {\Gamma \big(\frac 1 2 +\nu \big)^2} {\Gamma (1+2\nu) y^{\frac 1 2+\nu}} {_2F_1}   \bigg(  \frac 1 2 + \nu, \frac 1 2 + \nu ; 1+ 2 \nu ; - \frac 1 y  \bigg)   \bigg\}. 
\end{align}
For instance, 
the Maass-form contribution reads: 
\begin{align*}
	2 \sum_{f  } \frac {L \big(\frac 1 2, f\big)^3} {L (1, \mathrm{Sym}^2 f)} \Lambda ( r_f; \varg);
\end{align*}
here the harmonic weight $ 2 /  L (1, \mathrm{Sym}^2 f)$ is from \cite[\S 2.2.2]{SH-Liu-Maass}.

The Motohashi formula was used in \cite{IM-4th-Moment,Motohashi-Riemann} to prove the currently best asymptotic formula for the fourth moment of the Riemann $\zeta (s)$: 
\begin{align*}
	\int_{0}^T \left| \zeta  \big(\tfrac 1 2 + it \big) \right|^{4} d      \shskip t = T P_{4} (\log T) + O \big(T^{  2/3} \log^8 T\big),
\end{align*}
where $P_4$ is an explicit polynomial of degree $4$. Note that $O (T^{2/3+\vepsilon})$ was obtained in \cite{Zavorotnyi-4th}. 

The Motohashi formula has been extended on  real quadratic fields and the Gaussian field $\BQ(i)$ by himself and Bruggeman \cite{B-Mo-Real-Quad,B-Mo}. However, in the complex setting, the $  \Xi$-function in \cite{B-Mo} is {\it not} expressed by the hypergeometric function---instead,  it is exhibited as the Hankel transform of Bessel functions over $\BC$ (see \cite[(15.5), (15.7)]{B-Mo}). Here, as our first objective, we would like to fill   the gap and make their formula entirely explicit.

This expression by Hankel transform is actually in tune with the ideas in their third joint work  \cite{B-Mo-New}, in which a new approach to the spectral formula of Motohashi is introduced without recourse to the Kloosterman-spectral formula of Kuznetsov \cite{Kuznetsov}. 
More explicitly, if we define 
\begin{equation}\label{0eq: Bessel, R}
	\begin{split}
	& 	j_{\nu} (x)  = \frac {\pi} {\sin (\pi  \nu) } \sqrt{x}  \big( J_{-2 \nu} (4 \pi \sqrt {x }) - J_{2 \nu} (4 \pi \sqrt {x }) \big), \\
	&j_{\nu} (-x )      = \frac {\pi} {\sin (\pi \nu) } \sqrt{x} \big( I_{-2  \nu} (4 \pi \sqrt {x }) - I_{2 \nu} (4 \pi \sqrt {x }) \big) 
	=   {4 \cos (\pi \nu)}   \sqrt{x} \shskip K_{2 \nu} (4 \pi \sqrt {x }), 
	\end{split}
\end{equation}
for $x \in \BR_+$, then, in the case $\Re (\nu) = 0$, 
we have 
\begin{align}\label{0eq: Hankel integral, R}
	\Xi ( y ; \nu) = \frac 1 2 \int_{\BR^{\times}} j_{\nu} (x/y)  j_{0} (-x ) \frac {d  x} { |x|^{3/2} }; 
\end{align}
see \cite[(2.33), (2.35)]{B-Mo-New}. The integral in \eqref{0eq: Hankel integral, R} may be viewed as the Hankel transform with kernel $j_0 (x)/ \sqrt{|x|}$ of the Bessel function $j_{\nu} (x) $. Note that this Hankel transform arises in the \Voronoi summation formula \cite{Voronoi} for the divisor function $\vtau (n)$ or the Eisenstein series $ \partial E(z; s) /\partial s |_{s= \frac 1 2}$.

The authors of   \cite{B-Mo-New} acknowledged the inspiration from the work of  Cogdell and Piatetski-Shapiro \cite{CPS}. It gives the first significant application of Bessel functions in representation theory---a direct approach  to the Kloosterman-spectral formula of Kuznetsov \cite{Kuznetsov}. A kernel formula is crucial  in \cite{CPS} to interpret the Weyl-element action on the Kirillov model as the Hankel integral transform. 

Finally, we remark that the reverse of variants of the Motohashi formula is very useful in the study of the cubic moment of $\GL_2$ or $\GL_2 \times \GL_1$ $L$-functions. This was initiated in the work of Conrey and Iwaniec \cite{CI-Cubic}, and has been implemented by many successors  \cite{Peng-Weight,Michel-Venkatesh-GL2,Petrow-Cubic,Young-Cubic,PY-Cubic,PY-Weyl2,PY-Weyl3,Frolenkov-Cubic,Nelson-Eisenstein,Wu-Motohashi,B-F-Wu-Weyl,Qi-Ivic}. 

\vskip 8pt

\subsubsection{The formula of Kuznetsov and Motohashi} 
In retrospect, this formula and its reverse were first discovered by Kuznetsov in \cite{Kuznetsov-Motohashi-formula} in 1983, and rigorously proven by Motohashi \cite{Motohashi-JNT-Mean,Motohashi-Divisor}  in 1992 and 1994. It transforms the spectral mean of $\vlambda_f (m) L \big(\frac 1 2, f\big)^2$, for $f$ Maass or Eisenstein,  to sums of the type of additive divisor problems.

Let $\Pi_{c} $ denote an orthonormal basis consisting of (even or odd) Hecke--Maass cusp forms for $ \SL_2 (\BZ)  $. 
For $f \in \Pi_{c} $, let $\vlambda_f (n)$ be its Hecke eigenvalues, and $\frac 1 4+r_f^2$ be its Laplace eigenvalue. For $n \in \BZ \smallsetminus \{0\}$ and $\nu \in \BC$, define 
\begin{align*}
\vtau_{\nu} (n) = \sum_{ ab = |n|} (a/b)^\nu . 
\end{align*}  
One may consider $ \vtau_{ir} (n)$   as the Fourier coefficients (Hecke eigenvalues) of the Eisenstein series $E \big(z; \frac 1 2 + i r \big)$. Moreover, let $\vtau (n) = \vtau_{0} (n) $. 

Let $m \in \BZ_+$.  Let $h \in \mathscr{H}$ 
be such that  
\begin{align*}
	h  (\pm  i/ 2  ) = 0 .
\end{align*}
and 
\begin{align*}
	h (r) \Lt \exp (-c |r|^2) , \qquad \text{($c > 0$)}, 
\end{align*}
in any fixed horizontal strip.  
Consider
\begin{align*}
	  \ScrL (m; h) = 2 \sum_{f \in \Pi_{c}} \frac { L \big(\frac 1 2, f\big)^2} {L(1, \mathrm{Sym}^2 f )} \vlambda_f (m) h (r_f) + \frac 1 {\pi} \int_{-\infty}^{\infty} \frac { \left|\zeta \big(\frac 12 + i r\big)\right|^4 } {|\zeta (1+2ir)|^2}  \vtau_{ir} (m) h (r) d      r. 
\end{align*}  
Then $ \ScrL (m; h) $ is equal to the sum of some diagonal terms, an Eisenstein residual term, and the  shifted divisor sum: 
\begin{align*}
\frac 2 { \pi^2} 	\sum_{n \neq 0, \shskip m} \frac {\vtau (n) \vtau (n-m)} {\sqrt{|n|}} \Phi (n/m; h) ,
\end{align*}
where $\Phi (\pm y; h)$ ($y \in \BR_+$) is defined by 
\begin{align}
	&\Phi (y; h) = \int_{(\beta)} \Gamma \lp \frac 1 2 - s \rp^2 \frac {\hat{h} (s)} {\cos (\pi s)} y^s d      s,  \\
	&\Phi (-y; h)  = \int_{(\beta)} \Gamma \lp  \frac 1 2 - s \rp^2 \tan (\pi s)  \hat{h} (s) y^s d      s,  
\end{align}
with $- 3 / 2 < \beta <   1 / 2$, and 
\begin{align}
	\hat{h} (s) = \frac 1 {2\pi} \int_{-\infty}^{\infty} \frac {\Gamma (s+i r)} {\Gamma (1-s+i r)} h(r) r \shskip d      r. 
\end{align}

	 The expression of $\ScrL (m; h)$ above with  $L(1, \mathrm{Sym}^2 f)$ is again from {\rm\cite[Lemma 6]{SH-Liu-Maass}}. 
	 The divisor sum is divided into {\it three} parts, in slightly different notation, in \cite[Lemma]{Motohashi-JNT-Mean} or {\rm\cite[Lemma 3.8]{Motohashi-Riemann}}. 

By contour shift, it is proven in \cite[\S 3.3]{Motohashi-Riemann} that 
\begin{equation}
\begin{split}
	\Phi (\pm y; h) =    \int_{-\infty}^{\infty}    \Theta^{\pm} (y; ir)  h(r ) r \tanh (\pi r) \shskip d      r & , 
\end{split}
\end{equation}
for $y > 1$, with 
\begin{align}\label{0eq: Theta+}
	\Theta^+ (y; \nu) =  \frac {i} {\sin (\pi \nu)} \frac{ \Gamma (\frac 1 2 + \nu)^2  } {  \Gamma (1+2\nu) y^{   \nu}    }  {_2F_1}    \bigg(  \frac {1} 2 + \nu , \frac {1} 2 + \nu ; 1+2\nu  ;  \frac 1 {y}  \bigg)   , 
\end{align}  
\begin{align}\label{0eq: Theta-}
\Theta^- (y; \nu) =   \frac{ \Gamma (\frac 1 2 + \nu)^2  } {  \Gamma (1+2\nu) y^{   \nu}    }  {_2F_1}    \bigg(  \frac {1} 2 + \nu , \frac {1} 2 + \nu ; 1+2\nu  ; - \frac 1 {y}  \bigg)  .
\end{align}

This explicit formula of Kuznetsov and Motohashi was applied to study the cubic and fourth moments of $ L \big(\frac 12 , f \big) $,  to establish a type of large sieve inequality for $\vlambda_ f(n)$ in \cite{Motohashi-JNT-Mean,Motohashi-Riemann,Ivic-t-aspect,Ivic-Moment-1,Jutila-4th-Moment}, and more recently, to prove a positive proportion of non-vanishing $ L \big(\frac 12 , f \big) $ in \cite{SH-Liu-Maass,BHS-Maass}.

One may see the clear  similarity between \eqref{0eq: Xi} and (\ref{0eq: Theta+}, \ref{0eq: Theta-}). Actually, in more general forms (similar to the right-hand sides of \eqref{0eq: Weber-Sch for J, Hankel}--\eqref{0eq: Weber-Sch for K, Hankel, 2}), both $ \Xi (y; \nu)$ and  $\Theta^{\pm} (y; \nu)$ were visible in   Kuznetsov's paper  and arose from the Hankel transform of Bessel functions over $\BR$ (see \cite[\S \S  3.1, 3.2, 3.4]{Kuznetsov-Motohashi-formula}).  However, Motohashi did not use this perspective as his proof is quite different. A crucial observation of Motohashi is the identity:
\begin{align*}
	\sum_{f  }   \frac { L \big(\frac 1 2, f\big)^2} {L(1, \mathrm{Sym}^2 f )} \vlambda_f (m) h (r_f) = \sum_{f  }  \epsilonup_f  \frac { L \big(\frac 1 2, f\big)^2} {L(1, \mathrm{Sym}^2 f )} \vlambda_f (m) h (r_f), 
\end{align*}
where $\epsilonup_f = \vlambda_f (-1) = \pm 1$ is the root number of $f$, as
$L \big(\frac 1 2, f\big) = 0$ trivially if $ \epsilonup_f = - 1 $, so that he may apply the Kuznetsov trace formula with $\epsilonup_f$ and   $K_{2ir} (x)$  (\cite[(2.5)]{Motohashi-JNT-Mean}), the Mellin--Barnes integral representation of $K_{2ir} (x)$,  and subsequently  the functional equation for the Estermann function (\cite[(2.21)]{Motohashi-JNT-Mean}).  

Recently in \cite{Qi-Liu-Moments}, by the {\it even} Kuznetsov trace formula, containing both $K_{2ir} (x)$ and $J_{2ir} (x)$ (see \cite[(3.17)]{CI-Cubic} or \cite[Lemma 5]{SH-Liu-Maass}), the Vorono\"i summation formula, and analysis for certain Fourier type integrals, the non-vanishing result in \cite{BHS-Maass} has been recovered, and, in a similar manner, extended to imaginary quadratic fields. 

The idea here is that, instead of an asymptotic formula, the Kuznetsov--Vorono\"i approach in \cite{Qi-Liu-Moments} should 
also yield the explicit formula, not only on $\BQ$ but also on $\BQ (i)$. After reading the original paper of Kuznetsov \cite{Kuznetsov-Motohashi-formula}  during the  writing process, 
the author realized that the `Kuznetsov--Vorono\"i approach' is  indeed the Kuznetsov approach!    
 However, in the technical aspect, neither Kuznetsov's version of Vorono\"i--Oppenheim summation formula   \cite[Theorem 2.3]{Kuznetsov-Motohashi-formula}  nor  Motohashi's complex analytic approach seems to work well  over $\BQ (i)$, as, for example, the Bessel function on $\BC$ is expressed as an {\it infinite} sum of  Mellin--Barnes integrals (see \cite[\S 3]{Qi-Bessel}). 

\subsection{Hankel--Mellin transform of Bessel functions}
Again, the following Bessel integral formulae are due to Weber and Schafheitlin\footnote{The Weber--Schafheitlin integral in \eqref{0eq: Weber-Sch for J, Hankel} is  {\it discontinuous} in the sense of different expressions the cases $y > 1$ and $y < 1$ and consequently a discontinuity in the formula  when $y = 1$. For simplicity, we only record here the formula  in the case $y > 1$.}: 
\begin{equation}\label{0eq: Weber-Sch for J, Hankel}
	\begin{split}
		\int_0^\infty  J_{\nu} (4 \pi x)  J_{\mu} (4 \pi x y)    x^{\shskip \rho - 1 }  d      x  & =   \frac {\Gamma  (  (\rho+\nu+\mu) /2 )   } {2 (2\pi)^{ \rho} \Gamma  (1 -   (\rho+\nu-\mu) /2)  \Gamma (1+\nu) y^{\shskip \rho+\nu } } \\
		& \qquad    \cdot  {_2F_1} \hskip -1pt \bigg( \frac {\rho + \nu +\mu} 2, \frac {\rho+\nu - \mu} 2  ; 1+ \nu ;  \frac 1 {y^2}  \bigg) ,
	\end{split}
\end{equation}
in the case $y > 1$, for $ - \Re(\nu+\mu)  <  \Re (\rho) < 2$ (see \cite[6.574]{G-R} or \cite[13.41 (3)--(5)]{Watson}), 
\begin{align}\label{0eq: Weber-Sch for KK, Hankel}
	\begin{aligned}
		& \int_0^\infty  K_{\nu} (4 \pi x)  K_{\mu} (4 \pi x y)    x^{\shskip \rho - 1 }  d      x    =  -  \frac {  \pi  } { 8 (2\pi)^{ \rho} \sin (\pi\nu)  } \\
		&	\cdot  \hskip -2pt  \bigg\{    \frac {    \Gamma  (  (\rho+\nu+\mu) /2 ) \Gamma  (  (\rho+\nu-\mu) /2 )} {       \Gamma (1+\nu) y^{\shskip \rho+\nu}  }  {_2F_1} \hskip -1pt \bigg( \frac {\rho + \nu +\mu} 2, \frac {\rho+\nu - \mu} 2  ; 1+ \nu ;    \frac 1 {y^2}  \bigg)   \\
		& \hskip -4pt - \frac {    \Gamma  (  (\rho-\nu+\mu) /2 )  \Gamma  (   (\rho-\nu-\mu) /2 )} {     \Gamma (1-\nu) y^{\shskip \rho-\nu}  }  {_2F_1} \hskip -1pt \bigg( \frac {\rho - \nu +\mu} 2, \frac {\rho-\nu - \mu} 2  ; 1- \nu ;   \frac 1 {y^2}  \bigg)   \bigg\}   ,
	\end{aligned}
\end{align}
in the case $y > 1$, for $   \Re (\rho) > |\Re (\nu)| + |\Re (\mu)| $ (see \cite[6.576 4]{G-R}),\footnote{Actually, it follows from  \cite[6.576 4]{G-R} that 
	\begin{align*}
		\int_0^{\infty}    K_{\nu}   (4\pi x)   K_{\mu} & (4\pi xy)  x^{\shskip \rho-1}  d      x  = \frac {1  } {8 (2\pi)^{\rho}\Gamma (\rho) y^{\shskip \rho+\nu} } \Gamma \bigg( \frac {\rho+\nu+\mu} 2 \bigg) \Gamma \bigg( \frac {\rho+\nu-\mu } 2 \bigg)   \\
	&	\cdot  \Gamma \bigg( \frac {\rho-\nu+\mu} 2 \bigg) \Gamma \bigg( \frac {\rho-\nu-\mu} 2 \bigg) {_2F_1} \bigg( \frac {\rho+\nu+\mu } 2, \frac {\rho+\nu-\mu } 2; \rho ; 1 - \frac 1 {y^2} \bigg) ,
\end{align*}
valid for $y > 0$ and  $\Re (\rho) > |\Re (\mu)| + |\Re (\nu)|$. To see \eqref{0eq: Weber-Sch for KK, Hankel},  in the case $y > 1$, we use  the transformation formula for ${_2F_1} (a,b;c; z)$ with respect to $z \ra 1 - z$ (see \cite[\S 2.4.1]{MO-Formulas}) and the reflection  formula for $\Gamma (s)$.
}
\begin{equation}\label{0eq: Weber-Sch for K, Hankel}
	\begin{split}
		\int_0^\infty  J_{\nu} (4 \pi x)  K_{\mu} (4 \pi x y)    x^{\shskip \rho - 1 }  d      x  & =   \frac {\Gamma  (  (\rho+\nu+\mu) /2 ) \Gamma  (  (\rho+\nu-\mu) /2 )    } {4 (2\pi)^{ \rho}    \Gamma (1+\nu) y^{\shskip \rho+\nu }  } \\
		& \quad    \cdot  {_2F_1} \hskip -1pt \bigg( \frac {\rho + \nu +\mu} 2, \frac {\rho+\nu - \mu} 2  ; 1+ \nu ;  -  \frac 1 {y^2}  \bigg) ,
	\end{split}
\end{equation}
for $\Re (\rho) > |\Re (\mu)| - \Re (\nu)$, and
\begin{equation}\label{0eq: Weber-Sch for K, Hankel, 2}
	\begin{split}
		& \int_0^\infty  K_{\nu} (4 \pi x)  J_{\mu} (4 \pi x y)    x^{\shskip \rho - 1 }  d      x    = - \frac {  \pi  } { 4 (2\pi)^{ \rho} \sin (\pi\nu)  } \\
	&	\cdot  \hskip -2pt  \bigg\{   \frac {    \Gamma  (  (\rho+\nu+\mu) /2 )} {     \Gamma  (1 -  (\rho+\nu-\mu) /2 ) \Gamma (1+\nu) y^{\shskip \rho+\nu}  }  {_2F_1} \hskip -1pt \bigg( \frac {\rho + \nu +\mu} 2, \frac {\rho+\nu - \mu} 2  ; 1+ \nu ;  -   \frac 1 {y^2}  \bigg)    \\
		 & \hskip -6pt -  \frac {    \Gamma (  (\rho-\nu+\mu) /2 )} {     \Gamma  (1 -   (\rho-\nu-\mu) /2 ) \Gamma (1-\nu) y^{\shskip \rho-\nu}  }  {_2F_1} \hskip -1pt \bigg( \frac {\rho - \nu +\mu} 2, \frac {\rho-\nu - \mu} 2  ; 1- \nu ;  -   \frac 1 {y^2}  \bigg)  \bigg\}   ,
	\end{split}
\end{equation}
for $\Re (\rho) > |\Re (\nu)| - \Re (\mu)$ (see \cite[6.576 3]{G-R} or \cite[13.45 (1)]{Watson}).\footnote{Note here that \eqref{0eq: Weber-Sch for K, Hankel, 2} is deducible from \eqref{0eq: Weber-Sch for K, Hankel}, and vice versa,  by the changes $y \ra 1/y$ and $\nu \leftrightarrow \mu$, the transformation formula for ${_2F_1} (a,b;c; z)$ with respect to $z \ra 1/z$ (see \eqref{2eq: transformation z - 1/z} or \cite[\S 2.4.1]{MO-Formulas}), along with the reflection  formula for $\Gamma (s)$. }

Now set  $\Re (\nu) = 0$,  $ \mu = 0 $, and $\rho = 1$. Then, in view of \eqref{0eq: Xi} and \eqref{0eq: Bessel, R}, one may deduce  \eqref{0eq: Hankel integral, R} directly from \eqref{0eq: Weber-Sch for K, Hankel} and \eqref{0eq: Weber-Sch for K, Hankel, 2}. See also \cite[(2.39)--(2.42)]{B-Mo-New}. Similarly, if one were to derive the  Kuznetsov--Motohashi formula  via the Kuznetsov  approach, then \eqref{0eq: Theta+} and \eqref{0eq: Theta-} would also follow from \eqref{0eq: Weber-Sch for J, Hankel}--\eqref{0eq: Weber-Sch for K, Hankel, 2}. See \cite[(46)--(48), (69)--(70)]{Kuznetsov-Motohashi-formula}.

Moreover, it is of interest to observe that \eqref{0eq: Weber-Sch for J} and \eqref{0eq: Weber-Sch for K} may be regarded as a special case of \eqref{0eq: Weber-Sch for J, Hankel} and \eqref{0eq: Weber-Sch for K, Hankel, 2} since 
\begin{align}\label{0eq: J1/2}
	J_{\frac 1 2} (z) = \lp \frac {2} {\pi z}\rp^{\frac 1 2} \sin z, \qquad J_{-\frac 1 2} (z) = \lp \frac {2 } {\pi z}\rp^{\frac 1 2} \cos z,  
\end{align}
as in \cite[3.4 (3), (6)]{Watson}.

\subsection{Double Fourier--Mellin transform of Bessel functions}\label{sec: double integral} Next, the following double Bessel integral formulae seem to be new and will be proven in Appendix \ref{sec: Fourier-Mellin, R}. 

For real $y,  \varw  $ with  $|y \varw| > 1$,  we have  
\begin{equation}\label{1eq: double Fourier, 1}
	\begin{split}
		\int_0^{\infty} \hskip -2pt	\int_0^{\infty} J_{\nu} (4\pi & \hskip -1pt \sqrt{\varv x})   e ( - x y  -  \varv \varw) x^{ \shskip  \beta  -1} \varv^{     \gamma - 1} d      x d      \varv  =   \\
	  &	 \frac { \Gamma (\beta + \nu/2) \Gamma (\gamma+\nu/2)} {(2\pi)^{\beta+\gamma} \Gamma (1+\nu) (iy)^{\beta + \frac 12 \nu} (i \varw)^{\gamma +   \frac 12 \nu} }	\displaystyle  {_2F_1} \bigg(  \beta + \frac \nu 2, \gamma + \frac {\nu} 2 ; 1+\nu; \frac 1 {y \varw} \bigg),
	\end{split}
\end{equation}
if  
\begin{align}\label{1eq: condition for J}
	-   \Re ( \nu/2) <   \Re (  \beta) < \frac 5 4, \quad  -   \Re ( \nu/2) <   \Re (  \gamma) <  \min \big\{ 1+ \Re (  \beta), 2- \Re (  \beta) \big\}  , 
\end{align}
and 
\begin{equation}\label{1eq: double Fourier, 2}
\begin{split}
	\int_0^{\infty} \hskip -2pt	\int_0^{\infty} K_{\nu} (4\pi & \hskip -1pt \sqrt{\varv x})   e ( - x y  -  \varv \varw) x^{  \shskip \beta  -1} \varv^{     \gamma - 1} d      x d      \varv   = - \frac {\pi} {2 (2\pi)^{\beta + \gamma } \sin (\pi \nu)} \\
	\cdot \bigg\{	&  \frac { \Gamma (\beta + \nu/2) \Gamma (\gamma+\nu/2) } { \Gamma (1+\nu) (iy)^{\beta + \frac 12 \nu} (i \varw)^{\gamma +   \frac 12 \nu} } {_2F_1} \bigg(\beta + \frac \nu 2, \gamma + \frac {\nu} 2 ; 1+\nu ; - \frac 1 {y \varw} \bigg) \\
	- & \frac { \Gamma (\beta + \nu/2) \Gamma (\gamma+\nu/2)    } { \Gamma (1-\nu) (iy)^{\beta - \frac 12 \nu} (i \varw)^{\gamma -   \frac 12 \nu} } {_2F_1} \bigg( \beta - \frac \nu 2, \gamma - \frac {\nu} 2; 1-\nu; - \frac 1  {y \varw}\bigg) \bigg\} ,
\end{split}
\end{equation}
if 
\begin{align}\label{1eq: condition for K}
 |  \Re (\nu/2) |  < 	 \Re ( \beta) , \quad  |  \Re (\nu/2) | < \Re  ( \gamma)  < \Re (\beta) + 1,
\end{align} 
where  the left sides must be regarded as iterated integrals. Note that these double integrals are {\it not} absolutely convergent, so it is probably most suitable to  interpret them in the sense of distributions. 

On setting $\beta  = (\rho - \mu)/2$, $\gamma = (\rho+\mu)/2$, and $y = \varw  $, the hypergeometric functions on the right of \eqref{1eq: double Fourier, 1} and \eqref{1eq: double Fourier, 2} are the same as those of  \eqref{0eq: Weber-Sch for J, Hankel}--\eqref{0eq: Weber-Sch for K, Hankel, 2}.  This is {\it not} a coincidence! The integrals in \eqref{1eq: double Fourier, 1} and \eqref{1eq: double Fourier, 2} were derived by the Bessel distributions  from a certain Hermitian form for unitary irreducible representations  of $\GL_2 (\BR)$ arising in the study of H. Wu  \cite{Wu-Motohashi}  of the Motohashi formula via the Godement--Jacquet theory  \cite{Godement-Jacquet}.   The author feels that there is also an approach to the Motohashi formula via Bessel distributions  \cite{Baruch-GL2,BaruchMao-NA,BaruchMao-Real,Chai-Qi-Bessel}. This is left to future investigations, for which the Bessel integral formulae \eqref{1eq: double Fourier, 1} and \eqref{1eq: double Fourier, 2} will  probably be very helpful.

\subsection{Purpose of the paper}\label{sec: purpose}

This paper is divided into  the analytic part and the arithmetic part. 

In the first part, we would like to 
establish  connections  over $\BC$ between the Bessel functions and hypergeometric functions via the Hankel--Mellin transform and the double Fourier--Mellin transform---more explicitly, to prove the analogues  over $\BC$ of the Bessel integral formulae \eqref{0eq: Weber-Sch for J, Hankel}--\eqref{0eq: Weber-Sch for K, Hankel, 2} and \eqref{1eq: double Fourier, 1}--\eqref{1eq: condition for K}.  

In the second part, we would like to 
establish the explicit spectral formulae over $\BQ (i)$ of Bruggeman--Motohashi and Kuznetsov--Motohashi, as have been discussed over $\BQ$ in \S \ref{sec: spectral}, by application of the  Bessel integral formula of Hankel--Mellin type. For the latter, we shall mainly use the original ideas of Kuznetsov  \cite{Kuznetsov-Motohashi-formula}, although it is indicated by Motohashi \cite{Motohashi-JNT-Mean,Motohashi-Riemann} that his proof is not quite rigorous.   Our new input is the regularization method introduced in \cite{Qi-BE}.   Of course, our argument may be used to rectify the proof of Kuznetsov as it also works over $\BQ$ (see \S \ref{sec: remarks Kuznetsov}).  

Furthermore, as suggested in \S \ref{sec: double integral}, the Bessel integrals of double Fourier type (over $\BR$ or $\BC$) might also be helpful for the study of the Motohashi formula from the perspective of Bessel distributions.

\begin{acknowledgement}
	The author thanks Professor J. W. Cogdell for  his patience, encouragement, and a great deal of mathematics learned from him at OSU. 
\end{acknowledgement}

{\large \part*{I. Analytic Part}  }

\section{Main Results I: Complex Bessel Integral Formulae} \label{sec: main, analytic}

In this section, we provide the basic definitions (the notation over $\BC$ is unavoidably more complicated) and   state our main analytic results. As the theorems in \S \ref{sec: double Fourier}  will not be applied later in Part \hyperref[part: arithm]{II}, after   \S \ref{sec: Hankel-Mellin} the reader may safely jump to \S \ref{sec: prelim} or even Part \hyperref[part: arithm]{II}  at a first reading.

\subsection{Bessel functions  over \protect \scalebox{1.06}{$\BC$}}\label{sec: Bessel C}
We start by introducing the definition of Bessel functions over $\BC$ (see \cite[\S 15.3]{Qi-Bessel}, \cite[(6.21), (7.21)]{B-Mo} or \cite[(4.58), (9.26)]{B-Mo2}).

For $\nu \in \BC$, $  p \in \BZ  $, and $z \in \BC \smallsetminus \{0\}$,  define\footnotemark
\begin{equation}\label{0def: J mu m (z)}
	J_{\nu ,\shskip  p} (z) = J_{\nu + p }    (z)   J_{\nu -  p  }    ({\widebar z} ) ,
\end{equation}
and
\begin{equation}\label{0eq: defn of Bessel}
	\boldJ_{ \nu,\shskip  p} (z)  =   \frac {1} {\sin (\pi \nu)} \lp J_{-\nu,\shskip  -p} (4 \pi   z) - J_{\nu,\shskip  p} (4 \pi   z)  \rp.    
\end{equation}
The Bessel functions $J_{\nu ,\shskip  p} (z)$ and  $\boldJ_{ \nu,\shskip  p} (z)$ are even and real analytic in $z$, and (complex) analytic in $\nu$.  It is understood that $\boldJ_{ \nu,\shskip  p} (z)$ is defined by the limit in the non-generic case when  $ \nu $ is integral,  as $J_{-n} (z) = (-1)^n J_{n} (z) $ ($n =  1, 2, 3, ...$). 
 
\footnotetext{The Bessel functions for $\GL_2 (\BC)$ or $\SL_2 (\BC)$ are derived by different means in \cite{Qi-Bessel} and \cite{B-Mo,B-Mo2}.  
	It should be remarked that it is unnecessary and inconvenient to  modify   $$J_{\nu} (z) = \sum_{n=0}^{\infty} \frac {(-1)^n (z/2)^{\nu +2n}} {n! \Gamma (\nu + n + 1)} $$ by the factor $(z/2)^{-\nu}$ (to make it analytic on $\BC$) as in \cite[(6.21)]{B-Mo}. This would destroy the Bessel equation!  }


It follows from \eqref{0eq: J1/2} that
\begin{align}\label{1eq: J1/2 (z)}
	 \boldJ_{\frac 1 2, 0} (z) = \frac 1 {2\pi^2 |z|} \cos (8\pi \mathrm{Re} (z)). 
\end{align} 

Moreover, if $p$ is a half integer ($2p$ odd), the (odd) Bessel function is defined by 
\begin{align}
 \boldJ_{\nu ,\shskip  p} (z) =	 \frac {i} {\cos (\pi \nu)} \lp J_{-\nu,\shskip   - p} (4 \pi   z) + J_{\nu,\shskip  p} (4 \pi   z)  \rp. 
\end{align}
However, the spectral formulae  on $\BQ (i)$  concern only even Bessel functions. Actually, we need $p$  even (not just $2p$ even) for the root number to be $+1$. 



\subsection{Hypergeometric functions over \protect \scalebox{1.08}{$\BC$}} \label{sec: hyper}
The definition of hypergeometric functions over $\BC$ is  more subtle and complicated.    

Let $  \nu, \mu, \rho \in \BC$ and $ d,   p \in \BZ$.  
Let $z \in \BC \smallsetminus \{0\}$.
\delete{For 
\begin{align}
a = 	\frac {\rho + \nu +\mu} 2 , \qquad b = \frac {\rho + \nu -\mu} 2, \qquad c = 1+\nu, 
\end{align}
define 
\begin{equation} 
	{_{\phantom{\nu}}^{\rho} \hskip -1pt \boldF_{\nu}^{\mu}} (z)   = 
	\frac{\Gamma (a) \Gamma (b)} {\Gamma (c)} \cdot z^{\frac 1 2 (a+b)}	{_2F_1}  ( a, b  ; c ;  z    ) ,
\end{equation}
or, more explicitly,} 

First, we introduce   
\begin{equation} \label{1eq: hypergoem, C, 1.0}
	{_{\phantom{\nu}}^{\rho} \hskip -1.5pt f_{\nu}^{\mu}} (z) =   z^{\frac 1 2  \nu }	{_2F_1} \hskip -1pt \bigg(   \frac {\rho + \nu +\mu} 2, \frac {\rho+\nu - \mu} 2  ; 1+ \nu ;  z    \bigg) , 
\end{equation} 
and
\begin{equation} \label{1eq: hypergoem, C, 1}
	{_{\phantom{\nu}}^{\rho} \hskip -1pt  {F}_{\nu}^{\mu}} (z) \hskip -0.5pt = \hskip -0.5pt	\frac {\Gamma  (   ({\rho + \nu +\mu}) / 2    ) \Gamma  (     {(\rho+\nu - \mu)}/ 2     )} {\Gamma (1+\nu)} \cdot  {_{\phantom{\nu}}^{\rho} \hskip -1.5pt f_{\nu}^{\mu}} (z)  .
	\end{equation} 
Note that ${_{\phantom{\nu}}^{\rho} \hskip -1pt  {F}_{\nu}^{- \mu}} (z) = {_{\phantom{\nu}}^{\rho} \hskip -1pt  {F}_{\nu}^{\mu}} (z)$. 
Define
  \begin{equation}\label{1eq: hypergoem, C, 2}
  	{_{\phantom{\nu}}^{\rho} \hskip -1.5pt {F}_{\nu, \shskip p}^{\mu, \shskip  d}} (z)  =  \big(   \cos (\pi(\rho+\nu))   -   	(-1)^{d+p} \cos (\pi\mu) \big) \cdot {_{\phantom{\nu}}^{\rho}    {F}_{\nu+p}^{\mu+d}} (z) \, {_{\phantom{\nu}}^{\rho}    {F}_{\nu-p}^{\mu-d}} (\widebar{z}),  
  \end{equation}
and
\begin{equation}\label{1eq: hypergoemetric}
	{_{\phantom{\nu}}^{\rho} \hskip -1pt \boldF_{\nu, \shskip p}^{\mu, \shskip  d}} (z) = 
	 	\frac 1 {\sin (\pi \nu)} \big(  {_{\phantom{\nu}}^{\rho} \hskip -1.5pt {F}_{\nu, \shskip p}^{\mu, \shskip  d}} (z)  - {_{\phantom{\nu}}^{\rho} \hskip -1.5pt {F}_{-\nu, \shskip -p}^{\mu, \shskip  d}}  (z)  \big). 
\end{equation} 
Let ${_{\phantom{\nu}}^{\rho} \hskip -1.5pt {C}_{\nu, \shskip p}^{\mu, \shskip  d}}$ be defined such that 
\begin{align}\label{1eq: hypergeometric, 3}
	{_{\phantom{\nu}}^{\rho} \hskip -1pt \boldF_{\nu, \shskip p}^{\mu, \shskip  d}} (z) = {_{\phantom{\nu}}^{\rho} \hskip -1.5pt {C}_{\nu, \shskip p}^{\mu, \shskip  d}} \cdot {_{\phantom{\nu}}^{\rho} \hskip -2pt f_{\nu+p}^{\mu+d}} (z) \, {_{\phantom{\nu}}^{\rho} \hskip -2pt f_{\nu-p}^{\mu-d}} (\widebar z)  + {_{\phantom{\nu}}^{\rho} \hskip -1.5pt {C}_{-\nu, \shskip -p}^{\mu, \shskip  d}} \cdot {_{\phantom{\nu}}^{\rho} \hskip -2pt f_{-\nu-p}^{\mu+d}} (z) \, {_{\phantom{\nu}}^{\rho} \hskip -2pt f_{-\nu+p}^{\mu-d}} (\widebar z) .
\end{align}

 The function ${_{\phantom{\nu}}^{\rho} \hskip -1.5pt {F}_{\nu, \shskip p}^{\mu, \shskip  d}} (z)$ is real analytic in $z$ for $|\arg (1-z)| < \pi$ and (complex) analytic in $\nu$, $\mu$, and $\rho$, although there could be (simple) poles when $ \rho +\nu \pm \mu $ is a non-positive integer (to see this, examine the poles of the gamma factors in \eqref{1eq: hypergoem, C, 1} and the zeros of the trigonometric factor in \eqref{1eq: hypergoem, C, 2}). The function ${_{\phantom{\nu}}^{\rho} \hskip -1pt \boldF_{\nu, \shskip p}^{\mu, \shskip  d}} (z)$  is real analytic for all $z \neq 0$ except for $z = 1$. This is due to the transformation formula for ${_2F_1} (a,b;c; z)$ with respect to $z \ra 1/z$ (see \eqref{2eq: transformation z - 1/z} or Lemma \ref{lem: reciprocity}).  Moreover, when $ \mathrm{Re} (\rho) < 1 $, one may define  ${_{\phantom{\nu}}^{\rho} \hskip -1pt \boldF_{\nu, \shskip p}^{\mu, \shskip  d}} (1) $   by continuity (the Gauss   formula). 
Again, the definition of ${_{\phantom{\nu}}^{\rho} \hskip -1pt \boldF_{\nu, \shskip p}^{\mu, \shskip  d}} (z)$ must be in the limit form for integral $ \nu  $. To see this, use the limit formula (see \cite[\S 2.1]{MO-Formulas}):
\begin{align}
	\lim_{c \ra 1- n} \frac 1 {\Gamma (c)} {_2F_1} (a, b; c; z) = \frac {(a)_{n} (b)_{n}} {n !} \cdot z^{n} {_2F_1} (a+n, b+n; n+1; z), 
\end{align}
for $n  =  1, 2, 3, ... 
$, where $(a)_n = \Gamma (a+n) / \Gamma (a)$.

\subsection*{Convention} Subsequently, for brevity, the $0$ in the superscript  or subscript  will be suppressed. For example, we shall let $ \boldJ_{ \nu } (z) = \boldJ_{ \nu,\shskip  0} (z)  $, ${_{\phantom{\nu}}^{\rho} \hskip -1pt \boldF_{\nu, \shskip p}^{\mu }} (z) = {_{\phantom{\nu}}^{\rho} \hskip -1pt \boldF_{\nu, \shskip p}^{\mu, \shskip  0}} (z)$, and ${_{\phantom{\nu}}^{\rho} \hskip -1pt \boldF_{\nu}^{\mu }} (z) = {_{\phantom{\nu}}^{\rho} \hskip -1pt \boldF_{\nu, \shskip 0}^{\mu, \shskip  0}} (z)$.

\subsection{Hankel--Mellin transform of Bessel functions over \protect \scalebox{1.08}{$\BC$}}  \label{sec: Hankel-Mellin}

Our first theorem (in particular \eqref{1eq: main, 2}) is the complex analogue of the Weber--Schafheitlin integral   formulae in \eqref{0eq: Weber-Sch for J, Hankel}--\eqref{0eq: Weber-Sch for K, Hankel, 2}.

\begin{thm}\label{thm: Hankel}
	 Let $  \nu, \mu, \rho \in \BC$ and $ d,   p \in \BZ$.   Let $u \in \BC \smallsetminus \{0 \}$. 
	 We have\footnote{As we shall do calculus using  $\partial / \partial z$ and $ \partial / \partial \widebar{z} $,  it will be more natural to use the volume (differential) form $ i d      z \vwedge d      \widebar{z} $ for the integration on $\BC^{\times}$---the reader may as well consider it as twice the standard Lebesgue measure since $  i d      z \vwedge d      \widebar{z}  = 2 d x \vwedge d y$ for $z = x+iy$. }
	 \begin{align}\label{1eq: main, 1}
	 	\viint_{\BC^{\times}} \boldJ_{ \nu,\shskip  p}  (uz)  \boldJ_{ \mu,\shskip  d} (z)     |z|^{2 \rho   - 2}     i d      z \vwedge d      \widebar{z} =   \frac {2} {(2\pi)^{2\rho+2}} {_{\phantom{\nu}}^{\rho} \hskip -1pt \boldF_{\nu, \shskip p}^{\mu, \shskip  d}} (u^2) ,  
	 \end{align}
 or equivalently
    \begin{align}\label{1eq: main, 2}
    	\viint_{\BC^{\times}}  \boldJ_{ \nu,\shskip  p} (z)  \boldJ_{ \mu,\shskip  d}  (uz)   |z|^{2 \rho   - 2}     i d      z \vwedge d      \widebar{z} =   \frac {2} {(2\pi)^{2\rho+2} |u|^{2\rho}} {_{\phantom{\nu}}^{\rho} \hskip -1pt \boldF_{\nu, \shskip p}^{\mu, \shskip  d}} \lp \frac 1 {u^2} \rp ,  
    \end{align}
 for $  |\Re (\nu)| + |\Re (\mu)| < \Re (\rho ) < 1$. Moreover, the validity of the formulae may be extended to $|\Re (\nu) | + |\Re (\mu)| <   \Re (\rho) <   3 / 2$ if $|u| \neq 1$. 
\end{thm}

The equivalence between \eqref{1eq: main, 1} and \eqref{1eq: main, 2} is clear by the changes of variables $z \ra  z/ u$ and $u \ra 1/u$. Alternatively, it is readily seen from the reciprocity formula in Lemma \ref{lem: reciprocity}. 

Note that  \eqref{0eq: general W-S, C} is indeed a special case of \eqref{1eq: main, 2} in view of \eqref{1eq: J1/2 (z)} and \eqref{1eq: special hypergeometric}.

For later applications, it is easier to rephrase Theorem \ref{thm: Hankel}  in the following form. 


\begin{cor}\label{cor: Hankel}
	Let $  \nu, \mu, \rho \in \BC$ and $ d,   p \in \BZ$.   Let $\varv, \varw \in \BC \smallsetminus \{0 \}$. We have
	 \begin{equation}\label{1eq: main, 3}
	 	\viint_{\BC^{\times}} \boldJ_{ \nu,\shskip  p}  (\hskip -1pt \sqrt{\varv z}) \boldJ_{ \mu,\shskip  d} (\hskip -1pt \sqrt{\varw z})     |z|^{  \rho   - 2}     i d      z \vwedge d      \widebar{z} =   \frac {  1 } {\pi^2 (2\pi)^{2\rho} | \varw|^{\rho} } {_{\phantom{\nu}}^{\rho} \hskip -1pt \boldF_{\nu, \shskip p}^{\mu, \shskip  d}} \lp \frac {\varv} {\varw} \rp ,  
	 \end{equation}
  for $  |\Re (\nu)| + |\Re (\mu)| < \Re (\rho ) < 1$, and for $|\Re (\nu) | + |\Re (\mu)| <   \Re (\rho) <  3 / 2$ if $|\varv| \neq |\varw|$. 
\end{cor}

By the Gauss formula for $ {_2F_1} (a, b; c; 1) $  (see \cite[\S 2.1]{MO-Formulas}):
\begin{align*}
	{_2F_1} (a, b; c; 1) = \frac {\Gamma (c) \Gamma (c-a-b)} {\Gamma (c-a) \Gamma (c-b)}, \hskip 15 pt \Re (a+b-c ) < 0, \, c \neq 0, -1, -2,...,
\end{align*}
 it is easy to express the integral 
\begin{align*}
	\viint_{\BC^{\times}} \boldJ_{ \nu,\shskip  p}  (\hskip -1pt \sqrt{\varv z})  \boldJ_{ \mu,\shskip  d} (\hskip -1pt \sqrt{\varv z})     |z|^{  \rho   - 2}     i d      z \vwedge d      \widebar{z} 
\end{align*}
in terms of  $\Gamma (s)$,  $\cos (\pi s)$, or $ \sin (\pi s) $, but the formula is quite long. For later use,  we only record the formula here in the spherical case. 

\begin{cor}\label{cor: Gauss} For  $  |\Re (\nu)| + |\Re (\mu)| < \Re (\rho ) < 1$, we have
	\begin{equation}
	\begin{split}
			\viint_{\BC^{\times}}  \boldJ_{ \nu } (& \hskip -1pt \sqrt{\varv z})   \boldJ_{ \mu }  (\hskip -1pt \sqrt{\varv z})   |z|^{  \rho   - 2}     i d      z \vwedge d      \widebar{z} = \frac {  \sin (\pi \rho) \Gamma (1-\rho)^2 } {2 \pi^6 (2\pi)^{2\rho} | \varv|^{\rho} }    \\
	& \cdot \big(\cos (\pi(\rho+\nu)) -	  \cos (\pi\mu) \big) \big(\cos (\pi(\rho-\nu)) -	  \cos (\pi\mu) \big)   \\
	& \cdot \Gamma \lp \frac {\rho+\nu+\mu}  2 \rp^2 \Gamma \lp \frac {\rho+\nu-\mu}  2 \rp^2 \Gamma \lp \frac {\rho-\nu+\mu}  2 \rp^2 \Gamma \lp \frac {\rho-\nu-\mu}  2 \rp^2    . 
	\end{split}
	\end{equation} 
\end{cor}

\subsection{Revisiting the hypergeometric functions over \protect \scalebox{1.08}{$\BC$}}\label{sec: hyper, 2}

Let us also introduce the hypergeometric functions that arose in the author's previous work \cite{Qi-BE}. 

Define 
\begin{align}\label{1eq: defn of F, 21}
	F^{\rho}_{ \nu } (z) = \frac {      \Gamma (\rho + \nu)     } {   \Gamma  (1 +  \nu )  } \cdot  z^{\frac 1 2 (\rho + \nu)} {_2 F_1}  \bigg(   \frac {\rho  +   \nu } {2}, \frac {1   +   \rho +   \nu  } {2}; 1   +  \nu  \shskip; z   \bigg) .
\end{align}
Further, we define
\begin{equation}\label{1eq: sim hyper, 1}
	F^{\rho}_{\nu, \shskip p} (z) =	  2^{-2\nu} \sin (\pi (\rho+\nu)) F^{\rho}_{\nu+p} (z) F^{\rho}_{\nu-p} (\widebar{z}),  
\end{equation}
and
\begin{equation}\label{1eq: sim hyper, 2}
	\boldF^{\rho}_{\nu, \shskip p} (z) = 	\frac 1 {\sin (\pi \nu)} \big( {F}^{\rho}_{-\nu, \shskip - p} (z) -  {F}^{\rho}_{\nu, \shskip p} (z)   \big).
\end{equation}

Note that  
\begin{align*}
	{  \sqrt{\pi}  }	F^{\rho}_{ \nu } (z) =     2^{\rho + \nu - 1} \shskip   z^{\frac 1 2 \rho } \cdot {_{\phantom{\nu}}^{ \rho + \frac 1 2} \hskip -1pt  {F}_{\nu}^{\pm \frac 1 2}} (z) ,
\end{align*} 
by the duplication formula for $\Gamma (s)$. It follows that 
\begin{equation}\label{1eq: special hypergeometric}
	\pi   \boldF^{\rho}_{\nu, \shskip p} (z) =   {  2^{2\rho - 2}  }   |z|^{\rho } \cdot    {_{\phantom{\nu}}^{  \rho + \frac 1 2} \hskip -1pt \boldF_{\nu, \shskip p}^{\frac 1 2 , \shskip  0}} (z). 
\end{equation}  

The reader may find a type of hypergeometric functions in \eqref{2eq: defn of F, 21}--\eqref{3eq: F, 32} which is more general than  $  \boldF^{\rho}_{\nu, \shskip p} (z) $ but still regarded as a special case of  ${_{\phantom{\nu}}^{\rho} \hskip -1pt \boldF_{\nu, \shskip p}^{\mu, \shskip  d}} (z)$.

\subsection{Fourier transform of Bessel functions over \protect \scalebox{1.08}{$\BC$}}  \label{sec: double Fourier}

As the complex analogues of \eqref{0eq: Weber's formula}--\eqref{0eq: Weber-Sch for K}, the author proved in \cite{Qi-Sph,Qi-II-G,Qi-BE} the following Bessel integral formulae\footnote{The formulae in the author's previous papers are expressed   in the polar coordinates, but here it will be simpler to work in the Cartesian coordinates.  
}:
\begin{equation}\label{0eq: Fourier, C}
	\begin{split}
		& \viint_{\BC^{\times}}   \boldJ_{ \nu,\shskip  p}   (\hskip -1.5pt \sqrt{z})  e (- 2 \mathrm{Re} (uz) )    \frac {i d      z \vwedge d      \widebar{z}} {|z|}
		=  \frac 1 { 2|u|} e \bigg(\hskip -1pt \Re \bigg(\frac 1 {u} \bigg) \hskip -1pt \bigg) \boldJ_{ \frac 1 2 \nu,  \frac 1 2 p} \bigg(\frac 1 {4 u}\bigg), 
	\end{split}
\end{equation}
for $ |\Re (\nu)| < 1$, and 
\begin{equation}\label{0eq: general W-S, C}
	\begin{split}
		\viint_{\BC^{\times}}   \boldJ_{ \nu,\shskip  p}   (z)  & e (- 2 \mathrm{Re} (uz) )   |z|^{2 \rho - 2}  i d      z \vwedge d      \widebar{z}
		=  \frac 2 {(4\pi)^{2 \rho}}  \boldF^{\rho}_{\nu, \shskip p} \bigg( \hskip -1pt  \frac 4  {u^2} \hskip -1pt  \bigg)  ,
	\end{split}
\end{equation} 
for $  |\Re (\nu)| < \Re (\rho ) <  1 / 2$. Moreover, \eqref{0eq: general W-S, C} is valid 
for $|u| > 2$ if we only assume $  |\Re (\nu)| < \Re (\rho ) <  1 $.


\subsection{Double Fourier--Mellin transform of Bessel functions  over \protect \scalebox{1.08}{$\BC$}}

Our second theorem  (in particular, the second identity in \eqref{1eq: double Fourier, 1C}) is the complex analogue of the   integral formulae in \eqref{1eq: double Fourier, 1}--\eqref{1eq: condition for K} in the special case $\beta =  1 / 2$. It may be easily deduced by a composition of  \eqref{0eq: Fourier, C} and \eqref{0eq: general W-S, C} (hence the quadruple integral is considered  iterated), along with Kummer's quadratic transformation law (see Lemma \ref{lem: quad trans} ). 

\begin{thm}\label{thm: double Fourier} Let $p$ be even. Then for  $  |\Re (\nu)|, \Re ( \gamma ) < 1$ with $ |\Re (\nu)| <   \Re (2 \gamma ) $, when $ \Re (1/  u \varw) < 1$, we have 
	\begin{equation}\label{1eq: double Fourier, 1C}
		\begin{split}
			\viint \hskip -2pt \viint   \boldJ_{ \nu,\shskip  p}   (\hskip -1pt \sqrt{\varv z})    e ( - 2   \mathrm{Re}  ( uz + \varw \varv) )      \frac {i d      z \vwedge d      \widebar{z}} {|z|} &    \frac {i d      \varv  \vwedge d      \widebar{\varv }} { |\varv |^{2 -2\gamma} } = \\
		&	\left\{  \begin{aligned}
			& \displaystyle 	\frac {{|u |^{2\gamma - 1}} } {\pi^{2\gamma} } \boldF^{\gamma}_{\frac 1 2 \nu, \shskip \frac 1 2 p} \bigg( \hskip -1pt  \frac 1  {(1 - 2    u \varw)^2} \hskip -1pt  \bigg),  \\
			& \displaystyle \frac { 1 } {\pi (2\pi)^{   2\gamma  } |u|  |\varw|^{2\gamma} }           {_{\phantom{\nu}}^{\frac 1 2 + \gamma } \hskip -1.5pt \boldF_{ \nu, p}^{\frac 1 2 - \gamma   }} \bigg( \hskip -1pt \frac 1 {   u \varw} \hskip -1pt  \bigg),  
			\end{aligned} \right.
		\end{split}
	\end{equation} 
and the first identity is valid as long as $u \varw \neq 0$ for  $  |\Re (\nu)| <   \Re (2 \gamma ) < 1$. 
\end{thm}

By examining the arguments over $\BR$ in Appendix \ref{sec: Fourier-Mellin, R}, we formulate the following conjecture over $\BC$.

\begin{conj}\label{conj: double Fourier}
	Let $ \Re (1/  u\varw) < 1$. 	The integral formula  {\rm(}here $p$ is not necessarily even{\rm)}
	\begin{equation}\label{1eq: double Fourier, 2C} 
		\begin{split}
			\viint \hskip -2pt \viint   \boldJ_{ \nu,\shskip  p}   (\hskip -1pt \sqrt{\varv z})    e ( - 2  \mathrm{Re}  ( uz + \varw \varv) )   &   \frac {i d      z \vwedge d      \widebar{z}} {|z|^{2-2\beta} }     \frac {i d      \varv  \vwedge d      \widebar{\varv }} { |\varv |^{2 -2\gamma} } \\
			& \qquad  =   \frac { 2 } {(2\pi)^{2 \beta + 2 \gamma  } |u|^{2\beta } |\varw|^{2 \gamma}  }           {_{\phantom{\nu}}^{ \beta + \gamma  } \hskip -1.5pt \boldF_{ \nu, p}^{\beta - \gamma }} \bigg( \hskip -1pt \frac 1 {  u \varw} \hskip -1pt  \bigg) 
		\end{split}
	\end{equation} 
	is valid for 
	\begin{align}
		|\Re (\nu) | < \Re (2\beta) < \frac 3 2, \quad  	|\Re (\nu) | < \Re (2\gamma) < \min \big\{ 1 + \Re (2\beta), 3 - \Re (2\beta)  \big\}. 
	\end{align} 
\end{conj}

To prove this conjecture, one would need a theory of Kummer's confluent hypergeometric function over $\BC$. Moreover, it seems difficult to generalize the formula \eqref{0eq: Fourier, C}---the first step for Theorem \ref{thm: double Fourier}---as its proof in \cite{Qi-II-G} relies on the radial exponential formula 
\begin{equation*}
	\viint  \boldJ_{ \nu,\shskip  p}   (\hskip -1.5pt \sqrt{z})  \exp (- 2 \pi c |z| )    \frac {i d      z \vwedge d      \widebar{z}} {|z|}
	= \frac {4 i^{2p}} {\pi c} K_{\nu} \lp \frac {4\pi} {c} \rp I_{p} \lp \frac {4\pi} {c} \rp , \quad \text{($\Re (c) > 0$)}, 
\end{equation*}
and the combinatorial arguments therein fail  if the $1/|z|$ in the integral were replaced by $1 / |z|^{2-2\rho}$.


A heuristic way of circumventing any knowledge of Kummer's confluent hypergeometric function over $\BC$ 
is to reformulate the double Fourier--Mellin type integral by the {formal} integral representation:  \begin{equation}\label{0eq: formal integral}
	 {2\pi^2} \boldJ_{ \mu }  (z)  =   \viint e (2 \Re (z \varv + z / \varv)) \shskip |\varv|^{2\mu - 2} i d      \varv  \vwedge d      \widebar{\varv };  
\end{equation}
the integral above is {\it formal} as it is {\it not} absolutely convergent (see \cite[\S 6]{Qi-Bessel}). Thus,   unsurprisingly,  \eqref{1eq: double Fourier, 2C} may be deduced in a {\it formal} and {\it unrigorous} manner from \eqref{1eq: main, 2} or \eqref{1eq: main, 3}.  Nevertheless, by this idea, we may prove a weak version of Conjecture {\rm \ref{conj: double Fourier}}. 

\begin{thm}\label{thm: double Fourier, 2}
The integral formula {\rm\eqref{1eq: double Fourier, 2C}}, for $u \varw \neq 0$, in  Conjecture {\rm \ref{conj: double Fourier}} is valid in the polar coordinates with $z = x e^{i\phi}$ and $\varv = y e^{i\omega}$ for
 \begin{align}
 	|\Re (\nu) | < \Re (2\beta), \Re (2\gamma), \quad \Re (\beta+\gamma) < \frac 1 2, \quad |\Re (\beta-\gamma)| < \frac 1 4, 
 \end{align} if the quadruple integral in {\rm\eqref{1eq: double Fourier, 2C}} is integrated in the order $ d \phi \, d \omega \,  d x\, d y$. 
\end{thm}



\section{Preliminaries} \label{sec: prelim}


\subsection{The classical Bessel functions} 

Recall that the Bessel function $J_{\nu} (z)$ is defined by the series
\begin{align}\label{2eq: defn of classical J}
	J_{\nu} (z) = \sum_{n=0}^{\infty} \frac {(-1)^n (z/2)^{\nu +2n}} {n! \Gamma (\nu + n + 1)}, 
\end{align}
and it satisfies the Bessel differential equation 
\begin{equation}\label{eq: Bessel Equation}
	z^2 \frac {d^2 \txw} {d z^2}  (z) + z \frac {d \txw} {d z} (z) + (  z^2 - \nu ^2  )  \txw(z) = 0 .
\end{equation}


For  $|z| \Lt 1$, when $\nu  \neq -1, -2, -3, ...$, it  follows from \cite[3.13 (1)]{Watson}  that  
\begin{equation}\label{2eq: bound for J}
	  J_{\nu } (z)   = \frac {(z/2)^{\nu }} {\Gamma (\nu +1)} \lp 1 + O_{\nu }   ( |z|^2 )  \rp ,  
\end{equation}
and 
\begin{equation}\label{2eq: bound for J, 2}
	z^{r} (d/dz)^{r} J_{\nu } (z) \Lt_{\, r, \shskip \nu } \left|z^{\nu } \right| .
\end{equation} 

For $|z| \Gt 1$, it is convenient to introduce the Hankel functions  $H^{(1)}_{\nu }  (z) $ and $H^{(2)}_{\nu }  (z) $. According to \cite[3.61 (1, 2)]{Watson},  
\begin{align}
	\label{2eq: J and H} 
	& J_\nu  (z) = \frac {H_\nu ^{(1)} (z) + H_\nu ^{(2)} (z)} 2, \hskip 30 pt 
	J_{-\nu } (z) =  \frac {e^{\pi i \nu } H_\nu ^{(1)} (z) + e^{-\pi i \nu } H_\nu ^{(2)} (z) } 2.
\end{align}
By \cite[7.2 (1, 2)]{Watson} and \cite[7.13.1]{Olver},  
\begin{align}\label{2eq: asymptotic H (1)}
	H^{(1)}_{\nu } (z) &= \lp \frac 2 {\pi z} \rp^{\frac 1 2} e^{ i \lp z - \frac 12 {\pi \nu }    - \frac 14 \pi    \rp }    
	\lp \sum_{n=0}^{N-1} \frac {(-1)^n \cdot (\nu , n) } {(2iz)^n} + E^{(1)}_N (z) \rp, \\
	\label{2eq: asymptotic H (2)}
	H^{(2)}_{\nu } (z) &= \lp \frac 2 {\pi z} \rp^{\frac 1 2} e^{ - i \lp z - \frac 12 {\pi \nu }    - \frac 1 4 \pi   \rp }   
	\lp \sum_{n=0}^{N-1} \frac {  (\nu , n) } {(2iz)^n} + E^{(2)}_N (z) \rp,
\end{align} 
with $$(\nu , n) = \frac {\Gamma \lp \nu  + n +   1 / 2 \rp}   {n! \Gamma \lp \nu  - n +   1 / 2 \rp},  $$ of which \eqref{2eq: asymptotic H (1)} is valid when $z$ is such that $- \pi +  \delta \leqslant \arg z \leqslant 2 \pi -   \delta$,  and  \eqref{2eq: asymptotic H (2)} when $- 2 \pi +  \delta \leqslant \arg z \leqslant   \pi -  \delta$, for $  \delta  $ any acute angle, and  
\begin{align}\label{2eq: estimates for E}
	z^r (d/dz)^r E_N^{(1, \shskip 2)} (z)  \Lt_{\shskip \delta, \shskip r, \shskip N, \shskip \nu } 1/ |z|^{N } , 
\end{align} 
for   $|z| \Gt 1$. 

\subsection*{Remarks on the uniformity and analyticity in the orders} One may choose all the implied constants above uniformly for $\nu $  in any given compact set. Consequently,   the Bessel integrals of concern are analytic functions in $ \nu $, $\mu$, ... on the domain of convergence. As it is quite tedious, in the sequel, this kind of implied uniformity will 
not be mentioned, but the analyticity in the orders will be presumed.

\subsection{The Bessel function $\boldJ_{\nu ,\shskip  p} (z)$} 
Define 
\begin{align}\label{2eq: nabla, 0}
	\nabla = z \frac {\partial  } {\partial z }, \qquad \overline  \nabla =  \widebar z \frac {\partial  } {\partial \widebar z }. 
\end{align}
We introduce 
\begin{equation}\label{2eq: nabla}
	\nabla_{\nu }   = 4 \nabla^2 +  16 \pi^2 z  - \nu ^2 = 4 z^2 \frac {\partial^2 } {\partial z^2} + 4 z \frac {\partial  } {\partial z }    +  16 \pi^2 z  - \nu ^2 ,
\end{equation}
and
\begin{equation}\label{2eq: nabla bar}
	\overline  \nabla_{\nu }  = 4 \overline  \nabla^2 + 16 \pi^2 \widebar z  - \nu ^2  = 4 \widebar z^2 \frac {\partial^2 } {\partial \widebar z^2} +   4 \widebar z \frac {\partial  } {\partial \widebar z }  + 16 \pi^2 \widebar z  - \nu ^2  .
\end{equation}
From the definition of $\boldJ_{ \nu,\shskip  p} (z) $ as in \eqref{0def: J mu m (z)}, \eqref{0eq: defn of Bessel} and the Bessel equation \eqref{eq: Bessel Equation}, we infer that 
\begin{align}\label{2eq: nabla J = 0}
	\nabla_{\nu + p} \lp \boldJ_{ \nu,\shskip  p} (\sqrt{z} ) \rp = 0, \qquad  \overline{\nabla}_{  \nu - p} \lp \boldJ_{ \nu,\shskip  p} (\sqrt{z} ) \rp = 0.
\end{align} 
Note that $  \boldJ_{ \nu,\shskip  p} (z) $ is an even function on $\BC \smallsetminus \{0\}$.


Next, we define  
\begin{equation}\label{0def: P mu m (z)}
	D_{\nu ,\shskip  p} (z) = \frac {(|z|/2)^{2\nu} (z/|z|)^{2p} } {\Gamma (\nu +p+1) \Gamma (\nu -p+1)} ,
\end{equation}
and
\begin{equation}\label{0eq: defn of P}
	\boldsymbol{P}_{ \nu,\shskip  p} (z)  =   \frac {1} {\sin (\pi \nu)} \lp D_{-\nu,\shskip  -p} (4 \pi   z) - D_{\nu,\shskip  p} (4 \pi   z)  \rp. 
\end{equation}
Note that $D_{\nu ,\shskip  p} (z)$ is the leading term in  $J_{\nu ,\shskip  p} (z)$ in view of the expansion \eqref{2eq: defn of classical J}. 

Set $\vlambda = |\Re (\nu)|$.  It follows from \eqref{2eq: bound for J} and \eqref{2eq: bound for J, 2} that if $\nu $ is not integral\footnote{On choosing $\vlambda > |\Re (\nu)|$, the validity of   \eqref{2eq: bound for J mu m} and \eqref{2eq: bound for J mu m, weak, 2} may be extended for $\nu $ integral. However, this requires some calculations by the formulae  of $ \left.( \partial J_{\nu }  (z) /\partial \nu  ) \right|_{\nu  = \pm n}$  ($n = 0, 1, 2,...$) in \cite[\S 3.52 (1, 2)]{Watson}. It would be easier to avoid the integral case by the principle of analytic continuation. }, then, for $|z| \Lt 1$, we  have
\begin{align}\label{2eq: bound for J mu m}
	\boldJ_{ \nu,\shskip  p} (z) - \boldsymbol {P}_{ \nu,\shskip  p} (z)  \Lt_{\, \nu,\shskip  p} |z|^{2 - 2 \vlambda} ,
\end{align}  
\begin{align}\label{2eq: bound for J mu m, weak, 1}
  \boldJ_{ \nu,\shskip  p} (z)   \Lt_{  \nu,\shskip  p}  |z|^{- 2 \vlambda } ,
\end{align}
and, more generally,
\begin{align}\label{2eq: bound for J mu m, weak, 2}
	z^{\shskip r} \widebar z^{ \,s} (\partial /\partial z)^r (\partial / \partial \widebar z)^{s} \boldJ_{ \nu,\shskip  p} (z)   \Lt_{\, r,\shskip s, \shskip \nu,\shskip  p}  |z|^{- 2 \vlambda }.
\end{align}

In view of the connection formulae in \eqref{2eq: J and H}, we have an alternative expression of $ \boldJ_{ \nu,\shskip  p} (z)$ in terms of Hankel functions:
\begin{equation}\label{2eq: J = H1 + H2}
	\boldJ_{ \nu,\shskip  p} (z) =  \frac i  2 \hskip -1 pt \big( e^{  \pi i \nu} H^{(1)}_{\nu,\shskip  p} \lp 4 \pi   z \rp  - e^{-   \pi i \nu} H^{(2)}_{\nu,\shskip  p} ( 4 \pi   z ) \big),
\end{equation}
with the definition
\begin{equation}\label{0def: H mu m (z)}
	H^{(1,\shskip  2)}_{\nu,\shskip  p} (z) = H^{(1,\shskip  2)}_{\nu + p} \lp   z \rp  H^{(1,\shskip  2)}_{\nu - p} \lp  { \widebar z} \rp.
\end{equation}
The reader should be warned that the product in \eqref{0def: H mu m (z)} is {\it not well-defined} as function on $\BC \smallsetminus \{0\}$.
By \eqref{2eq: asymptotic H (1)}--\eqref{2eq: estimates for E}, we may write
\begin{align}\label{1eq: J = W + W}
	\boldJ_{ \nu,\shskip  p} (z) =   {e (4 \Re  (z))}   \boldsymbol{W}_{ \nu,\shskip  p}   (z) +   {e (- 4 \Re (z))}   \boldsymbol{W}_{ \nu,\shskip  p}   (- z) + \boldsymbol{E}^N_{ \nu,\shskip  p}  (z), \
\end{align}
where 
$\boldsymbol{W} (z)$ and $ \boldsymbol{E}_N (z) $ are real analytic functions on $\BC \smallsetminus \{0\}$ satisfying 
\begin{align}\label{1eq: derivatives of W}
	& z^{  r} \shskip \widebar z^{ \,s} (\partial /\partial z)^ r (\partial / \partial \widebar z)^{s} \boldsymbol{W}_{ \nu,\shskip  p}  (z) \Lt_{\,  r,\shskip s, \shskip N, \shskip \nu,\shskip  p} 1 / |z|,     \\
	\label{1eq: derivatives of E}
	& (\partial /\partial z)^ r (\partial / \partial \widebar z)^{s} \boldsymbol{E}_{ \nu,\shskip  p}^N   (z) \Lt_{\,  r,\shskip s, \shskip N, \shskip \nu,\shskip  p} 1 / |z|^{N+1},   
\end{align}   
for $|z|  \Gt 1$. It follows that 
\begin{align}\label{2eq: bound 1/z}
	 \boldJ_{ \nu,\shskip  p} (z) = O_{\nu, \shskip p} (1/|z|) , 
\end{align}
for $|z| \Gt 1$, and by \eqref{2eq: bound for J mu m, weak, 1},   for all $z$ provided that $ \vlambda \leqslant 1/2$.

\subsection{Mellin transform of $ \boldJ_{ \nu,\shskip  p}   (z)$} 

As will be seen later in \S \ref{sec: asymptotic}, the definition of the hypergeometric function $  {_{\phantom{\nu}}^{\rho} \hskip -1pt \boldF_{\nu, \shskip p}^{\mu, \shskip  d}} (z)$ 
as in \eqref{1eq: hypergoem, C, 1.0}--\eqref{1eq: hypergeometric, 3} stems from the following Mellin-type integral formula. 

\begin{lem}\label{lem: Mellin}
	Let   $\gamma, \nu \in \BC$ and $ h,  p  \in \BZ$.   Then 
	\begin{equation}\label{3eq: Mellin}
		\begin{split}
		&	\viint_{\BC^{\times}}   \boldJ_{ \nu,\shskip  p}   (z)    |z|^{2 \gamma   - 2} (z/|z|)^{2h}  i d      z \vwedge d      \widebar{z} = - \frac {2  (\cos (\pi\gamma) -	(-1)^{h+p} \cos (\pi\nu)) } {(2 \pi)^{2\gamma + 2}} \\
	& \cdot 	\Gamma \hskip -1.5pt  \lp \hskip -1pt \frac {\gamma+ h+\nu+p}  2 \hskip -1pt \rp \hskip -1.5pt \Gamma \hskip -1.5pt \lp \hskip -1pt \frac {\gamma+ h-\nu-p}  2 \hskip -1pt \rp \hskip -1.5pt \Gamma \hskip -1.5pt  \lp \hskip -1pt \frac {\gamma-h+\nu+p}  2 \hskip -1pt \rp \hskip -1.5pt \Gamma \hskip -1.5pt \lp \hskip -1pt \frac {\gamma-h-\nu-p}  2 \hskip -1pt \rp	\hskip -1.5pt  , 
		\end{split}
	\end{equation}
	for   $ |\mathrm{Re} (\nu)| < \mathrm{Re}(\gamma) <  1 / 2 $. 
\end{lem}



\begin{proof}
	Recall the definition of $ F^{\gamma}_{ \nu } (z)  $ as in \eqref{1eq: defn of F, 21}:
	\begin{align}\label{2eq: defn of F, 21}
		F^{\gamma}_{ \nu } (z) = \frac {      \Gamma (\gamma + \nu)     } {   \Gamma  (1 +  \nu )  } \cdot  z^{\frac 1 2 (\gamma + \nu)} {_2 F_1}  \bigg(   \frac {\gamma  +   \nu } {2}; \frac {1   +   \gamma +   \nu  } {2}; 1   +  \nu  \shskip; z   \bigg) .
	\end{align}
	Define 
	\begin{equation}\label{3eq: F, 31}
		F^{\gamma, \shskip h}_{\nu,\shskip  p} (z) =	2^{-2\nu} \sin (\pi (\gamma+\nu)) F^{\gamma+h}_{\nu+p} (z) F^{\gamma - h}_{\nu-p} (\widebar{z}) ,
	\end{equation}  
	and  
	\begin{equation}\label{3eq: F, 32}
		\boldF^{\gamma, \shskip h}_{\nu,\shskip  p} (z) =  \frac {1} {\sin (\pi \nu)} \big( {F}^{\gamma, \shskip h}_{-\nu, \shskip -p} (z) -  {F}^{\gamma, \shskip h}_{\nu,\shskip  p} (z)   \big) . 
	\end{equation}	 
	By \cite[Theorem A.1]{Qi-BE}, we have  
	\begin{align}\label{3eq: cos Mellin}
		\viint_{\BC^{\times}}   \boldJ_{ \nu,\shskip  p}   (z)  \cos (4\pi \mathrm{Re}(u z))  |z|^{2 \gamma   - 2} (z/|z|)^{2h}  i d      z \vwedge d      \widebar{z} = \frac 2 {(4\pi)^{2\gamma}}  \boldF^{\gamma, \shskip h}_{\nu,\shskip  p} \bigg( \hskip -1pt \frac 4  {u^2} \hskip -1pt  \bigg). 
	\end{align}
	Thus, to evaluate the Mellin integral in \eqref{3eq: Mellin}, we just need to let $4/u^2 \ra \infty$. To this end, we invoke the   transformation formula for ${_2F_1} (a,b;c; z)$ with respect to $z \ra 1/z$ in \eqref{2eq: transformation z - 1/z}. It follows that 
	\begin{align*}
		\mathop{\lim_{|z| \ra \infty}}_{\arg (z) \in (0,2\pi)} F^{\gamma}_{ \nu } (z) & =   \frac {\sqrt{\pi} \Gamma (\gamma+\nu) e^{  \frac 1 2 \pi i (\gamma + \nu)} } { \Gamma ((1+\gamma+\nu)/2) \Gamma ((2-\gamma+\nu)/ 2)}   , \\
		\mathop{\lim_{|z| \ra \infty}}_{\arg (z) \in (0,2\pi)} F^{\gamma}_{ \nu } (\widebar{z}) & =   \frac {\sqrt{\pi} \Gamma (\gamma+\nu) e^{ - \frac 1 2 \pi i (\gamma + \nu)} } { \Gamma ((1+\gamma+\nu)/2) \Gamma ((2-\gamma+\nu)/ 2)}  .
	\end{align*} 
	Note that we need to choose  $\arg(- 1) = - \pi$ or $\arg(- 1) =  \pi$ in order for $ |\arg (-z)| < \pi $ or  $ |\arg (-\widebar{z})| < \pi $ to hold, respectively. Finally, we let \eqref{3eq: F, 31}--\eqref{3eq: cos Mellin} pass to the limit, and conclude the proof by the duplication and reflection formulae for $\Gamma (s)$ followed by some   trigonometric calculations.  
\end{proof}

\begin{cor} \label{cor: Mellin, spherical}
		We have
		\begin{equation}\label{3eq: Mellin, spherical}
			\viint_{\BC^{\times}}   \boldJ_{ \nu }   (\hskip -1.5pt \sqrt{\varv z})    |z|^{  \gamma   - 2}   i d      z \vwedge d      \widebar{z} =   \frac {  \cos (\pi\nu) -   \cos (\pi\gamma)     } {\pi^2 (2 \pi)^{2\gamma } |\varv|^{\gamma} } \Gamma \lp \frac {\gamma+\nu}  2 \rp^2 \Gamma \lp \frac {\gamma-\nu}  2 \rp^2, 
		\end{equation}
		for   $ |\mathrm{Re} (\nu)| < \mathrm{Re}(\gamma) <   1/ 2 $.  
\end{cor}

\subsection{The classical hypergeometric function}\label{sec: classical hyper} Recall that the hypergeometric function ${_2F_1} (a, b; c; z)$  is defined by the Gauss series 
\begin{align}\label{2eq: Gauss series}
	{_2F_1} (a, b; c; z) = \sum_{n=0}^{\infty} \frac {(a)_n (b)_n } { (c)_n  n!} z^n =  \frac {\Gamma (c)} {\Gamma (a) \Gamma (b)} \sum_{n=0}^{\infty} \frac {\Gamma (a+n) \Gamma (b+n)} {\Gamma (c+n) n!} z^n
\end{align} 
within its circle of convergence $|z| < 1$, and  by analytic continuation elsewhere. The series is absolutely convergent on the unit circle $|z| = 1$ if $\Re(a+b-c) < 0$. 
The function $ {_2F_1} (a, b; c; z) $ is a single-valued analytic function of $z$ 
with a branch cut along the positive real axis from
$1$ to $\infty$. Moreover, $ 1/\Gamma (c) \cdot {_2F_1} (a, b; c; z)$ is analytic in $a$, $b$, and $c$. 
See \cite[\S 2.1]{MO-Formulas}. 

The hypergeometric differential equation reads: 
\begin{align}\label{2eq: hypergeometric equation}
	z (1-z) \frac {d^2 \txw} {d z^2}  (z) + (c - (1+a+b) z ) \frac {d \txw} {d z} (z) - a b  \txw(z) = 0.
\end{align}
and it has three regular singular
points $z = 0, 1, \infty$. According to   \cite[\S 2.2]{MO-Formulas}, in the generic case when none of  $c$, $ a-b$, and $ c-a-b$ is an integer, the functions
\begin{align*}
	 {_2F_1} (a, b; c; z), \qquad z^{1-c}  {_2F_1} (1+a-c, 1+b-c; 2 - c; z), 
\end{align*} 
form a system of (linearly independent) solutions of \eqref{2eq: hypergeometric equation} in the vicinity of $z = 0$. 

It is easy to prove that in the generic case when none of $ \nu $, $\mu$, and  $\rho$ is an integer, the functions (see \eqref{1eq: hypergoem, C, 1.0}) 
\begin{align*}
& {_{\phantom{\nu}}^{\rho} \hskip -1.5pt f_{\nu}^{\mu}} (z) =	z^{\frac 1 2  \nu }	{_2F_1} \hskip -1pt \bigg(   \frac {\rho + \nu +\mu} 2, \frac {\rho+\nu - \mu} 2  ; 1+ \nu ;  z    \bigg), \\
& {_{\phantom{\nu}}^{\rho} \hskip -1.5pt f_{-\nu}^{\mu}} (z) = z^{- \frac 1 2  \nu }	{_2F_1} \hskip -1pt \bigg(   \frac {\rho - \nu +\mu} 2, \frac {\rho-\nu - \mu} 2  ; 1 - \nu ;  z    \bigg) , 
\end{align*} form a system of solutions near $z = 0$ of the differential equation: 
\begin{align}\label{2eq: hypergeometric equation, 2}
	  z (1-z) \frac {d^2 \txw} {d z^2}  (z) +   (1- (1+\rho) z) \frac {d \txw} {d z} (z) - \bigg(   \frac {  \rho^2  - \mu^2} {4} + \frac {\nu^2} {4 z} \bigg)  \txw(z) = 0.
\end{align}

Finally, we record here  the transformation formula with respect to $z \ra 1/z$ (see \cite[\S 2.4.1]{MO-Formulas}):
\begin{equation}\label{2eq: transformation z - 1/z}
	\begin{split}
		{_2F_1} (a, b; c; z)  =\, & \frac {\Gamma (c) \Gamma (b-a)} {\Gamma (b) \Gamma (c-a) } (-z)^{-a} {_2F_1} ( a, a-c+1; a-b+1; 1/z ) \\
		+\, & \frac {\Gamma (c) \Gamma (a-b)} {\Gamma (a) \Gamma (c-b) } (-z)^{-b} {_2F_1} ( b, b-c+1; b-a+1; 1/z ),\\
		&\hskip 38 pt |\arg (-z)| < \pi, \ a-b \neq \pm n, \ \text{($n = 0, 1, 2, ...$),}
	\end{split}
\end{equation}
and a quadratic transformation formula of Kummer (see \cite[\S 2.4.2]{MO-Formulas} or \cite[\S 2.1.6 (32), (33)]{Erdelyi-HTF-1}):
\begin{align}\label{2eq: quad transform}
	{_2F_1} (a, b; 2b; z) = \lp 1- \frac z 2 \rp^{-a} {_2F_1} \hskip -1pt \lp \frac {a} 2, \frac {a+1} 2; b+\frac 1 2 ; \frac {z^2} { (2-z)^2 } \rp, \quad |z| < |2-z|. 
\end{align} 
It is only assumed that $z$ is in the vicinity of $0$ for the list of  quadratic transformation formulae in \cite[\S 2.4.2]{MO-Formulas}. The condition $ |z| < |2-z|  $ (or $\Re (z) < 1$) in \eqref{2eq: quad transform} arises from examining its proof in \cite[\S 2.1.6]{Erdelyi-HTF-1}.

\subsection{The hypergeometric function  $   {_{\protect\phantom{\nu}}^{\rho} \hskip -1pt \boldF_{\nu, \shskip p}^{\mu, \shskip  d}} (z)$} \label{sec: prel hyper} 
For $   \nabla$ and $\overline {\nabla}$ as in \eqref{2eq: nabla, 0}, define 
\begin{align}\label{2eq: nabla, 1}
	{_{\phantom{\nu}}^{\rho} \hskip -0.5pt   \nabla_{\nu }^{\mu}} =  (1-z) \nabla^2 - \rho z \shskip  \nabla - \bigg(   \frac {  \rho^2  - \mu^2} {4} z + \frac {\nu^2} {4  } \bigg),  
\end{align}
and its conjugate 
\begin{align}\label{2eq: nabla, 2}
	{_{\phantom{\nu}}^{\rho} \hskip -0.5pt \overline{\nabla}{}_{\nu }^{\mu}} = 
	(1-\widebar z) \overline\nabla^2 - \rho \widebar z \shskip \overline \nabla - \bigg(   \frac {  \rho^2  - \mu^2} {4} \widebar z + \frac {\nu^2} {4  } \bigg) . 
\end{align}
Then, in view of \eqref{1eq: hypergoem, C, 1.0}--\eqref{1eq: hypergoemetric} and \eqref{2eq: hypergeometric equation, 2}, it is clear that 
\begin{align}
	 {_{\phantom{\nu}}^{\rho} \hskip -0.5pt \nabla_{\nu + p }^{\mu + d}} \big( {_{\phantom{\nu}}^{\rho} \hskip -1pt \boldF_{\nu, \shskip p}^{\mu, \shskip  d}} (z) \big) = 0, \qquad {_{\phantom{\nu}}^{\rho} \hskip -0.5pt \overline \nabla{}_{\nu - p }^{\mu - d}} \big( {_{\phantom{\nu}}^{\rho} \hskip -1pt \boldF_{\nu, \shskip p}^{\mu, \shskip  d}} (z) \big) = 0. 
\end{align}

Next, in parallel to \eqref{1eq: hypergoem, C, 1.0}--\eqref{1eq: hypergoemetric}, we define 
\begin{equation} \label{3eq: hyper power, C, 1}
	{_{\phantom{\nu}}^{\rho} \hskip -1pt  {D}_{\nu}^{\mu}} (z) \hskip -0.5pt = \hskip -0.5pt	\frac {\Gamma  (   ({\rho + \nu +\mu}) / 2    ) \Gamma  (     {(\rho+\nu - \mu)}/ 2     )} {\Gamma (1+\nu)} \cdot  z^{\frac 1 2 \nu}  , 
\end{equation} 
\begin{equation}\label{1eq: hyper power, C, 2}
	{_{\phantom{\nu}}^{\rho} \hskip -1.5pt {D}_{\nu, \shskip p}^{\mu, \shskip  d}} (z)  =  \big( \hskip -1pt   \cos (\pi(\rho+\nu))  \hskip -0.5pt -  \hskip -1pt	(-1)^{d+p} \cos (\pi\mu) \big) \cdot {_{\phantom{\nu}}^{\rho} \hskip -1pt D_{\nu+p}^{\mu+d}} (z) {_{\phantom{\nu}}^{\rho} \hskip -1pt D_{\nu-p}^{\mu-d}} (\widebar{z}),  
\end{equation}
and 
\begin{equation}\label{1eq: hyper power}
	{_{\phantom{\nu}}^{\rho} \hskip -1pt \boldsymbol{P}_{\nu, \shskip p}^{\mu, \shskip  d}} (z) = 
	\frac 1 {\sin (\pi \nu)} \big(  {_{\phantom{\nu}}^{\rho} \hskip -1.5pt {D}_{\nu, \shskip p}^{\mu, \shskip  d}} (z)  - {_{\phantom{\nu}}^{\rho} \hskip -1.5pt {D}_{-\nu, \shskip -p}^{\mu, \shskip  d}}  (z)  \big). 
\end{equation} 
Note that $ {_{\phantom{\nu}}^{\rho} \hskip -1.5pt {D}_{\nu, \shskip p}^{\mu, \shskip  d}} (z)$ is the leading term in  ${_{\phantom{\nu}}^{\rho} \hskip -1.5pt {F}_{\nu, \shskip p}^{\mu, \shskip  d}} (z) $ in view of the Gauss series expansion \eqref{2eq: Gauss series}.  
More succinctly, one may also   write 
\begin{align}\label{1eq: hyper power, 3}
{_{\phantom{\nu}}^{\rho} \hskip -1pt \boldsymbol{P}_{\nu, \shskip p}^{\mu, \shskip  d}} (z) = 
	{_{\phantom{\nu}}^{\rho} \hskip -1.5pt {C}_{\nu, \shskip p}^{\mu, \shskip  d}} \cdot |z|^{\nu} (z/|z|)^p  + {_{\phantom{\nu}}^{\rho} \hskip -1.5pt {C}_{-\nu, \shskip -p}^{\mu, \shskip  d}} \cdot |z|^{-\nu} (z/|z|)^{-p} , 
\end{align}
in comparison with  \eqref{1eq: hypergeometric, 3}. It is   clear that 
\begin{align}\label{2eq: asymp F sim P}
{_{\phantom{\nu}}^{\rho} \hskip -1pt \boldF_{\nu, \shskip p}^{\mu, \shskip  d}} (z) 	\sim
{_{\phantom{\nu}}^{\rho} \hskip -1pt \boldsymbol{P}_{\nu, \shskip p}^{\mu, \shskip  d}} (z), \qquad z \ra 0, 
\end{align}
provided   $|\Re (\nu)| <  {1} / {2}$. 

\subsection{Transformation formulae for \protect$ {_{\protect \phantom{\nu}}^{\rho} \hskip -1pt \boldF_{\nu, \shskip p}^{\mu, \shskip  d}} (z)$} 

In this section, we apply the transformation formulae for $ {_2F_1} (a, b; c; z) $ in  \eqref{2eq: transformation z - 1/z} and \eqref{2eq: quad transform} to deduce the corresponding formulae for ${_{\phantom{\nu}}^{\rho} \hskip -1pt \boldF_{\nu, \shskip p}^{\mu, \shskip  d}} (z)$ and $  \boldF^{\rho}_{\nu, \shskip p} (z)$. 
For the proofs, we shall always assume that $\nu$ and $\mu$ are non-integral---the non-generic case follows by the limit. This kind of deduction should be easily adapted for other transformation formulae. 

The following reciprocity formula for $   {_{\phantom{\nu}}^{\rho} \hskip -1pt \boldF_{\nu, \shskip p}^{\mu, \shskip  d}} (z) $ is due to  the transformation formula for $ {_2F_1} (a, b; c; z) $ with respect to $z \ra 1/z$ as in \eqref{2eq: transformation z - 1/z}.  Since $z \leftrightarrow 1/z$ yields $(0, 1)  \leftrightarrow (1, \infty)$, the real analyticity of $   {_{\phantom{\nu}}^{\rho} \hskip -1pt \boldF_{\nu, \shskip p}^{\mu, \shskip  d}} (z) $ on $\BC \smallsetminus \{0, 1\}$ is indeed  a direct consequence. 

\begin{lem}\label{lem: reciprocity} For $z \in \BC \smallsetminus \{0, 1\}$ we have 
	\begin{equation}\label{2eq: reciprocity}
		|z|^{\rho} \cdot  {_{\phantom{\nu}}^{\rho} \hskip -1pt \boldF_{\nu, \shskip p}^{\mu, \shskip  d}} (z) =   {_{\phantom{\nu}}^{\rho} \hskip -1pt \boldF^{\nu, \shskip p}_{\mu, \shskip  d}} \lp   1 / z \rp . 
	\end{equation}
\end{lem}

\begin{proof}

First let  $\arg (z) \in (0, 2\pi)$.  By \eqref{1eq: hypergoem, C, 1.0}, \eqref{1eq: hypergoem, C, 1}, and \eqref{2eq: transformation z - 1/z}, along with the reflection  formula for $\Gamma (s)$, we infer that
	 \begin{align*}
	 z^{ \frac 1 2  \rho  } \cdot	{_{\phantom{\nu}}^{\rho} \hskip -1.5pt F_{\nu}^{\mu}} (z)   =  & -  \frac { \sin (\pi (\rho-\nu+\mu)   /2)    } {\sin (\pi \mu)    } e^{\frac 1 2 \pi i (\rho+\nu+\mu)} \cdot {_{\phantom{\nu}}^{\rho} \hskip -1pt  {F}_{\mu}^{\nu}}  (1/z) \\
	 	& +  \frac { \sin (\pi (\rho-\nu-\mu)  /2)    } {\sin (\pi \mu)    } e^{\frac 1 2 \pi i (\rho+\nu-\mu)} \cdot  {_{\phantom{\nu}}^{\rho} \hskip -1pt  {F}_{-\mu}^{\nu}}  (1/z) ,
	 \end{align*}
 and similarly
 \begin{align*}
 \widebar{z}^{ \frac 1 2  \rho  } \cdot	{_{\phantom{\nu}}^{\rho} \hskip -1.5pt F_{\nu}^{\mu}} (\widebar{z})   =  & -  \frac { \sin (\pi (\rho-\nu+\mu)   /2)    } {\sin (\pi \mu)    } e^{- \frac 1 2 \pi i (\rho+\nu+\mu)} \cdot {_{\phantom{\nu}}^{\rho} \hskip -1pt  {F}_{\mu}^{\nu}}  (1/ \widebar{z}) \\
 	& +  \frac { \sin (\pi (\rho-\nu-\mu)  /2)    } {\sin (\pi \mu)    } e^{- \frac 1 2 \pi i (\rho+\nu-\mu)}  \cdot {_{\phantom{\nu}}^{\rho} \hskip -1pt  {F}_{-\mu}^{\nu}}  (1/ \widebar{z}) .
 \end{align*}
 It follows from applying these to \eqref{1eq: hypergoem, C, 2} and \eqref{1eq: hypergoemetric} that $|z|^{\rho} \cdot  {_{\phantom{\nu}}^{\rho} \hskip -1pt \boldF_{\nu, \shskip p}^{\mu, \shskip  d}} (z)$ is a linear combination of ${_{\phantom{\nu}}^{\rho} \hskip -1pt  {F}_{\mu+d}^{\nu+p}}  (1/z) {_{\phantom{\nu}}^{\rho} \hskip -1pt  {F}_{\mu-d}^{\nu-p}}  (1/\widebar{z})$, ${_{\phantom{\nu}}^{\rho} \hskip -1pt  {F}_{\mu+d}^{\nu+p}}  (1/z) {_{\phantom{\nu}}^{\rho} \hskip -1pt  {F}_{-\mu+d}^{\nu-p}}  (1/\widebar{z})$, and two other similar products.  By some   trigonometric calculations, it turns out that the coefficient of ${_{\phantom{\nu}}^{\rho} \hskip -1pt  {F}_{\mu+d}^{\nu+p}}  (1/z) {_{\phantom{\nu}}^{\rho} \hskip -1pt  {F}_{\mu-d}^{\nu-p}}  (1/\widebar{z})$ is exactly $  \big( \hskip -1pt   \cos (\pi(\rho+\mu))  \hskip -0.5pt -  \hskip -1pt	(-1)^{d+p} \cos (\pi\nu) \big) / \sin (\pi \mu) $, while that of ${_{\phantom{\nu}}^{\rho} \hskip -1pt  {F}_{\mu+d}^{\nu+p}}  (1/z) {_{\phantom{\nu}}^{\rho} \hskip -1pt  {F}_{-\mu+d}^{\nu-p}}  (1/\widebar{z})$ is equal to zero!

Thus \eqref{2eq: reciprocity} is proven for $z \in \BC \smallsetminus \overline{\BR}_+$ ($\arg (z) \in (0, 2\pi)$). However, its left or right hand side is real analytic on $(0, 1)$ or $(1, \infty)$, respectively, so  \eqref{2eq: reciprocity} remains valid for all $ z \in \BC \smallsetminus \{0, 1\} $. 
\end{proof}

The next formula is a   consequence of Kummer's quadratic transformation law \eqref{2eq: quad transform}. 

\begin{lem}
	\label{lem: quad trans} 
	For $ \Re (z) < 1 $ and $z \neq 0$ we have 
	\begin{equation}\label{2eq: quad trans}
		\pi \boldF^{\rho}_{\nu, \shskip p} \lp \frac {z^2} { (2-z)^2} \rp =	     {(|z|/2)^{2\rho}}      \cdot  {_{\phantom{\nu}}^{\frac 1 2 + \rho } \hskip -1.5pt \boldF_{2\nu, \shskip 2p}^{\frac 1 2 - \rho   }} (z) .  
	\end{equation} 
\end{lem}

\begin{proof}
	It follows from    \eqref{1eq: hypergoem, C, 1.0}, \eqref{1eq: hypergoem, C, 1} \eqref{1eq: defn of F, 21}, and \eqref{2eq: quad transform}, along with the duplication  formula for $\Gamma (s)$,  that 
	\begin{align*}
	\sqrt{\pi}	F^{\rho}_{ \nu } \lp \frac {z^2} { (2-z)^2} \rp  =    {2^{ \nu - \rho} }   z^{ \shskip  \rho  } \cdot {_{\phantom{\nu}}^{\frac 1 2 + \rho } \hskip -1.5pt F_{2\nu}^{\frac 1 2 - \rho   }} (z)  . 
	\end{align*}
Then, in view of  the definitions \eqref{1eq: hypergoem, C, 2}--\eqref{1eq: hypergoemetric} and \eqref{1eq: sim hyper, 1}--\eqref{1eq: sim hyper, 2}, the identity \eqref{2eq: quad trans} follows directly by simple   trigonometric calculations.
\end{proof}

\section{Analytic Lemmas} 

\begin{defn}\label{defn: annulus}
	For an interval $ I   \subset \BR_+$, define the  annulus 
	\begin{align*}
		\BA (I)= \big\{ z \in \BC\smallsetminus \{ 0 \} : |z| \in I \big\}{\rm;}
	\end{align*}
	usually,  we shall write $\BA (I) = \BA (c, d)$, for example,  if $I = (c, d)$ is an open interval. 
\end{defn}

\subsection{A simple lemma on oscillatory integrals} 


By Stokes' theorem, one can easily prove the following lemma. 

	\begin{lem}\label{lem: integration by parts}
	Let $a,  b \in \BC$ with $ |a| < | b| $. 
	Let $\gamma <  {1}/{2} $.  
	Suppose that $ f (z) $ is a $C^{2M}$ function on $\BC$ that is supported on the annulus $\BA [ c, \infty)$ and satisfies  $$z^{  r} \shskip \widebar z^{ s} (\partial /\partial z)^r (\partial / \partial \widebar z)^{s} f (z) \Lt_{r, \shskip s, \shskip \gamma} S |z|^{2 \gamma - 2}, $$
	for any $r+s \leqslant 2 M$.  Define
	\begin{align*}
		I ( a, b; f ) = \viint_{\BA [ c, \infty) }  f (z) e ( \mathrm{Re}  ((a  + b ) z) ) i d      z \vwedge d      \widebar{z} .
	\end{align*}
	Then  we have
	\begin{align*}
		I ( a, b; f ) =  \frac {I ( a, b ;    {\mathrm{D}^{M} f (z)} )} { (\pi i |a+b|)^{2M} } , \qquad  
		\mathrm{D} =  \partial^2 / \partial z  \partial \widebar{z}   , 
	\end{align*} 
where $ \mathrm{D}^{M} f (z)  $ is a continuous function  with support in $\BA[ c, \infty)$ and bound  
	\begin{align*}
		  \mathrm{D}^{M} f (z)  \Lt_{M, \shskip \gamma} S { |z|^{  2 \gamma -  2 M - 2}    } . 
	\end{align*}
Consequently, if $M > 0$, then  $I ( a, b; f ) $ is convergent {\rm(}it is already absolutely convergent if $\gamma < 0${\rm)} and 
	\begin{align*}
	I ( a, b; f ) \Lt_{M, \shskip \gamma } \frac {S} {c^{2 M - 2\gamma} (|b| - |a|)^{2 M} } .
	\end{align*}
Moreover, the integral $I  ( a, b; f )$ gives rise to a continuous function of $ a $ on $\BA (0, |b|)$ {\rm(}the continuity extends to all complex values of $ a $ if $ \gamma < 0 ${\rm)}.  
\end{lem}

\subsection{A uniqueness lemma for the hypergeometric equations} 

Let notation and definitions be as in \S\S  \ref{sec: hyper}, \ref{sec: classical hyper}, and \ref{sec: prel hyper}. 

\begin{lem}\label{lem: hypergeometric equation} 
	Let $ \nu $, $\mu$, and $\rho$ be non-integral.  Assume $|\Re (\nu)| <   1 / 2 $. Let $f (z)$ be a continuous function on $ \BC \smallsetminus \{0, 1\}$ which is a solution of the following two differential equations{\rm:}
	\begin{equation}\label{2eq: differential equations}
		{_{\phantom{\nu}}^{\rho} \hskip -0.5pt \nabla_{\nu + p }^{\mu + d}} \, \txw = 0, \qquad {_{\phantom{\nu}}^{\rho} \hskip -0.5pt \overline \nabla{}_{\nu - p }^{\mu - d}} \, \txw = 0 .
	\end{equation}   Suppose further that $f (z)$ admits the following asymptotic{\rm:}
\begin{align}\label{2eq: asymptoic f (z)}
	f ( z ) \sim c_{+}  |z|^{\nu} (z/|z|)^p  + c_{-} |z|^{-\nu} (z/|z|)^{-p}, \qquad |z| \ra 0.
\end{align}
Then  \begin{align}\label{2eq: identity f (z)}
	f ( z) = c_{+} \cdot {_{\phantom{\nu}}^{\rho} \hskip -2pt f_{\nu+p}^{\mu+d}} (z) \, {_{\phantom{\nu}}^{\rho} \hskip -2pt f_{\nu-p}^{\mu-d}} (\widebar z)  + c_{-} \cdot {_{\phantom{\nu}}^{\rho} \hskip -2pt f_{-\nu-p}^{\mu+d}} (z) \, {_{\phantom{\nu}}^{\rho} \hskip -2pt f_{-\nu+p}^{\mu-d}} (\widebar z)
\end{align}
for all $z \in \BC \smallsetminus \{0, 1\}$.
\end{lem}

\begin{proof}
Since, in the generic case, any holomorphic (or anti-holomorphic) solution to $ 	{_{\phantom{\nu}}^{\rho} \hskip -0.5pt \nabla_{\nu   }^{\mu }} \, \txw = 0 $ (or ${_{\phantom{\nu}}^{\rho} \hskip -0.5pt \overline \nabla{}_{\nu }^{\mu  }} \, \txw = 0$) is a linear combination of $ {_{\phantom{\nu}}^{\rho} \hskip -2pt f_{\pm \nu }^{ \mu }} (z) $ (or ${_{\phantom{\nu}}^{\rho} \hskip -2pt f_{\pm \nu}^{\mu }} (\widebar z)$),    it follows from \eqref{2eq: differential equations} that $f (z)$ may be uniquely written as a linear combination of the products 
	\begin{align*}
	  {_{\phantom{\nu}}^{\rho} \hskip -2pt f_{\nu+p}^{\mu+d}} (z) \, {_{\phantom{\nu}}^{\rho} \hskip -2pt f_{\nu-p}^{\mu-d}} (\widebar z)  , \quad {_{\phantom{\nu}}^{\rho} \hskip -2pt f_{\nu+p}^{\mu+d}} (z) \, {_{\phantom{\nu}}^{\rho} \hskip -2pt f_{-\nu+p}^{\mu-d}} (\widebar z)  , \quad 
	  {_{\phantom{\nu}}^{\rho} \hskip -2pt f_{-\nu-p}^{\mu+d}} (z) \, {_{\phantom{\nu}}^{\rho} \hskip -2pt f_{\nu-p}^{\mu-d}} (\widebar z)  , \quad {_{\phantom{\nu}}^{\rho} \hskip -2pt f_{-\nu-p}^{\mu+d}} (z) \, {_{\phantom{\nu}}^{\rho} \hskip -2pt f_{-\nu+p}^{\mu-d}} (\widebar z) ,
	\end{align*}
or, in other words,
\begin{align*}
	f (z) = f_+ (z) + f_- (z), 
\end{align*}
with
\begin{align*}
	f_+ (z) & = a_{+} \cdot  {_{\phantom{\nu}}^{\rho} \hskip -2pt f_{\nu+p}^{\mu+d}} (z) \, {_{\phantom{\nu}}^{\rho} \hskip -2pt f_{\nu-p}^{\mu-d}} (\widebar z) + a_{-}   \cdot {_{\phantom{\nu}}^{\rho} \hskip -2pt f_{-\nu-p}^{\mu+d}} (z) \, {_{\phantom{\nu}}^{\rho} \hskip -2pt f_{-\nu+p}^{\mu-d}} (\widebar z) , \\
	f_- (z) & = b_{+} \cdot  {_{\phantom{\nu}}^{\rho} \hskip -2pt f_{\nu+p}^{\mu+d}} (z) \, {_{\phantom{\nu}}^{\rho} \hskip -2pt f_{-\nu+p}^{\mu-d}} (\widebar z) 
	 + b_{-}  \cdot  {_{\phantom{\nu}}^{\rho} \hskip -2pt f_{-\nu-p}^{\mu+d}} (z) \, {_{\phantom{\nu}}^{\rho} \hskip -2pt f_{\nu-p}^{\mu-d}} (\widebar z) ,
\end{align*}
for all $z$ away from $0$ and the branch cut from $1$ to $\infty$.   

Let us first restrict to $z \in \BA (0, 1)$. 

Of course, we need to prove $f_- (z) = 0$. In view of the definition \eqref{1eq: hypergoem, C, 1.0}, the quotient of ${_{\phantom{\nu}}^{\rho} \hskip -2pt f_{\nu+p}^{\mu+d}} (z) \, {_{\phantom{\nu}}^{\rho} \hskip -2pt f_{-\nu+p}^{\mu-d}} (\widebar z) $ and $ (z/|z|)^{\nu} $ as well as that of $ {_{\phantom{\nu}}^{\rho} \hskip -2pt f_{-\nu-p}^{\mu+d}} (z) \, {_{\phantom{\nu}}^{\rho} \hskip -2pt f_{\nu-p}^{\mu-d}} (\widebar z )$ and $ (z/|z|)^{- \nu} $ are both single-valued (real analytic) functions on  $\BA (0, 1)$. However, since $   \nu $ is not an integer, the functions $ (z/|z|)^{  \nu}$ and  $ (z/|z|)^{  - \nu}$ are {not} single-valued. On choosing $\arg (z) = 0, 2 \pi$, it follows from simple considerations that any non-trivial linear combination of  ${_{\phantom{\nu}}^{\rho} \hskip -2pt f_{\nu+p}^{\mu+d}} (z) \, {_{\phantom{\nu}}^{\rho} \hskip -2pt f_{-\nu+p}^{\mu-d}} (\widebar z) $ and $ {_{\phantom{\nu}}^{\rho} \hskip -2pt f_{-\nu-p}^{\mu+d}} (z) \, {_{\phantom{\nu}}^{\rho} \hskip -2pt f_{\nu-p}^{\mu-d}} (\widebar z) $ is not independent on $\arg (z)$ modulo $2 \pi$. So we must have $b_{+} = b_{-} = 0$ and hence  $ f_{-}(z) = 0$.

Next, we need to prove $a_{+} = c_{+}$ and $a_{-} = c_{-}$. For this, we deduce from  \eqref{1eq: hypergoem, C, 1.0} and \eqref{2eq: Gauss series} that
\begin{align*}
	  {_{\phantom{\nu}}^{\rho} \hskip -2pt f_{\nu+p}^{\mu+d}} (z) \, {_{\phantom{\nu}}^{\rho} \hskip -2pt f_{\nu-p}^{\mu-d}} (\widebar z) & = |z|^{\nu} (z/|z|)^p  + O \lp \big||z|^{1+\nu} \big| \rp, \\
	  {_{\phantom{\nu}}^{\rho} \hskip -2pt f_{-\nu-p}^{\mu+d}} (z) \, {_{\phantom{\nu}}^{\rho} \hskip -2pt f_{-\nu+p}^{\mu-d}} (\widebar z) & = |z|^{-\nu} (z/|z|)^{-p}  + O \lp \big||z|^{1- \nu  } \big| \rp,
\end{align*}
as $|z| \ra 0$. Since $|\Re (\nu)| <   1 / 2$, the two error terms have order strictly lower than the least order of the two main terms. This forces  $a_{+} = c_{+}$ and $a_{-} = c_{-}$ in order for $f (z)$ to have the prescribed asymptotic in \eqref{2eq: asymptoic f (z)}.

It is now proven that the identity \eqref{2eq: identity f (z)} is valid except on the branch cut $(1, \infty)$. However, $f (z)$ is continuous on $\BC \smallsetminus \{0, 1\}$, the identity may be extended to the branch cut.
\end{proof}

Moreover, it is of interest to note that  the ratio $c_{+} /c_{-}$ has to be unique, since neither ${_{\phantom{\nu}}^{\rho} \hskip -2pt f_{\nu+p}^{\mu+d}} (z) \, {_{\phantom{\nu}}^{\rho} \hskip -2pt f_{\nu-p}^{\mu-d}} (\widebar z) $ nor $ {_{\phantom{\nu}}^{\rho} \hskip -2pt f_{-\nu-p}^{\mu+d}} (z) \, {_{\phantom{\nu}}^{\rho} \hskip -2pt f_{-\nu+p}^{\mu-d}} (\widebar z)   $ is single-valued on $\BA (1, \infty)$. Actually, by the continuity    of   $ {_{\phantom{\nu}}^{\rho} \hskip -1pt \boldF_{\nu, \shskip p}^{\mu, \shskip  d}} (z) $, the ratio must be ${_{\phantom{\nu}}^{\rho} \hskip -1.5pt {C}_{\nu, \shskip p}^{\mu, \shskip  d}}/ \, {_{\phantom{\nu}}^{\rho} \hskip -1.5pt {C}_{-\nu, \shskip - p}^{\mu, \shskip  d}}$ (see \eqref{1eq: hypergeometric, 3}, \eqref{1eq: hyper power, 3}, and \eqref{2eq: asymp F sim P}). Thus it makes sense to say that ${_{\phantom{\nu}}^{\rho} \hskip -1pt \boldF_{\nu, \shskip p}^{\mu, \shskip  d}} (z) $ is {\it the} hypergeometric function over $\BC  $.

\section{Proof of Theorem \ref{thm: Hankel}: The Hankel--Mellin Transform}

For $u \in \BC \smallsetminus \{0\}$, consider 
\begin{align}
{_{\phantom{\nu}}^{\rho} \hskip -1.5pt \boldsymbol{\varPhi}_{\nu, \shskip p}^{\mu, \shskip  d}} (u) =	\viint   \boldJ_{ \mu,\shskip  d} (z)  \boldJ_{ \nu,\shskip  p}  (uz)   |z|^{2 \rho   - 2}     i d      z \vwedge d      \widebar{z}, 
\end{align} 
and
\begin{align}\label{4eq: defn of Pi}
{_{\phantom{\nu}}^{\rho} \hskip -1.5pt \boldsymbol{\varPi}_{\nu, \shskip p}^{\mu, \shskip  d}} (u) =	\viint 	\boldJ_{ \mu,\shskip  d} (z)   \boldsymbol{P}_{ \nu,\shskip  p}  (uz)   |z|^{2 \rho   - 2}     i d      z \vwedge d      \widebar{z} . 
\end{align} 
Let us focus  on $ u \in \BA (0, 1)  $ (see Definition \ref{defn: annulus}),  as our main concern in \S \ref{sec: asymptotic} is the asymptotic of ${_{\phantom{\nu}}^{\rho} \hskip -1.5pt \boldsymbol{\varPhi}_{\nu, \shskip p}^{\mu, \shskip  d}} (u)$ near  the origin. By symmetry, we may also consider $ u \in \BA (1, \infty)  $. 

First of all, the convergence of  ${_{\phantom{\nu}}^{\rho} \hskip -1.5pt \boldsymbol{\varPhi}_{\nu, \shskip p}^{\mu, \shskip  d}} (u)$ and   ${_{\phantom{\nu}}^{\rho} \hskip -1.5pt \boldsymbol{\varPi}_{\nu, \shskip p}^{\mu, \shskip  d}} (u)$ may be easily examined with the help of Lemma \ref{lem: integration by parts}, along with \eqref{2eq: bound for J mu m, weak, 1} and \eqref{1eq: J = W + W}--\eqref{2eq: bound 1/z}.

	\begin{lem}\label{lem: convergence of F and G}
	Let  $\nu$ and $\mu$ be non-integral. 
	
	{\rm(1)}  ${_{\phantom{\nu}}^{\rho} \hskip -1.5pt \boldsymbol{\varPhi}_{\nu, \shskip p}^{\mu, \shskip  d}} (u)$ is absolutely convergent if $|\Re (\nu) | + |\Re (\mu)| <   \Re (\rho) < 1$, and, for $u  \in \BA (0, 1)$, convergent if $|\Re (\nu) | + |\Re (\mu)| <   \Re (\rho) <   3 / 2$. 
	
	{\rm(2)} ${_{\phantom{\nu}}^{\rho} \hskip -1.5pt \boldsymbol{\varPi}_{\nu, \shskip p}^{\mu, \shskip  d}} (u)$ is convergent for $|\Re (\nu) | + |\Re (\mu)|  <   \Re (\rho) < 1 - |\Re (\nu) | $. 
	
	\noindent		Moreover, if the convergence is ensured, then ${_{\phantom{\nu}}^{\rho} \hskip -1.5pt \boldsymbol{\varPhi}_{\nu, \shskip p}^{\mu, \shskip  d}} (u)$ and  ${_{\phantom{\nu}}^{\rho} \hskip -1.5pt \boldsymbol{\varPi}_{\nu, \shskip p}^{\mu, \shskip  d}} (u)$ are continuous functions of $u $. 
\end{lem}

\begin{proof}
	We partition the integrals ${_{\phantom{\nu}}^{\rho} \hskip -1.5pt \boldsymbol{\varPhi}_{\nu, \shskip p}^{\mu, \shskip  d}} (u)$ and  ${_{\phantom{\nu}}^{\rho} \hskip -1.5pt \boldsymbol{\varPi}_{\nu, \shskip p}^{\mu, \shskip  d}} (u)$ according to a partition of unity $(1- \txw(z)) + \txw(z) \equiv 1$  for $ \BA (  0, 2 ] \cup \BA [1, \infty)$,  say. 
	
	By \eqref{2eq: bound for J mu m, weak, 1}, it is obvious that the integral
	\begin{align*}
	\viint_{\BA (0,2] }    (1- \txw(z)) \boldJ_{ \mu,\shskip  d} (z)  \boldJ_{ \nu,\shskip  p}  (uz)   |z|^{2 \rho   - 2}     i d      z \vwedge d      \widebar{z} 
	\end{align*}
	is (absolutely) convergent if $ |\Re (\nu) | + |\Re (\mu)| <   \Re (\rho) $. Similarly, the same statement is true if $ \boldJ_{ \nu,\shskip  p } (uz)  $ is substituted by $\boldsymbol{P}_{ \nu,\shskip  p } (uz) $.
	
	Next, we consider the  integral
	\begin{align*}
		\viint_{\BA [1, \infty) }      \txw(z) \boldJ_{ \mu,\shskip  d} (z)  \boldJ_{ \nu,\shskip  p}  (uz)   |z|^{2 \rho   - 2}     i d      z \vwedge d      \widebar{z} . 
	\end{align*}
By \eqref{2eq: bound 1/z}, it is absolutely convergent if  $\Re (\rho) < 1$. Further, according to \eqref{1eq: J = W + W}, we may divide it into nine integrals. Four are of the same type,  containing $ \boldsymbol{W}_{\mu, \shskip d} ( \pm  z)   \boldsymbol{W}_{\nu, \shskip p} ( \pm u z) $, and their convergence for $  \Re (\rho) <   3 / 2 $ follows from an application of Lemma \ref{lem: integration by parts} with $ a = \pm 4 u$, $b = \pm 4 $, $f (z) = \txw(z)  \boldsymbol{W}_{\mu, \shskip d} ( \pm z)   \boldsymbol{W}_{\nu, \shskip p} ( \pm u z) |z|^{2 \rho   - 2} $, and $\gamma = \Re (\rho)- 1$ (by \eqref{1eq: derivatives of W}).  By \eqref{1eq: derivatives of W} and \eqref{1eq: derivatives of E}, the rest five integrals, containing $  \boldsymbol{E}_{ \mu,\shskip  d}^1   (z) $, $\boldsymbol{E}_{ \nu,\shskip  p}^1   (u z)$ or both,  are absolutely convergent when $  \Re (\rho) <   3 / 2 $. 
Moreover, similar arguments apply if $ \boldJ_{ \nu,\shskip  p } (uz)  $ is substituted by $\boldsymbol{P}_{ \nu,\shskip  p } (uz) $. Note that the application of Lemma \ref{lem: integration by parts} now uses  $ a = 0$, $b = \pm 4 $, $f (z) = \txw(z)  \boldsymbol{W}_{\mu, \shskip d} ( \pm z) \boldsymbol{P}_{\nu, \shskip p}   ( u z) |z|^{2 \rho   - 2}$, and $\gamma = \Re (\rho) + |\Re (\nu)| -   1 / 2 $. 
	\end{proof}

\subsection{Asymptotic of \protect ${_{\protect\phantom{\nu}}^{\rho} \hskip -1.5pt \boldsymbol{\varPhi}_{\nu, \shskip p}^{\mu, \shskip  d}} (u)$} \label{sec: asymptotic}

For brevity, set $\vlambda   = |\Re (\nu)|$, $\kappa = |\Re (\mu)|$, and $\beta = \Re (\rho)$. In this section, we assume 
\begin{align}\label{4eq: assump}
	\vlambda < \frac 1 4, \qquad \kappa < \frac 1 2, \qquad \vlambda + \kappa  < \beta < 1 - \vlambda. 
\end{align} Now  ${_{\phantom{\nu}}^{\rho} \hskip -1.5pt \boldsymbol{\varPhi}_{\nu, \shskip p}^{\mu, \shskip  d}} (u)$ and  ${_{\phantom{\nu}}^{\rho} \hskip -1.5pt \boldsymbol{\varPi}_{\nu, \shskip p}^{\mu, \shskip  d}} (u)$ are   both convergent.  Let  $\nu$ and $ \mu $ be non-zero. 

The first ingredient of our proof is  the following asymptotic: 
\begin{align}\label{4eq: asympt}
	{_{\phantom{\nu}}^{\rho} \hskip -1.5pt \boldsymbol{\varPhi}_{\nu, \shskip p}^{\mu, \shskip  d}} (u) \sim  \frac {2} {(2\pi)^{2\rho+ 2}}  {_{\phantom{\nu}}^{\rho} \hskip -1.5pt \boldsymbol{P}_{\nu, \shskip p}^{\mu, \shskip  d}} (u^2)   , \qquad u \ra 0.
\end{align}

\begin{lem} \label{lem: Phi sim Pi} We have 
	\begin{align}
	{_{\phantom{\nu}}^{\rho} \hskip -1.5pt \boldsymbol{\varPhi}_{\nu, \shskip p}^{\mu, \shskip  d}} (u) = 	{_{\phantom{\nu}}^{\rho} \hskip -1.5pt \boldsymbol{\varPi}_{\nu, \shskip p}^{\mu, \shskip  d}} (u) + o  (|u|^{2\vlambda  }  ) , \qquad u \ra 0. 
	\end{align}
\end{lem}

\begin{lem}\label{lem: Pi = P}
	We have  
	\begin{align}\label{4eq: Pi = P}
		{_{\phantom{\nu}}^{\rho} \hskip -1.5pt \boldsymbol{\varPi}_{\nu, \shskip p}^{\mu, \shskip  d}} (u) = \frac {2} {(2\pi)^{2\rho+ 2}}  {_{\phantom{\nu}}^{\rho} \hskip -1.5pt \boldsymbol{P}_{\nu, \shskip p}^{\mu, \shskip  d}} (u^2) . 
	\end{align}
\end{lem}

Lemmas \ref{lem: Phi sim Pi} and \ref{lem: Pi = P} clearly imply the asymptotic in \eqref{4eq: asympt}.

\begin{proof}[Proof of Lemma \ref{lem: Phi sim Pi}] Let $|u| < 1$ be sufficiently small. All the implied constants in our computation will only depend on  $\vlambda $, $\kappa$, $p$,  $ d $, and  $\beta$. 
	
	We split ${_{\phantom{\nu}}^{\rho} \hskip -1.5pt \boldsymbol{\varPhi}_{\nu, \shskip p}^{\mu, \shskip  d}} (u) - 	{_{\phantom{\nu}}^{\rho} \hskip -1.5pt \boldsymbol{\varPi}_{\nu, \shskip p}^{\mu, \shskip  d}} (u)  $ as the sum
	\begin{align*}
	{_{\phantom{\nu}}^{\rho} \hskip -1.5pt \boldsymbol{\varPhi}_{\nu, \shskip p}^{\mu, \shskip  d}} (u) - &\,	{_{\phantom{\nu}}^{\rho} \hskip -1.5pt \boldsymbol{\varPi}_{\nu, \shskip p}^{\mu, \shskip  d}} (u)    =  {_{\phantom{\nu}}^{\rho} \hskip -1.5pt \boldsymbol{D}_{\nu, \shskip p}^{\mu, \shskip  d}}  (u)  + {_{\phantom{\nu}}^{\rho} \hskip -1.5pt \acute{\boldsymbol{E}}_{\nu, \shskip p}^{\mu, \shskip  d}}   (u)  + {_{\phantom{\nu}}^{\rho} \hskip -1.5pt \grave{\boldsymbol{E}}_{\nu, \shskip p}^{\mu, \shskip  d}}   (u)  - {_{\phantom{\nu}}^{\rho} \hskip -1.5pt \hat{\boldsymbol{E}}_{\nu, \shskip p}^{\mu, \shskip  d}}   (u)  \\
	= &\,  \viint_{\BA(0, \, 2  |u|^{- \frac 1 2 }  ] }      \txw_{0}(|z|) \boldJ_{ \mu,\shskip  d} (z) \big(  \boldJ_{ \nu,\shskip  p}  (uz) - \boldsymbol{P}_{ \nu,\shskip  p}  (uz)  \big)    |z|^{2 \rho   - 2}     i d      z \vwedge d      \widebar{z} \\
	& + \viint_{\BA[|u|^{- \frac 1 2 }  , \, 2 |u|^{-1}  ] }      \txw_{\natural} (|z|) \boldJ_{ \mu,\shskip  d} (z)    \boldJ_{ \nu,\shskip  p}  (uz)     |z|^{2 \rho   - 2}     i d      z \vwedge d      \widebar{z} \\
	& + \viint_{\BA[ |u|^{-1}, \, \infty) }      \txw_{\infty} (|z|) \boldJ_{ \mu,\shskip  d} (z)    \boldJ_{ \nu,\shskip  p}  (uz)     |z|^{2 \rho   - 2}     i d      z \vwedge d      \widebar{z} \\
	& - \viint_{\BA[ |u|^{-\frac 1 2}, \, \infty) }      (\txw_{\natural} (|z|) + \txw_{\infty} (|z|) )\boldJ_{ \mu,\shskip  d} (z)    \boldsymbol{P}_{ \nu,\shskip  p}  (uz)     |z|^{2 \rho   - 2}     i d      z \vwedge d      \widebar{z} , 
	\end{align*}
	where  $ \txw_{0} (|z|) + \txw_{\natural} (|z|) + \txw_{\infty} (|z|) \equiv 1 $
	is a partition of unity such that $\txw_{0} (|z|) $, $\txw_{\natural} (|z|)  $, and $\txw_{\infty} (|z|) $ are smooth functions supported on $\BA  (  0, 2  |u|^{- \frac 1 2 }  ]$,       $ \BA  [ |u|^{- \frac 1 2 } , 2  |u|^{-1} ] $ and $ \BA  [  |u|^{- 1 } , \infty )$, respectively, and that $x^{r} \txw_{0}^{(r)} (x), x^r \txw_{\natural}^{(r)}  (x), x^r \txw_{\infty}^{(r)}   (x) \Lt_{\, r} \hskip -1pt 1 $.   
	
	It follows from  \eqref{2eq: bound for J mu m}, \eqref{2eq: bound for J mu m, weak, 1}, and \eqref{2eq: bound 1/z} (also by $\kappa <   1 / 2$) that   
 \begin{align*}
 {_{\phantom{\nu}}^{\rho} \hskip -1.5pt \boldsymbol{D}_{\nu, \shskip p}^{\mu, \shskip  d}}  (u) \Lt 	\viint_{\BA(0, \, 2  |u|^{- \frac 1 2 } ] }   |z|^{-1}  |uz|^{2-2\vlambda} |z|^{2 \beta   - 2}    | d      z \vwedge d      \widebar{z} | \Lt |u|^{\frac 3 2 - \vlambda - \beta}   . 
 \end{align*}

The $\boldsymbol{E}$-integrals $  {_{\phantom{\nu}}^{\rho} \hskip -1.5pt \acute{\boldsymbol{E}}_{\nu, \shskip p}^{\mu, \shskip  d}}   (u)$, $ {_{\phantom{\nu}}^{\rho} \hskip -1.5pt \grave{\boldsymbol{E}}_{\nu, \shskip p}^{\mu, \shskip  d}}   (u)$, and $ {_{\phantom{\nu}}^{\rho} \hskip -1.5pt \hat{\boldsymbol{E}}_{\nu, \shskip p}^{\mu, \shskip  d}}   (u)  $ are all oscillatory. It may be proven by choosing $ N$ and $M$  large in \eqref{1eq: J = W + W} and Lemma \ref{lem: integration by parts} that these are all negligibly small. 
However, $N = M = 1$ is sufficient for the estimates: 
\begin{align*}
	 {_{\phantom{\nu}}^{\rho} \hskip -1.5pt \acute{\boldsymbol{E}}_{\nu, \shskip p}^{\mu, \shskip  d}}   (u),   {_{\phantom{\nu}}^{\rho} \hskip -1.5pt \hat{\boldsymbol{E}}_{\nu, \shskip p}^{\mu, \shskip  d}}   (u)  = O ( |u|^{\frac 3 2 - \vlambda - \beta}) , \qquad {_{\phantom{\nu}}^{\rho} \hskip -1.5pt \grave{\boldsymbol{E}}_{\nu, \shskip p}^{\mu, \shskip  d}}   (u) = O (|u|^{2-2\beta}).
\end{align*}

According to \eqref{1eq: J = W + W}, we divide ${_{\phantom{\nu}}^{\rho} \hskip -1.5pt \acute{\boldsymbol{E}}_{\nu, \shskip p}^{\mu, \shskip  d}}   (u)$ into three integrals that contain $\boldsymbol{W}_{ \mu,\shskip  d}   (\pm z)$ or  $ \boldsymbol{E}^1_{ \mu,\shskip  d}  (z)$. For the first two integrals, we apply Lemma  \ref{lem: integration by parts} with $a = 0$, $b = \pm 4$, $c = |u|^{-\frac 1 2 }$, $M = 1$,  $f (z) = \txw_{\natural} (|z|)  \boldsymbol{W}_{ \mu,\shskip  d}   (\pm z)  \boldJ_{ \nu,\shskip  p}  (uz)     |z|^{2 \rho   - 2}$, $S = |u|^{-2\vlambda}$, and $\gamma = \beta - \vlambda -   1 / 2$ (by \eqref{2eq: bound for J mu m, weak, 2} and \eqref{1eq: derivatives of W}), and it follows that these two integrals are both $ O  ( |u|^{\frac 3 2 - \vlambda - \beta})$. For the third integral containing $ \boldsymbol{E}^1_{ \mu,\shskip  d}  (z)$, trivial estimation by \eqref{1eq: derivatives of E} yields the same bound  $ O  ( |u|^{\frac 3 2 - \vlambda - \beta})$.  In the same way, we may prove  $ {_{\phantom{\nu}}^{\rho} \hskip -1.5pt \hat{\boldsymbol{E}}_{\nu, \shskip p}^{\mu, \shskip  d}}   (u) = O  ( |u|^{\frac 3 2 - \vlambda - \beta})$. 

By  \eqref{1eq: J = W + W}, we divide ${_{\phantom{\nu}}^{\rho} \hskip -1.5pt \grave{\boldsymbol{E}}_{\nu, \shskip p}^{\mu, \shskip  d}}   (u) $ into seven integrals that contain $  \boldsymbol{W}_{\mu, \shskip d} ( \pm  z)   \boldsymbol{W}_{\nu, \shskip p} ( \pm u z)$, $\boldsymbol{W}_{\mu, \shskip d} ( \pm  z)   \boldsymbol{E}_{\nu, \shskip p}^1 (   u z) $, and $\boldsymbol{E}^1_{\mu, \shskip d} (    z)   \boldJ_{\nu, \shskip p} (   u z)$. For the first four integrals, we apply Lemma  \ref{lem: integration by parts} with $a = \pm 4 u$, $b = \pm 4$, $c = |u|^{- 1 }$, $M = 1$,  $f (z) = \txw_{\natural} (|z|)  \boldsymbol{W}_{ \mu,\shskip  d}   (\pm z)  \boldsymbol{W}_{\nu, \shskip p} ( \pm u z)   |z|^{2 \rho   - 2}$, $S = |u|^{-1}$, and $\gamma = \beta - 1$ (by \eqref{1eq: derivatives of W}). For the second two integrals, we apply Lemma  \ref{lem: integration by parts} with $a = 0$, $b = \pm 4$, $c = |u|^{- 1 }$, $M = 1$,  $f (z) = \txw_{\natural} (|z|)  \boldsymbol{W}_{ \mu,\shskip  d}   (\pm z)  \boldsymbol{E}^1_{\nu, \shskip p} ( u z)   |z|^{2 \rho   - 2}$, $S = |u|^{-2}$, and $\gamma = \beta -   1 / 2 $ (by \eqref{1eq: derivatives of W} and \eqref{1eq: derivatives of E}). In both cases we have $O (|u|^{3-2\beta})$. For the last integral, trivial estimation by  \eqref{1eq: derivatives of E} and \eqref{2eq: bound 1/z} yields $ O (|u|^{2-2\beta}) $. 

It is clear that the bounds $  O ( |u|^{\frac 3 2 - \vlambda - \beta})$ and $  O (|u|^{2-2\beta}) $ obtained above are  indeed $o (|u|^{2\vlambda})$ by the assumptions in  \eqref{4eq: assump}. 
\end{proof}

\begin{proof}[Proof of Lemma \ref{lem: Pi = P}] 
	The identity \eqref{4eq: Pi = P} is a direct consequence of Lemma \ref{lem: Mellin}, along with the definitions \eqref{0def: P mu m (z)}, \eqref{0eq: defn of P},  \eqref{3eq: hyper power, C, 1}--\eqref{1eq: hyper power}, and \eqref{4eq: defn of Pi}.  
\end{proof}

\subsection{Differential equations for ${_{\protect\phantom{\nu}}^{\rho} \hskip -1.5pt \boldsymbol{\varPhi}_{\nu, \shskip p}^{\nu, \shskip  d}} (\hskip -1pt \sqrt u)$} \label{sec: diff equations}

Next, we need to verify that ${_{\phantom{\nu}}^{\rho} \hskip -1.5pt \boldsymbol{\varPhi}_{\nu, \shskip p}^{\nu, \shskip  d}} (\hskip -1pt \sqrt u)$ satisfies the (hypergeometric) differential equations 
\begin{align}\label{4eq: hyper eq} 
	{_{\phantom{\nu}}^{\rho} \hskip -0.5pt \nabla_{\nu + p }^{\mu + d}} \big( 
	{_{\phantom{\nu}}^{\rho} \hskip -1.5pt \boldsymbol{\varPhi}_{\nu, \shskip p}^{\nu, \shskip  d}} (\hskip -1pt \sqrt u) \big) = 0, \qquad {_{\phantom{\nu}}^{\rho} \hskip -0.5pt \overline \nabla{}_{\nu - p }^{\mu - d}} \big(
	 {_{\phantom{\nu}}^{\rho} \hskip -1.5pt \boldsymbol{\varPhi}_{\nu, \shskip p}^{\nu, \shskip  d}} (\hskip -1pt \sqrt u) \big) = 0,  
\end{align}
with differential operators ${_{\phantom{\nu}}^{\rho} \hskip -0.5pt \nabla_{\nu }^{\mu}} $ and ${_{\phantom{\nu}}^{\rho} \hskip -0.5pt \overline{\nabla}{}_{\nu }^{\mu}}$ defined as in \eqref{2eq: nabla, 1} and   \eqref{2eq: nabla, 2}. By symmetry, we only need to verify the former, which, if we set $\valpha = \nu + p$ and $\gamma = \mu + d$ for simplify, may be explicitly written as
	\begin{align*}
		   (1-u) \nabla^2  \big( {_{\phantom{\nu}}^{\rho} \hskip -1.5pt \boldsymbol{\varPhi}_{\nu, \shskip p}^{\nu, \shskip  d}} (\hskip -1pt \sqrt u) \big)   -\rho  u \nabla   \big( {_{\phantom{\nu}}^{\rho} \hskip -1.5pt \boldsymbol{\varPhi}_{\nu, \shskip p}^{\nu, \shskip  d}} (\hskip -1pt \sqrt u) \big)  - \bigg(   \frac {  \rho^2  - \gamma^2} {4} u + \frac {\valpha^2} {4  } \bigg) {_{\phantom{\nu}}^{\rho} \hskip -1.5pt \boldsymbol{\varPhi}_{\nu, \shskip p}^{\nu, \shskip  d}} (\hskip -1pt \sqrt u) = 0 . 
	\end{align*}  
	As $  \boldJ_{ \nu,\shskip  p}  (z)$ and  $\boldJ_{ \mu,\shskip  d} (z)$ are even functions,  it is more convenient to write 
	\begin{align*}
		{_{\phantom{\nu}}^{\rho} \hskip -1.5pt \boldsymbol{\varPhi}_{\nu, \shskip p}^{\mu, \shskip  d}} (\hskip -1pt \sqrt u) = \frac 1 2 	\viint   \boldJ_{ \mu,\shskip  d} (\hskip -1pt \sqrt{z})  \boldJ_{ \nu,\shskip  p}  (\hskip -1pt \sqrt{uz})    |z|^{ \rho   - 2}     i d      z \vwedge d      \widebar{z} ,
	\end{align*}
so that, according to Appendix \ref{append: Hankel}, we have
\begin{align*}
	{_{\phantom{\nu}}^{\rho} \hskip -1.5pt \boldsymbol{\varPhi}_{\nu, \shskip p}^{\mu, \shskip  d}} (\hskip -1pt \sqrt u) = \frac 1 {4\pi^2}  \EuH_{\nu, \shskip p}   \big(\boldJ_{ \mu,\shskip  d} (\hskip -1pt \sqrt{z}) |z|^{ \rho   - 2}    \big) (u). 
\end{align*}
Of course, we need to assume $|\Re (\nu) | + |\Re (\mu)| <   \Re (\rho) < 1$ to ensure convergence. 

For $r = 0, 1, 2$, define 
\begin{align*}
\boldsymbol{I}_{r} (z) =	{_{\phantom{\nu}}^{\rho} \hskip -1.5pt \boldsymbol{I}_{r}^{\mu, \shskip  d}} (z) = \frac 1 {4\pi^2} \nabla^r \big( \boldJ_{ \mu,\shskip  d} (\hskip -1pt \sqrt{z} ) \big) |z|^{\rho - 1} , 
\end{align*}
regarded as distributions in the space  $\mathscr{T}_{\mathrm{sis}}^{\nu, \shskip p} (\BC^{\times})$ as in  Appendix \ref{append: Hankel}. 
Note that $ {_{\phantom{\nu}}^{\rho} \hskip -1.5pt \boldsymbol{\varPhi}_{\nu, \shskip p}^{\mu, \shskip  d}} (\hskip -1pt \sqrt u) = \EuH_{\nu, \shskip p} {_{\phantom{\nu}}^{\rho} \hskip -1.5pt \boldsymbol{I}_{0}^{\mu, \shskip  d}} (u) $. 
It is readily seen that  
\begin{align*}
	\nabla \boldsymbol{I}_r = \boldsymbol{I}_{r+1}   + \Big(\frac {\rho} 2 - 1 \Big) \boldsymbol{I}_r  ,  \qquad  \nabla^2 \boldsymbol{I}_0 = \boldsymbol{I}_2   +     (  {\rho}   - 2  ) \boldsymbol{I}_1  + \Big(\frac {\rho} 2 - 1 \Big)^2 \boldsymbol{I}_0  ,
\end{align*}
and hence 
\begin{align*}
	\nabla \boldsymbol{I}_0 + \boldsymbol{I}_0   = \boldsymbol{I}_1   +  \frac {\rho} 2   \boldsymbol{I}_0, \qquad 
	\nabla^2  \boldsymbol{I}_0   + 2 \nabla  \boldsymbol{I}_0   + \boldsymbol{I}_0     =    \boldsymbol{I}_2 +   {\rho}     \boldsymbol{I}_1  +   \frac {\rho^2} 4 \boldsymbol{I}_0 . 
\end{align*}
It follows from \eqref{2eq: nabla}--\eqref{2eq: nabla J = 0} that
\begin{align*}
	\boldsymbol{I}_2 (z) + 4\pi^2 z \boldsymbol{I}_0 (z) - \frac {\gamma^2} 4 \boldsymbol{I}_0 (z) = 0. 
\end{align*}
Therefore, by \eqref{append: nabla H, 1} and \eqref{append: nabla H, 21}, 
\begin{align*}
	  & \quad \ \nabla^2 \EuH_{\nu, \shskip p}   \boldsymbol{I}_0 + \rho \nabla \EuH_{\nu, \shskip p}   \boldsymbol{I}_0 + \frac {  \rho^2  - \gamma^2} {4}  \EuH_{\nu, \shskip p}   \boldsymbol{I}_0 \\
	  &=  \EuH_{\nu, \shskip p} \bigg(  (\nabla^2  \boldsymbol{I}_0   + 2 \nabla  \boldsymbol{I}_0   + \boldsymbol{I}_0 ) -	\rho ( \nabla \boldsymbol{I}_0 + \boldsymbol{I}_0) + \frac {\rho^2 -\gamma^2} {4} \boldsymbol{I}_0 \bigg) \\
	  &= \EuH_{\nu, \shskip p} \boldsymbol{I}_2   -  \frac {\gamma^2} 4 \EuH_{\nu, \shskip p} \boldsymbol{I}_0 \\
	  & = - 4\pi^2 \EuH_{\nu, \shskip p} (z  \boldsymbol{I}_0) . 
\end{align*}
Next, by \eqref{append: nabla H, 21'}, 
\begin{align*}
	\nabla^2 \EuH_{\nu, \shskip p} \boldsymbol{I}_0  - \frac {\valpha^2} 4 \EuH_{\nu, \shskip p}  \boldsymbol{I}_0   =  - 4 \pi^2 u \shskip \EuH_{\nu, \shskip p} (z  \boldsymbol{I}_0 ) . 
\end{align*}
Consequently, 
\begin{align*}
(1-u)	\nabla^2 \EuH_{\nu, \shskip p}   \boldsymbol{I}_0 - \rho u \nabla \EuH_{\nu, \shskip p}   \boldsymbol{I}_0 - \bigg( \frac {  \rho^2  - \gamma^2} {4} u + \frac {\valpha^2} 4   \bigg)   \EuH_{\nu, \shskip p}   \boldsymbol{I}_0 = 0, 
\end{align*}
as desired. 

\subsection{Conclusion} Combining \eqref{4eq: asympt} and \eqref{4eq: hyper eq}, we deduce from Lemma \ref{lem: hypergeometric equation},  under the conditions 
\begin{align}
	|\Re (\nu)| < \frac 1 4, \quad |\Re (\mu)| < \frac 1 2, \  \quad  |\Re (\nu)| + |\Re (\mu)|  \hskip -1pt < |\Re (\rho)| \hskip -1pt < 1 - |\Re (\nu)|,  
\end{align}
that
\begin{align}\label{4eq: final identity}
	{_{\phantom{\nu}}^{\rho} \hskip -1.5pt \boldsymbol{\varPhi}_{\nu, \shskip p}^{\mu, \shskip  d}} (u) =   \frac {2} {(2\pi)^{2\rho+ 2}}  {_{\phantom{\nu}}^{\rho} \hskip -1.5pt \boldF_{\nu, \shskip p}^{\mu, \shskip  d}} (u^2) . 
\end{align}
This is exactly  \eqref{1eq: main, 1} in Theorem \ref{thm: Hankel}. By Lemma \ref{lem: convergence of F and G} and the principle of analytic continuation,  the validity of \eqref{4eq: final identity} may be extended to $|\Re (\nu)| + |\Re (\mu)|    < |\Re (\rho)|  < 1$, or even $|\Re (\nu)| + |\Re (\mu)|    < |\Re (\rho)|  <   3 / 2$ for $ u \in \BA (0, 1)$ and  also $ u \in \BA (1, \infty)$ by symmetry. 

\section{Proof of Theorems \ref{thm: double Fourier} and  \ref{thm: double Fourier, 2}: The Double Fourier--Mellin Transform} 

In this section, we consider Theorems \ref{thm: double Fourier}, \ref{thm: double Fourier, 2}, and  Conjecture \ref{conj: double Fourier}. For simplicity, we assume with no lose of generality that $u = \varw$ (by $ z \ra \hskip -1.5pt \sqrt{\varw/u} \shskip z $ and $\varv \ra \hskip -1.5pt \sqrt{u/\varw}  \shskip \varv$).

\subsection{Proof of Theorem \ref{thm: double Fourier}}\label{sec: proof 2.4}

For $p$ even, we need to consider 
\begin{align*}
	\viint  \viint   \boldJ_{ \nu,\shskip  p}   (\hskip -1pt \sqrt{\varv z})  e (- 2 \mathrm{Re} ( uz ) )    |z|^{-1}  {i d      z \vwedge d      \widebar{z}} \cdot e ( - 2 \mathrm{Re} (u \varv     ) )  |\varv |^{2\gamma - 2} {i d      \varv  \vwedge d      \widebar{\varv }}. 
\end{align*}
First we use $z \ra \varv z$ and $u \ra u/\varv$ to transform \eqref{0eq: Fourier, C} into 
\begin{equation*} 
		  \viint    \boldJ_{ \nu,\shskip  p}   (\hskip -1pt \sqrt{\varv z})  e (- 2 \mathrm{Re} (u  z) )    \frac {i d      z \vwedge d      \widebar{z}} {|z|}
		=  \frac { 1 } { 2|u|} e \bigg(\hskip -1pt  \Re \bigg(\frac {\varv} {u} \bigg) \hskip -1pt  \bigg) \boldJ_{ \frac 1 2 \nu,  \frac 1 2 p} \bigg(\frac {\varv} {4 u}\bigg).
\end{equation*}
Thus  the quadruple integral becomes 
\begin{align*}
	  1 / {2|u|} \cdot	\viint   e (\Re (\varv/u)) \boldJ_{ \frac 1 2 \nu,  \frac 1 2 p} (\varv/4u) e ( - 2 \mathrm{Re} (u \varv    ) )  |\varv |^{2\gamma -2 }   {i d      \varv  \vwedge d      \widebar{\varv }},   
\end{align*}
or, by $\varv \ra 4u \varv $, 
\begin{align*}
	2 {|4u|^{2\gamma - 1}}  	\viint   \boldJ_{ \frac 1 2 \nu,  \frac 1 2 p} (\varv) e (- 4 \mathrm{Re} ((2  u^2 - 1  ) \varv   ) )  |\varv|^{2\gamma - 2} {i d      \varv \vwedge d      \widebar{\varv}}.
\end{align*}
Then by \eqref{0eq: general W-S, C} this integral is equal to
\begin{align*}
\frac {{|u |^{2\gamma - 1}} } {\pi^{2\gamma} } \boldF^{\gamma}_{\frac 1 2 \nu, \shskip \frac 1 2 p} \bigg( \hskip -1pt  \frac 1  {(1 - 2   u^2)^2} \hskip -1pt  \bigg) ,
\end{align*} 
as in the first identity of \eqref{1eq: double Fourier, 1C}, and by Lemma  \ref{lem: quad trans} (with $z = 1/   u^2$), it is further transformed   into 
\begin{align*}
	\frac { 1 } {\pi (2\pi)^{  2\gamma }  |u |^{1+2\gamma} }           {_{\phantom{\nu}}^{\frac 1 2 + \gamma } \hskip -1.5pt \boldF_{ \nu, p}^{\frac 1 2 - \gamma   }} \bigg( \hskip -1pt \frac 1 {   u ^2} \hskip -1pt  \bigg) , 
\end{align*} as in the second identity of \eqref{1eq: double Fourier, 1C}.

\subsection{Remarks on Conjecture \ref{conj: double Fourier}} Next, in a {formal} manner, we show how to transform  the integral  in \eqref{1eq: double Fourier, 2C} into the one in \eqref{1eq: main, 3} via the {formal} integral \eqref{0eq: formal integral}.  
We start with the change $\varv z \ra z$ so that the integral in \eqref{1eq: double Fourier, 2C} turns into 
\begin{align*}
	\viint     \boldJ_{ \nu,\shskip  p}   (\hskip -1pt \sqrt{ z}) \bigg( \viint  e ( - 2  \mathrm{Re}  ( uz/\varv  + u \varv) ) \frac {i d      \varv  \vwedge d      \widebar{\varv }} { |\varv |^{2 +2 \beta -2\gamma} }   \bigg)   \frac {i d      z \vwedge d      \widebar{z}} {|z|^{2-2\beta} }     . 
\end{align*}
By $  \varv \ra - \hskip -1.5pt  \sqrt{  z} \shskip \varv    $, the inner integral becomes
\begin{align*}
 |  z |^{\gamma - \beta  } 	 \viint  e (   2  \mathrm{Re}  ( u \hskip -1.5pt  \sqrt{  z}  (\varv  + 1/ \varv)) ) \shskip { |\varv |^{2 \gamma - 2\beta - 2} }  \shskip {i d      \varv  \vwedge d      \widebar{\varv }} ,
\end{align*}
which, in view of  \eqref{0eq: formal integral},   equals to $2 \pi^2 |  z  |^{\gamma - \beta} \boldJ_{ \gamma - \beta }  (u\hskip -1.5pt  \sqrt{    z} )$. Therefore, we arrive  at 
\begin{align*}
	 2 \pi^2  	\viint     \boldJ_{ \nu,\shskip  p}   (\hskip -1pt \sqrt{ z}) \boldJ_{ \gamma - \beta    }  (u \hskip -1.5pt  \sqrt{  z} ) |z|^{\gamma + \beta     - 2}  i d      z \vwedge d      \widebar{z}, 
\end{align*}
and we obtain \eqref{1eq: double Fourier, 2C} (in the case of $u = \varw$) by a simple  application of  \eqref{1eq: main, 3}. 

The arguments above are entirely {\it formal}, since the   integrals in  \eqref{1eq: double Fourier, 2C} and  \eqref{0eq: formal integral} are {\it not} absolutely convergent. 


\subsection{Proof of Theorem \ref{thm: double Fourier, 2}} 

To prove Theorem \ref{thm: double Fourier, 2} rigorously, we  reverse the line of arguments above and replace the  formal integral \eqref{0eq: formal integral} by the following convergent integral:
\begin{align}\label{5eq: integral}
	\boldJ_{ \mu } (x e^{i\phi}) = \frac {2} {\pi} \int_0^{\infty} J_0 (4\pi x |ye^{i\phi} + 1/ y e^{i\phi}|) y^{2\mu - 1} d y;
\end{align}
the integral is absolutely convergent for $|\Re (\mu)| <   1 / 4$, due to
\begin{align}\label{5eq: bound for J0}
	J_0 (x) \Lt \min \bigg\{ 1,  \frac 1  {\sqrt{x}}  \bigg\}. 
\end{align}  

Let us first show the deduction of \eqref{0eq: formal integral} from \eqref{5eq: integral} via 
\begin{align}\label{5eq: Bessel integral for J0}
	J_0 (2\pi |z|) = \frac 1 {2\pi} \int_0^{2\pi} e (\Re (z e^{i\omega}))  d \omega , 
\end{align}
which is a simple consequence of the Bessel integral representation (see \cite[\S 2.2]{Watson})
\begin{align*}
	J_0 (x) = \frac 1 {2\pi} \int_0^{2\pi} e^{i x \cos \omega} d \omega. 
\end{align*} 
Indeed, by applying \eqref{5eq: Bessel integral for J0} to \eqref{5eq: integral}, we obtain 
\begin{align*} 
	\boldJ_{ \mu } (x e^{i\phi}) & = \frac {1} {\pi^2 } \int_0^{\infty} \hskip -2pt \int_0^{2\pi} e (2 \Re (x e^{i\omega} ( ye^{i\phi} + 1/ y e^{i\phi} ) )) d \omega \,  y^{2\mu - 1}  d y \\
	& = \frac {1} {\pi^2 } \int_0^{\infty} \hskip -2pt \int_0^{2\pi} e (2 (xy \cos (\phi+\omega)  + x/y \cdot \cos (\phi - \omega) )  ) d \omega \,  y^{2\mu - 1}  d y . 
\end{align*}
Thus, in the Cartesian coordinates,  
\begin{align}
	\boldJ_{ \mu }  (z)  = \frac 1 {2\pi^2}   \viint e (2 \Re (z \varv + z / \varv)) \shskip |\varv|^{2\mu - 2} i d      \varv  \vwedge d      \widebar{\varv } ,
\end{align}
if we let $z = x e^{i\phi}$ and $\varv = y e^{i\omega}$. Keep in mind that it should really be interpreted as an iterated integral in the polar coordinates in the order $d \omega \, d y$.  

Now we return to the proof of Theorem \ref{thm: double Fourier, 2}.  In the polar coordinates, we start with the integral
\begin{align*}
	2 \int_0^{\infty} \hskip -2pt \int_0^{2\pi} \boldJ_{ \nu,\shskip  p} (x e^{i\phi})   \boldJ_{ \mu }  ( r x  e^{i (\theta + \phi )})   d \phi \, x^{2\rho-1} d x. 
\end{align*}
This is the integral in \eqref{1eq: main, 2} with $z = x e^{i\phi}$ and $u = r e^{i\theta}$.  By inserting \eqref{5eq: integral}, we obtain the triple integral
 \begin{align*}
 	\frac {4} {\pi} \int_0^{\infty} \hskip -2pt \int_0^{2\pi} \hskip -2pt \int_0^{\infty}  \boldJ_{ \nu,\shskip  p} \big(x e^{i\phi} \big)   J_0 \big(4\pi r x \big|ye^{i(\theta+\phi)} + 1/ y e^{i(\theta+\phi)} \big| \big) y^{2\mu-1} d y \,    d \phi \,  x^{2\rho-1} d x. 
 \end{align*}
On the assumptions
\begin{align}\label{5eq: assump}
	 |\Re (\nu) |  + |\Re (\mu) |  <   \Re (\rho) < \frac 12, \qquad |\Re (\mu)| <  \frac 1 4 ,
\end{align}  one may use the bounds in  \eqref{2eq: bound for J mu m, weak, 1}, \eqref{2eq: bound 1/z},  and \eqref{5eq: bound for J0}  to verify the absolute convergence. Next, we make the change $y \ra y /x$,  move the $ y$-integral out,   in the fashion above, apply 
\begin{align*}
	J_0 \big(4\pi r  \big|ye^{i(\theta+\phi)} + x^2 / y e^{i(\theta+\phi)} \big| \big) = \frac 1 {2\pi} \int_0^{2\pi} e \big(2\Re \big(ry \cdot e^{i(\omega +\theta+\phi )} + r x^2 / y \cdot e^{i(\omega -\theta -\phi )}  \big) \big)  d \omega , 
\end{align*} by \eqref{5eq: Bessel integral for J0}, and make the change $\omega \ra \omega - \phi$. It is more succinct to display the resulting integral in the Cartesian coordinates. Then, assuming \eqref{5eq: assump}, we have proven
\begin{align*}
	\viint \hskip -2pt  \boldJ_{ \nu,\shskip  p} (z)  \boldJ_{ \mu }  (uz)   |z|^{2 \rho   - 2}     i d      z \vwedge d      \widebar{z} \hskip -1pt = \hskip -1pt \frac 1 {2\pi^2} \hskip -2pt \viint \hskip -3pt \viint  \hskip -2pt \boldJ_{ \nu,\shskip  p}   ( {  z})    e (  2  \mathrm{Re}  ( u  (z^2/\varv \hskip -1pt +  \hskip -1pt \varv )  \hskip -1pt)  \hskip -1pt ) \frac {i d      z \vwedge d      \widebar{z}}    {|z|^{2-2\rho+2\mu} } \frac {i d      \varv  \vwedge d      \widebar{\varv }} {|\varv|^{2- 2\mu}}     ,
	\end{align*}
in which the quadruple integral is integrated in the order $ d \phi \, d \omega \,  d x\, d y$ (no longer absolutely convergent!). Finally, by $z \ra \sqrt{\varv z}$ (indeed, $x \ra \hskip -1.5pt \sqrt{xy}$ and $\phi \ra (\phi+\omega)/ 2$, thanks to the order of integration), we obtain 
\begin{align*}
	\viint \hskip -1pt  \boldJ_{ \nu,\shskip  p} (z)  \boldJ_{ \mu }  (uz)   |z|^{2 \rho   - 2}     i d      z \vwedge d      \widebar{z} \hskip -1pt = \hskip -1pt \frac 1 {4\pi^2} \hskip -2pt \viint \hskip -3pt \viint  \hskip -1pt \boldJ_{ \nu,\shskip  p}   ( \hskip -1pt \sqrt{\varv z})    e (  2  \mathrm{Re}  ( u (z \hskip -0.5pt +  \hskip -0.5pt \varv)  \hskip -0.5pt)  \hskip -0.5pt ) \frac {i d      z \vwedge d      \widebar{z}}    {|z|^{2-\rho+\mu} }   \frac {i d      \varv  \vwedge d      \widebar{\varv }} {|\varv|^{2-\rho-\mu}}    ,
\end{align*}
and hence \eqref{1eq: double Fourier, 2C}  becomes a direct consequence of \eqref{1eq: main, 2}.

 \begin{appendices}

\section{Double Fourier--Mellin Transform of Bessel Functions} \label{sec: Fourier-Mellin, R}

To prove the double Bessel integral formulae in \S \ref{sec: double integral}, we need to introduce Kummer's  confluent hypergeometric functions $ \Phi (a, c; z) $ and $\Psi (a, c; z)$. For the reference, we shall use Chapter VI of \cite{Erdelyi-HTF-1}.\footnote{Note that $  \Phi (a, c; z) = {_1F_1} (a;c;z)$. The notations $ M (a, c, z)  $ and $U (a, c, z)$ are also widely used in the literature.}

Define 
\begin{align}
	\Phi (a, c; z) = \sum_{n=0}^{\infty} \frac {(a)_n z^n } {(c)_n n! } = \frac {\Gamma (c)} {\Gamma (a)} \sum_{n=0}^{\infty} \frac {\Gamma (a+n) z^n } { \Gamma (c+n) n! } . 
\end{align}
According to \cite[6.5 (7)]{Erdelyi-HTF-1},  we have 
\begin{align}\label{Aeq: Psi}
	\Psi (a, c; z) = \frac {\Gamma (1-c)} {\Gamma (a-c+1)} \Phi (a, c; z) + \frac {\Gamma (c-1)} {\Gamma (a)} z^{1-c} \Phi (1+a-c, 2-c; z),
\end{align}
or its limit form if $c$ is an integer. 

For $y$ real and $|y| $ large, it follows from \cite[6.13.1 (1), (2)]{Erdelyi-HTF-1} the asymptotic formulae:
\begin{align}\label{Aeq: asymp Phi (iy)}
	\Phi (a, c;   i y) = \frac {\Gamma (c)  } {\Gamma (c-a)}  (- i y)^{-a} (1 + O (1/|y|)) + \frac {\Gamma (c)  } {\Gamma (a)}  e^{  i y} (i y)^{a-c} (1 + O (1/|y|)), 
\end{align} 
\begin{align}\label{Aeq: asymp Psi (iy)}
	\Psi (a, c;   i y) = (iy)^{-a} (1 + O (1/|y|)), 
\end{align}
where it is understood that $\arg (\pm i y) = \pm \pi / 2$ or $\mp \pi /2$ according as $y > 0$ or $y < 0$. 

By \cite[6.10 (8)]{Erdelyi-HTF-1}, for $\Re (z) > 0$, we have 
\begin{align}\label{Aeq: J}
	\int_0^{\infty} J_{\nu} (2 \sqrt{x}) \exp (-   x z) x^{\shskip a - \frac 1 2 \nu - 1} d      x =  \frac {\Gamma (a)} {\Gamma (1+\nu) } z^{-a} \Phi (a, 1+\nu; -1/z) , 
\end{align}
if $ \Re (a) > 0 $. 
For real $y \neq 0$, if we let $ z \ra i y $ in \eqref{Aeq: J}, 
then 
\begin{align}\label{Aeq: J, 2}
	\int_0^{\infty} J_{\nu} (2 \sqrt{x}) \exp (-  i x y) x^{\shskip a - \frac 1 2 \nu - 1} d      x =  \frac {\Gamma (a)} {\Gamma (1+\nu) } (iy)^{-a} \Phi (a, 1+\nu; i/y) , 
\end{align}
for $ 0 < \Re (a) < \Re (\nu/2) +   1 / 4$. 
One may extend its validity to $ 0 < \Re (a) < \Re (\nu/2) +   5 / 4$ by the asymptotic expansion for $J_{\nu} (x) $ (see \eqref{2eq: J and H}--\eqref{2eq: asymptotic H (2)}),  partial integration (twice!), and analytic continuation. Note that \eqref{0eq: Weber's formula} is essentially  \eqref{Aeq: J, 2} in the special case $ 2 a = \nu + 1  $ as (see \cite[6.9.1 (9)]{Erdelyi-HTF-1})
\begin{align*}
	J_{\nu} (x) = \frac {(x/2)^{\nu} \exp (-ix)} {\Gamma (1+\nu)}  \Phi \big(\tfrac 1 2 + \nu, 1+2\nu; 2i x\big) . 
\end{align*}

Let $y \neq 0$ be real. For  $\Re (b) > 0$ and  $\Re (z) >  0 $, it follows from  \cite[6.10 (5)]{ET-I} that
\begin{align}\label{Aeq: Phi, 1}
	\int_0^{\infty} \Phi (a, c; i x y ) \exp (- x z) x^{b-1} d      x = \Gamma (b) z^{-b} {_2F_1} (a, b; c; iy /z) , 
\end{align}
if $ |z| > |y| $, and 
\begin{align}\label{Aeq: Phi, 2}
	\int_0^{\infty} \Phi (a, c; i x y) \exp (- x z) x^{b-1} d      x = \Gamma (b) (z- iy)^{-b} {_2F_1} (c-a, b; c; iy/(iy-z)) , 
\end{align}
if $|z- iy| > |y|$. For real $\varw \neq 0$, now let  $ z \ra i \varw $,   then \eqref{Aeq: Phi, 1} and \eqref{Aeq: Phi, 2} turn into
\begin{align}\label{Aeq: Phi, 1.1}
	\int_0^{\infty} \Phi (a, c; i x y ) \exp (- i x \varw) x^{b-1} d      x = \Gamma (b) (i \varw)^{-b} {_2F_1} (a, b; c;  y /\varw) , 
\end{align}
if $ |\varw| > |y| $, and 
\begin{align}\label{Aeq: Phi, 2.1}
	\int_0^{\infty} \Phi (a, c; i x y) \exp (- i x \varw) x^{b-1} d      x = \Gamma (b) (i \varw- iy)^{-b} {_2F_1} (c-a, b; c; y/(y-\varw)) , 
\end{align}
if $|\varw- y| > |y|$, provide that $\Re (b-a), \Re (b+a-c) < 0 < \Re (b)$, or equivalently $ 0 < \Re (b) < \min \{ \Re (a), \Re (c-a) \} $, so that the integrals remain  absolutely convergent. One may extend their validity to $ 0 < \Re (b) < 1 + \min \{ \Re (a), \Re (c-a) \} $ by the asymptotic formula for $\Phi (a, c; \pm i y)$ (see \eqref{Aeq: asymp Phi (iy)}),  partial integration, and analytic continuation.

\begin{lem}\label{A-lem: J}
	Let $y, \varw \neq 0$ be real. Then, for \begin{align}
		0 < \Re (a) <   \Re (\nu/2) +   \frac 5 4, \quad   0 < \Re (b)  <  \min \big\{ 1+\Re (a), 2+\Re (\nu-a) \big \} , 
	\end{align}    we have 
	\begin{equation}\label{Aeq: formula J}
		\begin{split}
			\int_0^{\infty} \hskip -2pt	\int_0^{\infty}   J_{\nu}  & ( 2  \sqrt{\varv x})  \exp ( - i x y  -  i  \varv \varw) x^{ \shskip  a - \frac 1 2 \nu   -1} \varv^{     b - \frac 1 2 \nu - 1} d      x d      \varv = \frac {\Gamma (a) \Gamma (b)} {\Gamma (1+\nu) (iy)^{a} (i \varw)^{b} }  \\
			\cdot &	\left\{ 
			\begin{aligned}
				&	\displaystyle  {_2F_1} \bigg(  a, b; 1+\nu; \frac 1 {y \varw} \bigg), & & \text{ if $|y \varw|  > 1$,} \\
				& \displaystyle  \bigg( \frac {y \varw} {y \varw - 1} \bigg)^b {_2F_1} \bigg(  1+\nu-a, b; 1+\nu;  \frac 1 {1- y \varw}  \bigg), & & \text{ if $|y\varw- 1| > 1$.}
			\end{aligned}
			\right.
		\end{split}
	\end{equation}  
\end{lem}

\begin{proof}
	
	First, we use $x \ra \varv x$ and $ y \ra y /\varv $ to transform \eqref{Aeq: J, 2} into 
	\begin{align*}
		\int_0^{\infty} J_{\nu} (2 \sqrt{\varv x}) \exp (-  i x y) x^{\shskip a - \frac 1 2 \nu - 1} d      x =   \frac {\Gamma (a)} {\Gamma (1+\nu) }  \varv^{ \frac 12 \nu }  (iy)^{-a}    \Phi (a, 1+\nu; i \varv/y) . 
	\end{align*}
	Thus the double Bessel integral in \eqref{Aeq: formula J} is equal to
	\begin{align*}
		\frac {\Gamma (a)} {\Gamma (1+\nu) }    (iy)^{-a} 	\int_0^{\infty} \Phi (a, 1+\nu; i \varv / y ) \exp (- \varv z) \varv^{b-1} d      \varv. 
	\end{align*}
	Then, we use \eqref{Aeq: Phi, 1.1} and \eqref{Aeq: Phi, 2.1} to evaluate this integral, obtaining
	\begin{align*}
		\frac {\Gamma (a) \Gamma (b)} {\Gamma (1+\nu) }	 (iy)^{-a}  (i \varw)^{-b} {_2F_1} (a, b; 1+\nu; 1 /y \varw) , 
	\end{align*}
	if $ |y\varw|  > 1$, and 
	\begin{align*}
		\frac {\Gamma (a) \Gamma (b)} {\Gamma (1+\nu) }	 (iy)^{-a}  (i \varw - i/y)^{-b} {_2F_1} (1+\nu-a, b; 1+\nu; 1/(1-y \varw)) , 
	\end{align*}
	if $|y\varw- 1| > 1$, as desired. 
\end{proof}


Similar to \eqref{Aeq: J, 2}, it follows from \cite[6.10  (9)]{Erdelyi-HTF-1} that 
\begin{align}\label{Aeq: K, 2}
	\int_0^{\infty} K_{\nu} (2 \sqrt{x}) \exp (- i x y) x^{\shskip a - \frac 1 2 \nu - 1} d      x =  \frac 1 2 \Gamma (a) \Gamma (a-\nu) (iy)^{-a} \Psi (a, 1+\nu; -i/ y), 
\end{align}
for $\Re (a) >  \max  \{0 , \Re (\nu)  \} $. This case is  easier as $ K_{\nu} (2\sqrt{x}) $ is of exponential decay.  Further, by \eqref{Aeq: Psi}, \eqref{Aeq: asymp Psi (iy)}, \eqref{Aeq: Phi, 1.1}, and \eqref{Aeq: Phi, 2.1}, we may prove the following analogue of Lemma \ref{A-lem: J}. 
\begin{lem}\label{A-lem: K}
	Let $y, \varw \neq 0$ be real. 	Then, for   \begin{align}
		\max  \big \{0 , \Re (\nu)  \big \} < \Re (a), \quad 		 \max  \big \{0 , \Re (\nu)  \big \} <   \Re (b) < \Re (a) + 1  ,
	\end{align} the integral
	\begin{align}
		\int_0^{\infty} \hskip -2pt	\int_0^{\infty}   K_{\nu}   ( 2   \sqrt{\varv x})  \exp ( - i x y  -  i  \varv \varw) x^{ \shskip  a - \frac 1 2 \nu   -1} \varv^{     b - \frac 1 2 \nu - 1} d      x d      \varv  
	\end{align}
	is equal to 
	\begin{equation}\label{Aeq: integral K, 1}
		\begin{split}
			- \frac {\pi} {2 \sin (\pi \nu)} \bigg\{	&  \frac {\Gamma (a) \Gamma (b)   } { \Gamma (1+\nu) (iy)^a (i\varw)^b } {_2F_1} \bigg(a, b; 1+\nu ; - \frac 1 {y \varw} \bigg) \\
			- & \frac {\Gamma (a-\nu) \Gamma (b-\nu)    } { \Gamma (1-\nu) (iy)^{a-\nu} (i\varw)^{b-\nu} } {_2F_1} \bigg(a-\nu, b-\nu; 1-\nu; - \frac 1  {y \varw}\bigg) \bigg\} ,
		\end{split}
	\end{equation}
	if $ |y \varw|  > 1$,  and 
	\begin{equation}\label{Aeq: integral K, 2}
		\begin{split}
			-   \frac {\pi} {2 \sin (\pi \nu)}   \bigg\{  \frac {\Gamma (a) \Gamma (b)   } { \Gamma (1+\nu) (iy)^a (i\varw)^b } \bigg( \frac {y \varw} {y \varw + 1} \bigg)^b {_2F_1} \bigg(1+\nu-a, b; 1+ \nu; \frac 1 {1 + y \varw} \bigg) & \\
			\quad  - \frac {\Gamma (a-\nu) \Gamma (b-\nu)    } { \Gamma (1-\nu) (iy)^{a-\nu} (i\varw)^{b-\nu} } \bigg( \frac {y \varw} {y \varw + 1} \bigg)^{b-\nu} {_2F_1} \bigg(1-a, b-\nu ; 1-\nu;   \frac 1 {1 + y \varw} \bigg) \hskip -1pt \bigg\}  & ,
		\end{split}
	\end{equation}
	if $|y\varw + 1| > 1$.
\end{lem}

Finally, the integral formulae in \eqref{1eq: double Fourier, 1} and \eqref{1eq: double Fourier, 2} follows directly from Lemmas \ref{A-lem: J} and \ref{A-lem: K} on the changes $ 2a = 2\beta +   \nu $ and $2 b = 2 \gamma +   \nu$.   Note that the expressions in \eqref{Aeq: integral K, 1} and \eqref{Aeq: integral K, 2} become more symmetric in this way. 


\section{Distributional Hankel Transform} \label{append: Hankel}

In the sense of distributions, we establish some basic results on the Hankel transform that justify the calculations in \S \ref{sec: diff equations}.

Let $\nu \in \BC \smallsetminus \BZ$ and $p \in \BZ$.   According to \cite[\S\S 3.4, 17.3]{Qi-Bessel}, define the function space 
\begin{equation}
	\begin{split}
		\mathscr{S}_{\mathrm{sis}}^{\nu, \shskip p} (\BC^{\times}) & = |z|^{\nu} (z/|z|)^p \mathscr{S} (\BC) + |z|^{- \nu} (z/|z|)^{-p} \mathscr{S} (\BC) \\
		& = \big\{ |z|^{\nu} (z/|z|)^p \phi_+ (z) + |z|^{- \nu} (z/|z|)^{-p} \phi_-(z) : \phi_+ (z), \phi_- (z) \in \mathscr{S} (\BC) \big\},
	\end{split}
\end{equation}   
where $\mathscr{S} (\BC)$ is the Schwartz space on $\BC$. 
Note the symmetry  $  \mathscr{S}_{\mathrm{sis}}^{\nu, \shskip p} (\BC^{\times}) =  \mathscr{S}_{\mathrm{sis}}^{- \nu, \shskip - p} (\BC^{\times}) $.  
It is clear that $\mathscr{S}_{\mathrm{sis}}^{\nu, \shskip p} (\BC^{\times})$ is invariant under the operators $\nabla = z \partial / \partial z$ and  $\overline{\nabla} = \widebar{z} \partial / \partial \widebar{z}$.  The topology of $\mathscr{S}_{\mathrm{sis}}^{\nu, \shskip p} (\BC^{\times})$ may be easily induced from that of $\mathscr{S} (\BC)$ (see \cite[\S 4.2]{Zemanian-Distributions}), and we denote by $\mathscr{T}_{\mathrm{sis}}^{\nu, \shskip p} (\BC^{\times})  $ its dual space   of continuous linear functionals. Define the operators $ \nabla $ and $\overline{\nabla}$ on  $\mathscr{T}_{\mathrm{sis}}^{\nu, \shskip p} (\BC^{\times})  $ via
\begin{align}\label{AppendB: nabla}
	\langle \nabla f, \phi  \rangle = -   \langle  f, \nabla \phi + \phi  \rangle , \qquad 	 \langle \overline{\nabla} f, \phi  \rangle = -   \langle  f, \overline{\nabla} \phi + \phi  \rangle,  
\end{align}
for any $f \in \mathscr{T}_{\mathrm{sis}}^{\nu, \shskip p} (\BC^{\times})$ and $ \phi \in \mathscr{S}_{\mathrm{sis}}^{\nu, \shskip p} (\BC^{\times}). $ As usual, we write 
\begin{align}\label{append: inner product}
	\langle   f, \phi  \rangle = \viint_{\BC^{\times}} f (z) \phi (z)    {i d      z \vwedge d      \widebar{z}}, 
\end{align}
so that \eqref{AppendB: nabla} is just formal partial integration. For $|\Re (\nu)| < 1$, we have  $ \mathscr{S}_{\mathrm{sis}}^{\nu, \shskip p} (\BC^{\times}) \subset  \mathscr{T}_{\mathrm{sis}}^{\nu, \shskip p} (\BC^{\times}) $ as the $\langle \cdot, \cdot \rangle $ in \eqref{append: inner product} is indeed a  bilinear form on $  \mathscr{S}_{\mathrm{sis}}^{\nu, \shskip p} (\BC^{\times}) $. 

For $|\Re (\nu)| < 1$, it is proven in  \cite[\S  3.4]{Qi-Bessel} that the Hankel transform is  the integral transform: 
\begin{align}
	\EuH_{\nu, \shskip p} {\phi} (u) = 2\pi^2 \viint_{\BC^{\times}} \phi (z) \boldJ_{ \nu,\shskip  p} (\hskip -1.5pt \sqrt{uz})  {i d      z \vwedge d      \widebar{z}} , \qquad \text{($\phi \in \mathscr{S}_{\mathrm{sis}}^{\nu, \shskip p} (\BC^{\times})$)},
\end{align}
and $ \EuH_{\nu, \shskip p}$ is an automorphism on the space $\mathscr{S}_{\mathrm{sis}}^{\nu, \shskip p} (\BC^{\times})$ due to the Hankel inversion: 
\begin{align}
	\varPhi =  \EuH_{\nu, \shskip p} \phi, \qquad \phi =  \EuH_{\nu, \shskip p} \varPhi . 
\end{align}
As a standard and easy consequence, we have the Parseval formula:
\begin{align}\label{apped: Parseval}
	\langle  \EuH_{\nu, \shskip p} {\psi} ,  \EuH_{\nu, \shskip p} {\phi} \rangle =	\langle \psi, \phi \rangle   , \qquad \psi, \phi \in \mathscr{S}_{\mathrm{sis}}^{\nu, \shskip p} (\BC^{\times}). 
\end{align}
For $\phi \in \mathscr{S}_{\mathrm{sis}}^{\nu, \shskip p} (\BC^{\times})$,  by Stokes' formula, it is easy to prove
\begin{align}\label{append: nabla H}
	\nabla \EuH_{\nu, \shskip p} \phi = - \EuH_{\nu, \shskip p} (  \nabla \phi + \phi ), \qquad \overline \nabla \EuH_{\nu, \shskip p} \phi = - \EuH_{\nu, \shskip p} ( \overline \nabla \phi +   \phi ), 
\end{align}
\begin{align}\label{append: nabla H, 2}
	\nabla^2 \EuH_{\nu, \shskip p} \phi =   \EuH_{\nu, \shskip p}  \big( \nabla^2 \phi + 2   \nabla  \phi +    \phi  \big), \qquad \overline \nabla{}^2 \EuH_{\nu, \shskip p} \phi = \EuH_{\nu, \shskip p} \big(  \overline \nabla{}^2 \phi + 2  \overline \nabla \phi +  \phi \big) , 
\end{align}
and also by \eqref{2eq: nabla}--\eqref{2eq: nabla J = 0}, 
\begin{equation}\label{append: nabla H, 2'}
	\begin{split}
		& 4 \nabla^2 \EuH_{\nu, \shskip p} \phi = (\nu+p)^2 \EuH_{\nu, \shskip p}   \phi - 16 \pi^2 u \shskip \EuH_{\nu, \shskip p} (z \phi)   , \\
		& 4 \overline \nabla{}^2 \EuH_{\nu, \shskip p} \phi = (\nu-p)^2 \EuH_{\nu, \shskip p}   \phi - 16 \pi^2 \widebar{u} \shskip \EuH_{\nu, \shskip p} (\widebar z \phi) . 
	\end{split}
\end{equation}

Now we follow Schwartz's approach to the distributional Fourier transform (see \cite{Schwartz-Distributions} or \cite[Chapter 7]{Zemanian-Distributions}) to extend the Hankel transform  to distributions by generalizing the Parseval equation \eqref{apped: Parseval}. To this end, we define the Hankel transform $ \EuH_{\nu, \shskip p} {f}$ of a distribution $  f \in \mathscr{T}_{\mathrm{sis}}^{\nu, \shskip p} (\BC^{\times})$ by 
\begin{align}
	\label{append: distr Hankel} 
	\langle \EuH_{\nu, \shskip p} {f} , \EuH_{\nu, \shskip p} {\phi}  \rangle =	\langle f, \phi \rangle  , \qquad \phi \in \mathscr{S}_{\mathrm{sis}}^{\nu, \shskip p} (\BC^{\times}),
\end{align}
or equivalently, by Hankel inversion on  $\mathscr{S}_{\mathrm{sis}}^{\nu, \shskip p} (\BC^{\times})$, 
\begin{align}
	\label{append: distr Hankel, 2} 
	\langle \EuH_{\nu, \shskip p} {f} ,  {\phi}  \rangle =	\langle f, \EuH_{\nu, \shskip p} \phi \rangle  , \qquad \phi \in \mathscr{S}_{\mathrm{sis}}^{\nu, \shskip p} (\BC^{\times}). 
\end{align}
This determines $  \EuH_{\nu, \shskip p} {f} $ as a distribution in the dual space $\mathscr{T}_{\mathrm{sis}}^{\nu, \shskip p} (\BC^{\times})$.

\begin{thm}
	For $| \Re (\nu)| < 1$, the {\rm(}distributional{\rm)} Hankel transform  $\EuH_{\nu, \shskip p}$ is an automorphism on the space  $\mathscr{T}_{\mathrm{sis}}^{\nu, \shskip p} (\BC^{\times})$. 
\end{thm}

It is straightforward to prove that \eqref{append: nabla H}--\eqref{append: nabla H, 2'} remain valid on $\mathscr{T}_{\mathrm{sis}}^{\nu, \shskip p} (\BC^{\times})$. 

\begin{prop}
	Let $| \Re (\nu)| < 1$. Then for $f \in \mathscr{T}_{\mathrm{sis}}^{\nu, \shskip p} (\BC^{\times})$ we have 
	\begin{align}\label{append: nabla H, 1}
		\nabla \EuH_{\nu, \shskip p} f     = - \EuH_{\nu, \shskip p} (  \nabla f     + f     ), \qquad \overline \nabla \EuH_{\nu, \shskip p} f     = - \EuH_{\nu, \shskip p} ( \overline \nabla f     +   f     ), 
	\end{align}
	\begin{align}\label{append: nabla H, 21}
		\nabla^2 \EuH_{\nu, \shskip p} f     =   \EuH_{\nu, \shskip p}  \big( \nabla^2 f     + 2   \nabla  f     +    f      \big), \qquad \overline \nabla{}^2 \EuH_{\nu, \shskip p} f     = \EuH_{\nu, \shskip p} \big(  \overline \nabla{}^2 f     + 2  \overline \nabla f     +  f     \big) , 
	\end{align} 
	\begin{equation}\label{append: nabla H, 21'}
		\begin{split}
			& 4 \nabla^2 \EuH_{\nu, \shskip p} f     = (\nu+p)^2 \EuH_{\nu, \shskip p}   f     - 16 \pi^2 u \shskip \EuH_{\nu, \shskip p} (z f    )   , \\
			& 4 \overline \nabla{}^2 \EuH_{\nu, \shskip p} f     = (\nu-p)^2 \EuH_{\nu, \shskip p}   f     - 16 \pi^2 \widebar{u} \shskip \EuH_{\nu, \shskip p} (\widebar z f    ) . 
		\end{split}
	\end{equation}
\end{prop}

\subsection*{Remarks} As the Hankel inversion is established for fundamental Bessel functions for $\GL_n (\BR)$ and $\GL_n (\BC)$ in \cite[\S 3]{Qi-Bessel}, at least in the self-dual case, the same constructions of distributions and Hankel transform may be applied for $\GL_n$.  Note that in the $\GL_2$ case (see \cite[\S 18]{Qi-Bessel})  the Hankel inversion is equivalent to 
\begin{align*}
	\begin{pmatrix}
		& -1 \\ 1 & 
	\end{pmatrix}^2 = \begin{pmatrix}
		-1 & \\ & \hskip -4pt -1
	\end{pmatrix}. 
\end{align*}

The distributional Hankel transform for $J_{\mu} (x)$ ($ \mu  \geqslant -   1 / 2$) has been studied in \cite{Zemanian-Hankel} and \cite[Chapter V]{Zemanian-Int-Trans}. Although the original definition seems different,  his space $\mathscr{H}_{\mu}$ is actually similar to our $ \mathscr{S}_{\mathrm{sis}} $ space in view of \cite[Lemma 5.2-1]{Zemanian-Int-Trans}. 

  \end{appendices}

{\large \part*{II. Arithmetic Part} \addtocounter{part}{2}  \label{part: arithm}}

\section{Main Results II: Explicit Spectral Formulae}

\subsection{Notation over \protect \scalebox{1.08}{$\BQ (i)$}} \label{sec: notation Q(i)}

Denote $\RF = \BQ (i)$ and $ \CaloO = \BZ [i]$. For $p \in \BZ$, define 
\begin{align}
	\zeta_{\RF} (s, p) =  \sum_{(n) \shskip \subset \CaloO } \frac {(n/|n|)^{4 p}} {|n|^{2s}}, \qquad \text{($\Re (s) > 1$)}, 
\end{align}
which is the Hecke $L$-function associated with the Gr\"ossencharakter $(n) \ra (n/|n|)^{4 p}$. Note that  $$\zeta_{\RF} (s, 0) = \zeta_{\RF} (s) =  \sum_{(n) \shskip \subset \CaloO } \frac {1} {|n|^{2s}}, \qquad \text{($\Re (s) > 1$)},  $$ is the Dedekind zeta function. It is well known that $\zeta_{\RF} (s, p) $ has analytic continuation on the whole complex plane, except for a simple pole at $s = 1$ if $p =0$.  Define the  Stieltjes constant   $\gamma_{\RF}  $ 
by 
\begin{align}\label{2eq: zeta (s), s=1}
	\zeta_{\RF} (s) = \frac {\pi} 4 \cdot \frac {1} {s-1} + \gamma_{\RF} + O (|s-1|), \qquad \text{($s \ra 1$)}.  
\end{align}  

For $n \in \CaloO \smallsetminus \{0\}$ and $\nu \in \BC$, define 
\begin{align}\label{1eq: tau s}
	\vtau_{\nu} (n) = \sum_{ (a) (b) = (n)} |a/b|^{2\nu}. 
\end{align}  
Note that $\vtau_{-\nu} (n) = \vtau_{\nu} (n)$. By our convention, let $\vtau (n) = \vtau_0 (n)$.  

For $r \in \BR$ and $2 p \in \BZ$, let $\pi_{i r, p} $ be the unitary principle series of $\mathrm{PGL}_2 (\BC)$: the unique infinite dimensional constituent of the representation   unitarily induced from
\begin{align*}
	\begin{pmatrix}
		\sqrt{z} & \varv \\
		& 1/\sqrt{z}
	\end{pmatrix}   \ra 
	|z|^{2i r} (z/|z|)^{2 p}. 
\end{align*}
As the Selberg conjecture holds for $\mathrm{PGL}_2 (\CaloO)$, we do not consider complementary series. 

Let  $  L^2_{c} (\mathrm{PGL}_2 (\CaloO) \backslash \mathrm{PGL}_2 (\BC))  $ be the space of  cusp forms for $ \mathrm{PGL}_2 (\CaloO) $.  Let $\Pi_{c} $ denote the discrete   spectrum comprising the irreducible constituents of $L^2_{c} (\mathrm{PGL}_2 (\CaloO) \backslash \mathrm{PGL}_2 (\BC))$. It may be assumed that each  $\pi \in \Pi_{c}$ is Hecke invariant. Let $\vlambda_{\pi} (n)$ be the Hecke eigenvalues of $ \pi $. Note that the $\mathrm{PGL}_2 (\CaloO) $-invariance ensures that $ \vlambda_{\pi} (n) $ only depends on the ideal $(n)$. We have the Hecke relation:
\begin{align}\label{3eq: Hecke relation}
	\vlambda_{\pi} ( m) \vlambda_{\pi}  (n) = \sum_{ (d) | (m, n) } \vlambda_{\pi}  (m n / d^2  ). 
\end{align} 
Let $L(s, \pi)$ and $L (s, \mathrm{Sym}^2 \pi)$ be the standard and the symmetric square $L$-functions for $\pi$: 
\begin{align}\label{3eq: L-functions}
	L (s, \pi) = \sum_{(n)   \subset   \CaloO }  \frac { \vlambda_{\pi} (n) } {|n|^{2  s} }, \quad L (s, \mathrm{Sym}^2 \pi) = \zeta_{\RF} (2s) \sum_{(n)   \subset   \CaloO }  \frac { \vlambda_{\pi} (n^2) } {|n|^{2  s} }, \qquad \text{($\Re (s) > 1$)}. 
\end{align}
For $\pi \in \Pi_{c}$, let $  (r_{\pi}, p_{\pi})$ be the parameter $(r, p)$ such that  $ \pi \simeq \pi_{i r , p }$. Then   the root number of $ \pi $ is equal to $(-1)^{2 p_{\pi}}$. Finally, let $ \Pi_{c}^+ $ or $\Pi_{c}^0$ denote the set of even or spherical $\pi$ according as $2 p_{\pi}$ even or $p_{\pi} = 0$, respectively. In the spherical case $\Pi_{c}^0$ corresponds to an orthonormal basis of Hecke cusp forms on $ \mathrm{PGL}_2 (\CaloO) \backslash \BH^3  $.  By convention, we shall use $f  $ instead of $\pi$ for the spherical Hecke cusp forms in $\Pi_{c}^0$.  

\subsection{First application: The explicit Bruggeman--Motohashi formula on \protect \scalebox{1.08}{$\BQ (i)$}} 

First we use Theorem \ref{thm: Hankel} (or Corollary \ref{cor: Hankel}) to make the Bruggeman--Motohashi formula as in \cite[Theorem 14.1]{B-Mo}\footnote{Occasionally, our definitions are slightly different from those in \cite{B-Mo}---usually by a factor $2$---as theirs have incorporated the different ideal $(2)$.} more explicit in terms of the hypergeometric function.


\begin{thm}\label{thm: Bru-Moto} Let notation be as above. For $\varg \in \mathscr{H}$ as in Definition {\rm\ref{1defn: H}},  define  
	\begin{align}
		\SZ_2 (\varg; \RF) = \int_{-\infty}^{\infty} \left|\zeta_{\RF}  \big(\tfrac 12+it \big)\right|^4 \varg (t) d      \shskip t . 
	\end{align} 
	Then we have the identity
	\begin{equation}
		\begin{split}
			\SZ_2 (\varg; \RF) = \RM_{\RF} (\varg) & + \frac 1 {\pi^2}  \sum_{\pi \in \Pi_{c}^{+}}  \frac {L \big(\frac 1 2 , \pi \big)^3 } { L (1, \mathrm{Sym}^2 \pi) } \Lambda (    r_{\pi},  p_{\pi}; \varg) \\
			& + \frac 1 {  \pi^3} \sum_{ p = - \infty }^{\infty}  \int_{-\infty}^{\infty} \frac {\big|\zeta_{\RF} \big(\frac 1 2 + ir, p \big)\big|^6} {|\zeta_{\RF} (1+2ir, 2 p)|^2}  \Lambda ( r, 2 p; \varg) d r , 
		\end{split}
	\end{equation}
	where $ \RM_{\RF} (\varg)$ is   explicit in terms of $\varg (t) $, $\Gamma (s)$, and {\rm(}the Stieltjes constants of{\rm)}  $\zeta_{\RF} (s)$,
	\begin{align}\label{1eq: Lmabda}
		\Lambda (r, p; \varg) = \viint 	\frac { {\varg_c}  (2\log |1+1/u|) } {|u(1+u)|} \Xi (u ; ir, p ) \shskip i d      u \vwedge d      \widebar{u} ,
	\end{align} 
	\begin{align}
		\varg_c (x) = \int_{-\infty}^{\infty} \varg (t) \cos (t x) d      \shskip t 
	\end{align}
	is the cosine Fourier transform, and  
	\begin{equation} \label{1eq: Xi hyper} 
		\Xi (u; \nu, p) =  - \Re \bigg\{    \frac {   \cot (\pi \nu) } {  |u|^{1+2\nu} (u/|u|)^{2p} } F \bigg(\nu+ p; -\frac 1 u \bigg)  F \bigg(\nu - p; -\frac 1 {\widebar u} \bigg)  \bigg\} ,  
	\end{equation}
	with
	\begin{equation}\label{1eq: Xi hyper, 2}
		F (\nu; z) =	\frac { \Gamma \big(\frac 1 2 + \nu   \big)^2  } {\Gamma (1+2\nu)}  {_2F_1} \hskip -1pt \bigg(  \frac 1 2 + \nu  , \frac 1 2 + \nu   ; 1+ 2 \nu   ; z  \bigg) . 
	\end{equation}
\end{thm}

Our   $ \Lambda (r, p; \varg)  $ in \eqref{1eq: Lmabda}--\eqref{1eq: Xi hyper, 2} is clearly the complex analogue of  $\Lambda (r; \varg)$ in \eqref{0eq: Lambda}--\eqref{0eq: Xi}, while its  formula  given by  Lemmas 2.1, 14.1, and Theorem 14.1 in  \cite{B-Mo} is quite    involved and different. Actually, our proof is based on another formula of  Bruggeman and Motohashi  in their concluding \S 15:  
\begin{align}\label{1eq: Xi}
	\Xi (u; \nu, p) = \frac 1 {2  } \viint_{\BC^{\times}} j_{\nu, p} (z / u) j_{0} (-z) \frac {i d      z \vwedge d      \widebar{z}}   {|z|^{3 } },  
\end{align}
with
\begin{align}\label{1eq: j}
	j_{\nu, p} (z) = 2\pi^2 |z| \boldJ_{2 \nu,\shskip 2 p} ( \hskip -1.5pt  \sqrt{z}). 
\end{align}
It is therefore clear that  Theorem \ref{thm: Bru-Moto} is a direct consequence of  their Theorem 14.1  and our Theorem \ref{thm: Hankel}.

\subsection{Second application: The explicit Kuznetsov--Motohashi formula on \protect \scalebox{1.08}{$\BQ (i)$}} \label{sec: 2nd Appl}

By  the idea of Kuznetsov in \cite{Kuznetsov-Motohashi-formula} and the method of regularization in \cite{Qi-BE}, we shall prove in \S \ref{sec: proof of K-M} the following Kuznetsov--Motohashi formula over $\BQ (i)$.  

\begin{thm}\label{thm: Kuz-Moto} Let notation be as above.   Let $m \in \CaloO \smallsetminus \{0\}$.  Let $h \in \mathscr{H}$ satisfy
	\begin{align}\label{0eq: h (i/2)=0}
		h  (\pm   i/ 2  ) = 0,
	\end{align}
	and   
	\begin{align}\label{1eq: decay of h}
		h (r) \Lt \exp (-c |r|^2) , \qquad \text{{\rm($c > 0$)}}, 
	\end{align}
	in any fixed horizontal strip.  
	Define 
	\begin{equation} 
		\ScrL_2  ( m; h)  = \sum_{f \in \Pi_{c}^0}   \frac { L \big(\frac 1 2 , f \big)^2       } {L(1, \mathrm{Sym^2} f )} \vlambda_f ( m )    h  ( r_f ) +  \frac 1 {\pi}    \int_{-\infty}^{\infty} \hskip -2pt   
		\frac {\left| \zeta_\RF \big(\frac 1 2 +ir \big) \right|^4     } {|\zeta_\RF(1+2i r)|^2} \vtau_{i r} ( m ) h ( r )  \shskip   d  r   .
	\end{equation}
	Then we have the identity 
	\begin{equation}\label{1eq: Kuznetsov, Q(i)}
		\begin{split}
			\ScrL_2  ( m; h)  = 		\frac {2 \vtau  (m)} { \pi^3  | m| } \hskip -1pt \Big( \hskip -1pt \Big(     \gamma_{\RF} -  \frac {\pi} 4 \log |\pi^2m|   \Big)  \mathrm{B}(h)    + \frac {\pi} 2 \breve{\mathrm{B}} (h) \hskip -1pt \Big)  - \frac {i \pi     \vtau_{ \frac 1 2} (m)   } {4 \zeta_{\RF} (2 ) }  h' (i/2) & \\
			+    \frac 1 {2 \pi^3} \sum_{n \shskip  \in \CaloO \smallsetminus \{0, \shskip m  \} }   \frac {  \vtau (n-m) \vtau  (n  )   } {       |  n|  }  \Phi (n/m) &  , 
		\end{split}
	\end{equation}
	where 
	\begin{align}\label{1eq: Phi}
		\mathrm{B}(h) =  \int_{-\infty}^{\infty} h (r) r^2 d r, \qquad \breve{\mathrm{B}}(h)  = \int_{-\infty}^{\infty} \frac{\Gamma'} {\Gamma}  \lp \frac 1 2 + ir \rp  h (r) r^2 d r,
	\end{align}
	\begin{align}\label{1eq: Theta}
		\Phi (u; h) = 	\int_{-\infty}^{\infty}   \Theta (u; ir)  h (r)  r^2 d r ,
	\end{align} 
	\begin{align}\label{1eq: F}
		\Theta (u; \nu) = - \frac {   \cot (\pi \nu) } {  |u|^{ 2\nu}   } F \bigg(\nu ;  \frac 1 u \bigg)  F \bigg(\nu ;  \frac 1 {\widebar u} \bigg),
	\end{align}
	with $F (\nu; z)$ defined as in {\rm\eqref{1eq: Xi hyper, 2}.}
 
\end{thm}

The convergence of the $r$-integral in \eqref{1eq: Theta} and the $n$-sum in  {\rm \eqref{1eq: Kuznetsov, Q(i)}} will be discussed   in {\rm \S} {\rm \ref{sec: uniform bounds, hyper}}. 

According to {\rm\cite[(1), (4)]{Zagier-Dedekind}}, 
\begin{align*}
	\zeta_{\RF} (2) = \frac {\pi^2} 6 A(1), \qquad A(1) = \int_0^1 \frac {1} {1+t^2} \log \frac {4} {1+t^2} d t,
\end{align*}
so the second term in {\rm \eqref{1eq: Kuznetsov, Q(i)}} could be slightly more explicit.   

\subsection{Remarks on the Kuznetsov--Motohashi formula on \protect \scalebox{1.08}{$\BQ$}}\label{sec: remarks Kuznetsov}
Finally, we  explain and provide our solution to the issues in Kuznetsov's proof  \cite{Kuznetsov-Motohashi-formula}\allowbreak---Motohashi \cite{Motohashi-JNT-Mean} only said that Kuznetsov's argument `is sketchy and seems to need some revision'. 

By comparing with the argument for the holomorphic case as in \cite[\S 4]{BF-Moments}, it seems that Kuznetsov's proof of his Theorem 3.1 becomes valid if he shifted the contour of $$\frac {2i} {\pi} \int_{-\infty}^{\infty}   \frac {J_{2 i r }    (\hskip -1pt \sqrt{x})  } {\cosh (\pi r) }  h (r)   r  d r$$
to $\Re (ir) = 1$, say, prior to the application of Vorono\"i--Oppenheim. However, Kuznetsov did not provide a proof for his Theorem 3.3, indicating that it follows the scheme of the proof of  Theorem 3.1 by using $$\frac {2i} {\pi} \int_{-\infty}^{\infty}   \frac {I_{2 i r }    (\hskip -1pt \sqrt{x})  } {\cosh (\pi r) }  h (r)   r  d r. $$ There is a serious issue with the application of Vorono\"i--Oppenheim as $ I_{2 i r } ( {x})$ is of exponential growth, so, instead,  Motohashi considered 
$$\frac {4} {\pi^2} \int_{-\infty}^{\infty}     {K_{2 i r }    (\hskip -1pt \sqrt{x})  } {\sinh (\pi r) }  h (r)   r  d r, $$
and applied the Mellin--Barnes integral representation of $K_{2ir} (x)$,  and subsequently  the functional equation for the Estermann function. Nevertheless, it is explained in \cite[\S 3.3]{Motohashi-Riemann}  that his approach does not apply for  $J_{2ir} (x)$. 

Define 
\begin{align*}
	P_{\nu } (x) = \frac {(x/2)^{\nu } } {\Gamma (\nu+1)}.
\end{align*} 
Let $\txw (x) \in C^\infty(\BR_+)$ be such that  $\txw (x) \equiv 1$ for $x$ near $0$ and $\txw (x) \equiv  0$ for $x$ large.   Our idea is to write 
\begin{align*}
	J_{2ir} (\hskip -1pt \sqrt{x}) = \txw (x) P_{2ir} (\hskip -1pt \sqrt{x}) +  \big( J_{2ir} (\hskip -1pt \sqrt{x})  - \txw (x) P_{2ir} (\hskip -1pt \sqrt{x}) \big) , 
\end{align*}
or apply
\begin{align*}
	I_{2ir} (\hskip -1pt \sqrt{x}) = \txw (x) P_{2ir} (\hskip -1pt \sqrt{x}) + \big( I_{2ir} (\hskip -1pt \sqrt{x})  - \txw (x) P_{2ir} (\hskip -1pt \sqrt{x}) \big) , 
\end{align*}
in the formula
\begin{align*}
	K_{2 ir} (\hskip -1pt \sqrt {x }) = 	\frac {\pi i} {2 \sinh (2\pi \nu) }   \big( I_{2  ir} (\hskip -1pt  \sqrt {x }) - I_{2 ir } (\hskip -1pt  \sqrt {x }) \big) .     
\end{align*}
After the regularization, the Bessel functions would have sufficient decay at $0$ to apply Vorono\"i--Oppenheim, while $\txw (x) P_{2ir} (\hskip -1pt \sqrt{x})$ may be handled easily by the Mellin technique.

\section{Proof of Theorem \ref{thm: Bru-Moto}: The Explicit Bruggeman--Motohashi Formula over $\BQ(i)$} 

Let $\Re (\nu) = 0$ and $p$ be integral. By \eqref{1eq: Xi} and \eqref{1eq: j},   we have 
\begin{align}
	\Xi (u; \nu, p) =    {2 \pi^4}   \viint_{\BC^{\times}} \boldJ_{2\nu, 2p} (\hskip -1.5pt \sqrt{z } ) \boldJ_{0} (\hskip -1.5pt \sqrt{-uz} )   {i d      z \vwedge d      \widebar{z}}    . 
\end{align}
It follows from Corollary    \ref{cor: Hankel}   that 
\begin{align}
	\Xi (u; \nu, p) =   \frac {1} {2   |u| }  {_{\phantom{\nu}}^{1 } \hskip -1pt \boldF_{2\nu, \shskip 2p}^{0}} \lp - \frac 1 {u} \rp . 
\end{align}
As $\Re (\nu) = 0$, the formula for $  \Xi (u; \nu, p) $ given by \eqref{1eq: Xi hyper} and \eqref{1eq: Xi hyper, 2} follows immediately from the definition of ${_{\phantom{\nu}}^{1 } \hskip -1pt \boldF_{2\nu, \shskip 2p}^{0}} (z)$ as in \eqref{1eq: hypergoem, C, 1.0}--\eqref{1eq: hypergoemetric}.

\section{Preliminaries for the Proof of Theorem  \ref{thm: Kuz-Moto}}


In this section, we introduce some fundamental notions and tools that we need to prove Theorem  \ref{thm: Kuz-Moto}. For the notation over $\RF = \BQ (i)$, the reader is referred to \S \ref{sec: notation Q(i)}. 

\subsection{Kloosterman and Ramanujan sums}  
For $  m,  n \in \CaloO $ and $c \in \CaloO \smallsetminus \{0\}$,    we define the Kloosterman sum 
\begin{align}\label{2eq: defn Kloosterman KS}
	S_{\RF} (m, n ; c ) = \sum_{a \, \in (\CaloO /c)^{\times } } e \bigg(   \Re \bigg( \frac {  a m +   \widebar{a} n } {c} \bigg)  \bigg), 
\end{align} 
where $\widebar{a} $ is the multiplicative inverse of $a$ modulo $c$.   We have the Weil bound:
\begin{align}
	\label{7eq: Weil}
	 S_{\RF} (m, n; c)   \Lt \vtau (c) {|(m, n, c)|} {|c|}. 
\end{align} 
The sum $S_{\RF} (m, 0; c)$ is usually named after Ramanujan. We have 
\begin{align}\label{2eq: Ramanujan}
	S_{\RF} (m, 0; c) = \sum_{ (d) | ( m , c)}   \mu   ( c /d   ) |d|^2 ,
\end{align} 
where $\mu (n) $ is the M\"obius function  (of ideals) over $\RF$. For $m \neq 0$, we have the Ramanujan identity:
\begin{align}\label{7eq: sum of (c)}
	\sum_{(c)   \subset   \CaloO  }  \frac { S_{\RF} (m , 0; c)  } { |c  |^{2 s} } = \frac { \sigma_{1-s} (m) } {  \zeta_{\RF} (s) }, \qquad \text{($\Re (s) > 1$)},
\end{align}
with 
\begin{align}
	\sigma_{s} (m) = \sum_{(d) | (m)} |d|^{2s} .  
\end{align}
Note that $S_{\RF} (0, 0; c) = \varphi (c)$ is the Euler toitent function over $\RF$, and 
\begin{align}\label{7eq: sum of c, 2}
	 \sum_{(c)   \subset   \CaloO  }  \frac { \varphi (c)  } { |c  |^{2s} } = \frac {\zeta_{\RF}  (s-1)} {\zeta_{\RF} (s)}, \qquad \text{($\Re (s) > 2$)}.
\end{align}
To unify\eqref{7eq: sum of (c)} and \eqref{7eq: sum of c, 2}, it will be convenient to define 
\begin{align}\label{7eq: sigma(0)} 
	\sigma_{ s} (0) = \zeta_{\RF} (-s). 
\end{align}

\subsection{Kuznetsov trace formula over  \protect \scalebox{1.08}{$\BQ (i)$}} 

The version of Kuznetsov formula here for $\mathrm{PGL}_2 (\CaloO)$ is from   \cite[Proposition 3.3]{Qi-Liu-Moments}. It may be deduced from the Kuznetsov formula for $ \mathrm{SL}_2 (\CaloO)$ as in \cite[Theorem]{Motohashi-Kuznetsov-Picard-0}.  The reader  is also referred to \cite[Theorem 6.1]{BM-Kuz-Spherical},  \cite[Proposition 1]{Venkatesh-BeyondEndoscopy},  or \cite[Proposition 3.5]{Qi-GL(3)} for the Kuznetsov formula in various forms over number fields.

\begin{lem}\label{lem: Kuznetsov}
	 Let $ h \in \mathscr{H} $ be as in Definition {\rm \ref{1defn: H}}. 
	 We have the identity{\rm:} 
	 \begin{equation}\label{1eq: Kuznetsov} 
	 	\begin{aligned}
	 		\sum_{f \in  \Pi_{c}^{0}  }       \frac {  \vlambda_f ( m )     \vlambda_f  ( n )} {L(1, \mathrm{Sym^2} f )} & h  ( r_f ) + \frac 1 {\pi}  \int_{-\infty}^{\infty} \hskip -2pt   
	 		\frac {\vtau_{i r} (m )  \vtau_{i r} ( n )} {|\zeta_\RF(1+2i r)|^2} h ( r )  \shskip   d  r  \\
	 			=    &   \frac {1} { 2\pi^3 }  \delta_{(m),   (n)}  \mathrm{B} (h)  + \frac 1 {4 \pi }  \sum_{ \epsilon   \in   \CaloO^{\times} \hskip -1pt / \CaloO^{\times 2} } \sum_{c   \in   \CaloO \smallsetminus \{0\} } \frac {S_{\RF} (  m , \epsilon  n  ; c  ) } { |c|^2  } \mathrm{B}  \bigg( \frac { {\epsilon m n} } {  4 c^2     }; h  \bigg),
	 	\end{aligned}
	 \end{equation}
	 where $\delta_{(m),   (n)}$ is the Kronecker $\delta$ symbol for ideals, $ S_{\RF} ( m ,  \epsilon n  ; c  ) $ is the Kloosterman sum defined as in {\rm\eqref{2eq: defn Kloosterman KS}}, and  
	 \begin{align}\label{1eq: defn Bessel integral}
	 	\mathrm{B}(h)  = \int_{-\infty}^{\infty} h (r) r^2 d r, \qquad \mathrm{B} (z; h) = \int_{-\infty}^{\infty} h (r)  \boldJ_{2 i r} (\hskip -1.5pt \sqrt z ) r^2 d r  . 
	 \end{align} 
\end{lem}

\subsection{Vorono\"i--Oppenheim summation formula over  \protect \scalebox{1.08}{$\BQ (i)$}} 
The version of Vorono\"i--\allowbreak Oppenheim formula here with additive twist is from  \cite[Theorem 1.3]{Qi-VO} in the case $\RF = \BQ (i)$ (choose therein $\zeta = a/c$, $\mathfrak{a} = (1)$, $S = \big\{ \mathfrak{p} : \mathfrak{p} | (c) \big\}$, and  $\mathfrak{b} = (c^2)$). The special case of $a =  c = 1$ (with no additive twist) may also be found in \cite{BBT-VO}. 

\begin{lem}\label{lem: Voronoi} Let $\mu \in \BC$.    Let $a, \widebar{a},  c \in \CaloO$ with $c \neq 0$ be such that $(a, c) = (1)$ and $a \widebar{a} \equiv 1 (\mathrm{mod}\, c)$.  Then for $ \phi \in C_c^{\infty} (\BC^{\times}) $  we have the identity{\rm:}
	\begin{align}\label{7eq: Voronoi} 
		\begin{aligned}
			 	   \sum_{n  \shskip \in   \CaloO \smallsetminus \{0\}  }        e \Big(   \Re \Big(  \frac {a n} {c} \Big) \Big)        \vtau_{\mu}  (n )   \phi   (n )        
		&	=	  \frac {\zeta_{\RF} (1 + 2 \mu) } {2 |c|^{2 + 4 \mu}}  \EuM \phi ({2\mu}) + \frac {\zeta_{\RF} (1 - 2 \mu) } {2 |c|^{2 - 4 \mu}}  \EuM  \phi ({-2\mu} )\\
		&	  +   \frac 1 {|2 c|^2}   \sum_{ n \shskip  \in    \CaloO \smallsetminus \{0\} }    e \Big( \Re \Big( \hskip -2 pt - \frac {\widebar{a} n} {c} \Big) \Big)  \vtau_{\mu} (n  ) \EuH_{2\mu} \shskip  \phi  \bigg(\frac { {n}}{ {4} c^2 }\bigg)   ,
		\end{aligned}
	\end{align}
where $ \EuM \phi ({\mu}) $ is the Mellin transform 
\begin{align}
	\EuM  \phi ({\mu}) = \viint_{\BC^{\times} } \phi (z) |z|^{\mu}  i d      z \vwedge d      \widebar{z} ,
\end{align}
and $\EuH_{\mu} \shskip  \phi (u) $ is the Hankel transform {\rm(}see Appendix {\ref{append: Hankel}}{\rm)}
\begin{align}
\EuH_{\mu} \shskip  \phi (u)  =	2\pi^2 \viint_{\BC^{\times}} \phi (z) \boldJ_{ \mu } (\hskip -1.5pt \sqrt{uz})  {i d      z \vwedge d      \widebar{z}} . 
\end{align}
\end{lem}

Subsequently, let us assume $\Re (\mu) = 0$ and $\mu \neq 0$ for simplicity.   For our application, we have to enlarge the space $ C_c^{\infty} (\BC^{\times}) $. 
To this end, we introduce the following function spaces. 

\begin{defn}\label{7defn: function space, 1}
 	For $ \valpha > 0 > \beta$, define $\mathscr{T}_{\valpha, \shskip \beta}^{M} (\BC^{\times})  $ to be the space of functions $\phi   \in C^{2 M} (\BC^{\times})$  {\rm(}here $M = \infty$ is admissible{\rm)} such that 
\begin{align}\label{7eq: bound phi}
	  z^{  r} \shskip \widebar z^{ s} (\partial /\partial z)^r (\partial / \partial \widebar z)^{s} \phi (z) \Lt_{r, \shskip s  } \min \left\{ |z|^{\valpha},  |z|^{  \beta - 2 } \right\} ,  
\end{align}
\begin{align}\label{7eq: bound phi, 2}
   \frac {z^{  r} \shskip \widebar z^{ s} (\partial /\partial z)^r (\partial / \partial \widebar z)^{s} \phi (z) } { \min \left\{ |z|^{\valpha},  |z|^{  \beta - 2 } \right\}  } \longrightarrow 0, \qquad |z| \ra 0, \, \infty, 
\end{align}
for any $r+s \leqslant 2 M$.  For every $N \leqslant 2 M$, define the semi-norm
\begin{align}
	\rhoup_{N} (\phi) = \max_{r+s = N} \left\{ \sup_{z \shskip \in \BC^{\times} \hskip -2pt}   \frac{ |z|^{r+s}  \left| (\partial /\partial z)^r (\partial / \partial \widebar z)^{s} \phi (z) \right|}{\min \left\{ |z|^{\valpha},  |z|^{  \beta -2 } \right\} } \right\}  . 
\end{align} 
\end{defn}

\begin{defn}\label{7defn: function space, 2}
Fix $a \in \BC^{\times}$.	For $ \valpha > 0 > \beta$, define $\mathscr{P}^{\infty}_{\valpha, \shskip \beta}  (\BC^{\times}; a)  $ to be the space of functions $\phi   \in C^{\infty} (\BC^{\times})$ such that for each $M$ one may split $ \phi (z) $ into 
\begin{align}\label{7eq: phi split}
	 \phi (z) = {e (  \Re  (a\hskip -1.5pt \sqrt{z}))}     \phi_{M}  (\hskip -1.5pt \sqrt{z}) +   {e (-   \Re (a\hskip -1.5pt \sqrt{z}))}   \phi_{M}   (- \hskip -1.5pt \sqrt{z}) + \psi_{M} (  \hskip -1.5pt \sqrt{z}) , 
\end{align}
for  certain $  \phi_{M} ,   \psi_{M} \in  \mathscr{T}_{2\valpha, \shskip 2 \beta-2}^{M} (\BC^{\times}) $, with $\psi_M$ even.  
\end{defn} 

The topology of $\mathscr{T}_{\valpha, \shskip \beta}^{M} (\BC^{\times})  $ is defined by the semi-norms $\rhoup_N$ in the standard way. Naturally, the topology of $ \mathscr{P}^{\infty}_{\valpha, \shskip \beta}  (\BC^{\times}; a)   $ is induced from that of   $\mathscr{T}_{2\valpha, \shskip 2 \beta-2}^{M} (\BC^{\times})  $.   Note that  $C_c^{\infty} (\BC^{\times})$ is a dense  subspace of  $\mathscr{T}_{\valpha, \shskip \beta}^{M} (\BC^{\times})  $ or $ \mathscr{P}^{\infty}_{\valpha, \shskip \beta}  (\BC^{\times}; a)   $.   

\begin{lem}\label{lem: bound for H, 1} Let  $\Re (\mu) = 0$ and $\mu \neq 0$.  Let  $ \valpha > \vepsilon > 0 > \beta$  and $M$ be large {\rm(}say $ \vepsilon  M  \geqslant   (2+\valpha) (2+\valpha - \beta)  ${\rm)}.  Then, for $\phi \in \mathscr{T}_{\valpha, \shskip \beta}^{M} (\BC^{\times}) $ as in Definition {\rm\ref{7defn: function space, 1}},  its Hankel transform has bound
	\begin{align}\label{7eq: bound for H}
		 \EuH_{\mu} \shskip  \phi (u) = O_{\mu, \shskip  \vepsilon, \shskip \phi} (1 / |u|^{2+ \valpha - \vepsilon} )  , 
	\end{align}
for $|u| \Gt 1$, with the implied constant uniformly bounded if $\phi $ stays in a bounded set in   $\mathscr{T}_{\valpha, \shskip \beta}^{M} (\BC^{\times})   $. 
\end{lem}

\begin{lem}\label{lem: bound for H, 2} Let  $  \mu, \valpha, \beta, \vepsilon$ be as in Lemma {\rm\ref{lem: bound for H, 1}}. Fix $a \in \BC^{\times}$. Then, for $\phi \in \mathscr{P}^{\infty}_{\valpha, \shskip \beta}  (\BC^{\times}; a)  $ as in Definition {\rm\ref{7defn: function space, 2}},  its Hankel transform has bound
	\begin{align}\label{7eq: bound for H, 2}
		\EuH_{\mu} \shskip  \phi (u) = O_{\mu, \shskip  \vepsilon, \shskip \phi} (1 / |u|^{2+ \valpha - \vepsilon} )  , 
	\end{align}
	for $|u| \Gt 1$, with the implied constant uniformly bounded if the  $\phi_{M} ,   \psi_{M} $ in {\rm\eqref{7eq: phi split}}   stay in a bounded set in  $\mathscr{T}_{\valpha, \shskip \beta}^{M} (\BC^{\times})   $ for a large $M$ {\rm(}in terms of $\vepsilon$, $\valpha$, and $\beta${\rm)}. 
\end{lem}

\begin{rem}\label{rem: implied const}
	It is easy to see from the proof that the implied constant in Lemma {\rm\ref{lem: bound for H, 1}}  is linear in terms of $ \rhoup_{N} (\phi) $ {\rm(}similar is that in Lemma {\rm\ref{lem: bound for H, 2}}{\rm)}. 
\end{rem}

By uniform approximation using   functions in $C_c^{\infty} (\BC^{\times})$, we have the following corollary of Lemmas \ref{lem: bound for H, 1} and \ref{lem: bound for H, 2}. 

\begin{cor}\label{cor: V-O}
The Vorono\"i--Oppenheim summation formula in {\rm\eqref{7eq: Voronoi}} is valid for any $\phi$ in the spaces $  \mathscr{T}_{\valpha, \shskip \beta}^{\infty} (\BC^{\times})  $ and $\mathscr{P}^{\infty}_{\valpha, \shskip \beta}  (\BC^{\times}; a)$ as in Definitions {\rm\ref{7defn: function space, 1}} and {\rm\ref{7defn: function space, 2}}. 
\end{cor}

Now we prove Lemma \ref{lem: bound for H, 1}---the proof of Lemma \ref{lem: bound for H, 2} is similar.

\begin{proof}[Proof of Lemma \ref{lem: bound for H, 1}]
For $  \delta < 1$, let  $\txw_{0} (|z|) + \txw_{\infty} (|z|) \equiv 1$ be a smooth partition of unity for $\BA  (  0, 2  |u|^{- \delta }  ] \cup \BA  [ |u|^{- \delta } ,   \infty )$, satisfying  $x^{r} \txw_{0}^{(r)} (x),  x^r \txw_{\infty}^{(r)}   (x) \Lt_{\, r} \hskip -1pt 1 $. Accordingly, we partition  $\EuH_{\mu} \shskip  \phi (u) $ into 
	 \begin{align*}
	 	 \EuH_{\mu} \shskip  \phi (u) & = \EuH_{\mu}^0 \shskip  \phi (u) + \EuH_{\mu}^{\infty} \shskip  \phi (u) . 
	 \end{align*}
 By \eqref{2eq: bound for J mu m, weak, 2}, \eqref{2eq: bound 1/z}, and \eqref{7eq: bound phi}, trivial estimation yields
 \begin{align}\label{7eq: bound for H, 1}
 	 \EuH_{\mu}^0 \shskip  \phi (u) \Lt \viint_{\BA (0, \, 2 |u|^{- \delta }]} |z|^{\valpha  }    | d      z \vwedge d      \widebar{z} | \Lt |u|^{-  (2 + \valpha) \delta} . 
 \end{align} 
By  \eqref{1eq: J = W + W}, we further divide the integral $ \EuH_{\mu}^{\infty} \shskip  \phi (u) $ into two integrals containing $ \boldsymbol{W}_{ \mu}  (\pm \hskip -1.5pt \sqrt{uz}) $ and $\boldsymbol{E}^N_{ \mu }  (\hskip -1.5pt \sqrt{uz})$.  The change of variable $\pm \hskip -1.5pt \sqrt{z} \ra z$ turns the first integral into the form considered in Lemma \ref{lem: integration by parts}.  By an application of  Lemma \ref{lem: integration by parts} with $a = 0$, $b = 4 \hskip -1.5pt \sqrt{u}$, $c = |u|^{-\frac 12 \delta}$, $f (z) = \txw_{\infty} (|z|^2) \phi (z^2)   \boldsymbol{W}_{ \mu }   ( \hskip -1.5pt \sqrt{u} z)  |z|^2 $,   $S = 1 / \hskip -1.5pt \sqrt{u}$, and $\gamma = \beta -   1 / 2 $, we infer that the first integral is bounded by 
\begin{align}\label{7eq: bound for H, 2.1}
	\frac {|u|^{-\frac 1 2 }} {|u|^{- \delta (M - \beta + \frac 12  )} |u|^{M} } = |u|^{  -  \beta  \delta  -  (1-\delta) (M+\frac 1 2 )}. 
\end{align}
By \eqref{1eq: J = W + W} and \eqref{1eq: derivatives of E},   the second integral is trivially bounded by 
\begin{align}\label{7eq: bound for H, 2.2}
	\viint_{\BA [ |u|^{- \delta } ,  \, \infty ) }    {|z|^{\beta -2 }  }   {\, |zu|^{- \frac 1 2 (N+1)} } | d      z \vwedge d      \widebar{z} | \Lt |u|^{-    \beta \delta - \frac 1 2 (1-\delta) (N+1)}. 
\end{align}
It follows from \eqref{7eq: bound for H, 1}, \eqref{7eq: bound for H, 2.1}, and \eqref{7eq: bound for H, 2.2} that 
\begin{align}
	\EuH_{\mu} \shskip  \phi (u) = O_{\mu, \phi} \big(|u|^{-  (2 + \valpha)\delta } \big), 
\end{align}
as long as 
\begin{align*}
\frac {2M+1} 2, \ \frac {N+1} 2    \geqslant \frac { (2+\valpha -\beta )\delta}  {1-\delta } . 
\end{align*}
Consequently, we obtain \eqref{7eq: bound for H} upon setting $ \delta  =  1-  (2+\valpha) / \vepsilon $. Finally, the uniformity of the implied constant is clear from the statement of Lemma  \ref{lem: integration by parts}.  
\end{proof}

\subsection{Estermann function over \protect \scalebox{1.08}{$\BQ (i)$}} Let $s, \mu \in \BC$ and $2 h \in \BZ$. For $(a, c) = (1)$,  define the Estermann function
\begin{align}\label{7eq: defn of Estermann}
	E (s,  \mu,   h ; a/c) = \sum_{ n \shskip  \in    \CaloO \smallsetminus \{0\} } \frac {\vtau_{\mu} (n) e (\Re (an/c)) (n/|n|)^{2 h}} {|n|^{2s}}, \qquad \text{($\Re (s \pm \mu)  > 1$)}.  
\end{align}
Note that $ E (s,  \mu,   h ; - a/c) = (-1)^{2h} E (s,  \mu,   h ; a/c) $.

Now, in slightly  altered notation, we record here \cite[Lemma 3]{Motohashi-ImQuad}. 

\begin{lem}\label{lem: Estermann}
The Estermann function	$ E (s, \mu, h; a/c)$ has  analytic continuation to the whole $s$-plane except that in the case $h = 0$ there are two  simple poles at $ s = 1 \pm \mu $ with residue $ \pi |c|^{\mp 4 \mu -2} \zeta_{\RF} (1\pm 2 \mu) $ for $\mu \neq 0$. Moreover,  the following functional equation holds{\rm:}
\begin{align}\label{7eq: FE Estermann}
	E (s, \mu,  h ; a/c) = \lp   {\pi } / {|c|} \rp^{4s - 2} (c/|c|)^{4 h} \Gamma (s, \mu, h) E (1-s,   \mu, -  h; \widebar{a}/ c) ,
\end{align}
with $ a \widebar{a} \equiv 1 (\mathrm{mod}\, c)$, and 
\begin{align}
\Gamma (s, \mu, h) = 	\frac {  \Gamma (1-s+\mu+|h|) \Gamma (1-s-\mu+ |h|) } {\Gamma (s-\mu+|h|) \Gamma (s+\mu+|h|)} . 
\end{align}
\end{lem}

By the functional equation and the Phragm\'en-Lindel\"of convexity principle, it is clear that $ E (s,  \mu,   h ; a/c) $ is of moderate growth for $s $ on vertical strips. 

By the  Euler reflection formula, we may transform $\Gamma (1-s, \mu, h)$ into
\begin{align*}
	\frac { 	 (    \cos (2\pi\mu) -	(-1)^{2h}  \cos (2\pi s)    )    \Gamma \lp   {s+\mu + h}    \rp \Gamma \lp   {s+\mu - h}    \rp    \Gamma \lp   {s-\mu +h}    \rp \Gamma \lp   {s-\mu - h}    \rp  } {2 \pi^2}  , 
\end{align*} 
and this should be compared with the right-hand side  of  \eqref{3eq: Mellin} in Lemma \ref{lem: Mellin}. 

For brevity,  in the case $h = 0$ we shall  use  the notation $ E (s, \mu; a/c)  $ and $\Gamma (s, \mu)  $. 

\subsection{Bounds for  \protect \scalebox{1.06}{$  {J}_{\nu} (z) $} and \protect \scalebox{1.06}{$ \boldJ_{\nu} (z) $}} \label{sec: uniform bounds}

Recall from \eqref{0def: J mu m (z)} and \eqref{0eq: defn of Bessel} that 
\begin{align}\label{7eq: defn of J(z)}
	\boldJ_{\nu } (z) = \frac 1 { \sin (\pi \nu)} \big( J_{- \nu }    (4 \pi  z)   J_{- \nu }       (4 \pi  {\widebar z} ) -  J_{ \nu }    (4 \pi  z)   J_{ \nu }       (4 \pi  {\widebar z} ) \big) . 
\end{align}

For $|z| \Lt 1$, it follows from \cite[3.13 (1)]{Watson}  that   
\begin{equation}\label{7eq: bound for J}
	J_{\nu } (z)  \Lt \left| \frac {(z/2)^{\nu }} {\Gamma (\nu +1)}  \right| ,  
\end{equation} 
if $\Re (\nu) \geqslant 0$. 

For $|z| \Gt 1$, it is proven in \cite[Lemma 4.1]{Qi-Gauss} that 
\begin{equation}\label{7eq: bound for J, 2}
	  \boldJ_{\nu } (z)  \Lt \frac {  |\nu|^2 + 1} {|z|}  ,  
\end{equation} 
if $\Re (\nu) = 0$.

Actually, these bounds are very crude, but sufficient to verify (absolute) convergence. For instance, the convergence of   the off-diagonal $c$-sum in the Kuznetsov formula \eqref{1eq: Kuznetsov} may be easily seen from \eqref{7eq: Weil}, \eqref{7eq: bound for J}, and contour shift of the Bessel integral (up to $ \Re ( i r) = \vlambda   $ for $\vlambda >   1 / 4 $):  \begin{align}\label{7eq: H(z), 2}
	\mathrm{B} (z; h)   = 2 \int_{-\infty}^{\infty}   \frac {J_{2 i r }    (4 \pi  \hskip -1pt \sqrt{z})   J_{ 2i r  }    (4 \pi  \hskip -1pt \sqrt{\widebar z} )} {\sin (2\pi i r) }  h (r)   r^2 d r,
\end{align}  
by the fact that $ h (r)$ is even. 

\subsection{Bounds for \protect \scalebox{1.06}{${_{\protect\phantom{\nu}}^{\rho} \hskip -1.5pt f_{\nu}^{\mu}} (z)$} and \protect \scalebox{1.06}{$ {_{\protect\phantom{\nu}}^{\rho} \hskip -1pt \boldF^{\mu }_{\nu }} (z)$}} \label{sec: uniform bounds, hyper} Recall from \eqref{1eq: hypergoem, C, 1.0}--\eqref{1eq: hypergeometric, 3} that 
\begin{align} 
	{_{\phantom{\nu}}^{\rho} \hskip -1pt \boldF_{\nu }^{\mu }} (z) = {_{\phantom{\nu}}^{\rho} \hskip -1.5pt {C}_{\nu }^{\mu }} \cdot {_{\phantom{\nu}}^{\rho} \hskip -2pt f_{\nu }^{\mu }} (z) \, {_{\phantom{\nu}}^{\rho} \hskip -2pt f_{\nu }^{\mu }} (\widebar z)  + {_{\phantom{\nu}}^{\rho} \hskip -1.5pt {C}_{-\nu }^{\mu }} \cdot {_{\phantom{\nu}}^{\rho} \hskip -2pt f_{-\nu }^{\mu }} (z) \, {_{\phantom{\nu}}^{\rho} \hskip -2pt f_{-\nu }^{\mu }} (\widebar z) , 
\end{align}
with
\begin{equation} 
	{_{\phantom{\nu}}^{\rho} \hskip -1.5pt f_{\nu}^{\mu}} (z) =    z^{\frac 1 2  \nu }	{_2F_1} \hskip -1pt \bigg(   \frac {\rho + \nu +\mu} 2, \frac {\rho+\nu - \mu} 2  ; 1+ \nu ;  z    \bigg) ,
\end{equation} 
\begin{align}
	{_{\phantom{\nu}}^{\rho} \hskip -1.5pt {C}_{\nu }^{\mu }} = \frac {\cos (\pi(\rho+\nu))   -   	  \cos (\pi\mu)} {\sin (\pi \nu)} \frac {\Gamma  (   ({\rho + \nu +\mu}) / 2    )^2 \Gamma  (     {(\rho+\nu - \mu)}/ 2     )^2} {\Gamma (1+\nu)^2}  . 
\end{align}

Again, we need to estimate ${_{\phantom{\nu}}^{\rho} \hskip -1.5pt f_{\nu}^{\mu}} (z)$ and $ {_{\phantom{\nu}}^{\rho} \hskip -1pt \boldF^{\mu }_{\nu }} (z)$ uniformly for $\nu$ in a suitable vertical strip.  To this end, we invoke the following integral representation (see \cite[\S 2.5]{MO-Formulas})
\begin{equation}\label{7eq: hyper integral}
	\begin{split}
		{_2F_1} (a,b;c;z) = \frac {\Gamma (c)} {\Gamma (b) \Gamma (c-b)} & \int_0^1    \frac {t^{b-1} (1-t)^{c-b-1}} {(1-tz)^{a}} d t  , \qquad  \Re (c) > \Re (b) > 0,  
	\end{split}
\end{equation}
for $|\arg (1-z)| < \pi $.  

 Let $z \neq 0,  1$. For simplicity, let $\Re (\mu) = 0$.   Set $ \beta = \Re (\rho)$ and $\vlambda   = \Re (\nu)$.  Let   $\rho$ and $\mu$  be fixed, and $ 0 \leqslant \vlambda < 2 $ say. Assume $ - \vlambda < \beta < \vlambda + 2$.  Trivially, for $|z| \leqslant 1$,   
\begin{align*}
	\left|\int_0^1    \frac {t^{\frac 1 2 (\rho+\nu - \mu)-1} (1-t)^{- \frac 1 2 (\rho-\nu - \mu)}} {(1-tz)^{\frac 1 2 (\rho+\nu + \mu)}} d t \right| \Lt    \int_0^1    \frac {t^{\frac 1 2 (\beta+\vlambda )-1} (1-t)^{- \frac 1 2 (\beta-\vlambda )}} {e^{- \frac 1 2 \pi |\nu|} (1-t \Re (z) )^{\frac 1 2 (\beta+\vlambda )}} d t, 
\end{align*}
and it follows from \eqref{7eq: hyper integral} and the Stirling formula that  
\begin{align}\label{7eq: bound for f} 
 {_{\phantom{\nu}}^{\rho} \hskip -1.5pt f_{\nu}^{\mu}} (z) \Lt_{\rho, \shskip \mu}  
\exp (\pi |\nu|/2)	{\textstyle \sqrt{|\nu|+1}} \cdot {_{\phantom{\nu}}^{\beta} \hskip -1.5pt f_{\vlambda}^{0}} (\Re (z) ) . 
\end{align} 
For $|z| > 1$, in view of the reciprocity formula \eqref{2eq: reciprocity} in Lemma \ref{lem: reciprocity},   similar argument yields
\begin{align}\label{7eq: bound for F}
	    {_{\phantom{\nu}}^{\rho} \hskip -1pt \boldF_{\nu }^{\mu }} (z) \Lt_{\rho, \shskip \mu} \exp (2\pi |\nu| )   (|\nu| + 1)^{2\beta - 2}       \cdot    |z|^{ - \beta} \log (2|z|), 
\end{align}
if $\Re (\nu) = \Re (\mu) = 0$ (the $\log (2|z|)$ is removable if $\mu \neq 0$). 

As $h (r)$ is even,  the integral $ \Phi (u; h)$ defined by \eqref{1eq: Xi hyper, 2}, \eqref{1eq: Theta}, and \eqref{1eq: F} may be rewritten as 
\begin{align}\label{7eq: integral Phi}
\Phi (u; h) =	 \hskip -1pt    \int_{-\infty}^{\infty}   \hskip -2pt 	{_{\phantom{\nu}}^{1} \hskip -1pt {C}_{2ir }^{\shskip 0 }} \hskip -1pt  \cdot \hskip -1pt {_{\phantom{\nu}}^{1} \hskip -1.5pt f^{ \shskip 0 }_{2 i r }}  (1/u) {_{\phantom{\nu}}^{1} \hskip -1.5pt f^{ \shskip 0 }_{2 i r }}  (1/\widebar{u})  \cdot  h (r)  r^2 d r = \frac 1 2  \hskip -1pt \int_{-\infty}^{\infty}   \hskip -2pt 	 {_{\phantom{\nu}}^{1} \hskip -1pt \boldF^{ 0 }_{2 i r }}  (1/u) \cdot h (r)  r^2 d r . 
\end{align}
Thus $ \Phi (u; h)$ is convergent for any $u \neq 1$ due to \eqref{1eq: decay of h},  \eqref{7eq: bound for f}, and \eqref{7eq: bound for F}. To verify the convergence of the $n$-sum in the Kuznetsov--Motohashi formula in \eqref{1eq: Kuznetsov, Q(i)}, we   shift the contour of the first integral in \eqref{7eq: integral Phi} up to $\Re (2ir) =   \vlambda$ for $ \vlambda > 1 $. Note that  we need  $h  (- i/ 2  ) = 0$ as in \eqref{0eq: h (i/2)=0}   to pass through $r = - i/2$. 


\section{Proof of Theorem  \ref{thm: Kuz-Moto}: The Explicit Kuznetsov--Motohashi Formula over $\BQ (i)$} 
\label{sec: proof of K-M}

\subsection{Preparations}  

For $f \in \Pi_{c}^0$, by applying \cite[Theorem 4]{Molteni-L(1)} to its Rankin--Selberg $L$-function, we have
\begin{align}\label{8eq: Rankin-Selberg}
	\sum_{|n| \leqslant X} |   \vlambda_{f} ( n) |^2 \Lt_{\vepsilon}  (1 + |r_{f}|)^{\vepsilon} X, 
\end{align}
with the implied constant independent of $f$. Recall from \eqref{3eq: L-functions} that  
\begin{align}\label{8eq: L-functions}
	L (s, f) = \sum_{(n)   \subset   \CaloO }  \frac { \vlambda_{f} (n) } {|n|^{2  s} },  \qquad \text{($\Re (s) > 1$)}. 
\end{align}
By  the Rankin--Selberg bound \eqref{8eq: Rankin-Selberg}, the Dirichlet series in \eqref{8eq: L-functions} is indeed absolutely convergent for $\Re (s) > 1$. 

For  $\Re (s+\mu), \Re (s-\mu) > 1$, it follows from the Hecke relation \eqref{3eq: Hecke relation} that
\begin{align*} 
		L (s+\mu,  f   ) L (s-\mu, f) &  = \mathop{\sum\sum}_{(m), (n) \subset \CaloO}    \sum_{   (d) | (m, n )  } \frac { \vlambda_f  ( m n  / d^2  ) } {{|m|^{2s+2\mu} |n|^{2s-2\mu} }}  \\
		& = \sum_{ (d) \subset \CaloO} \frac 1 {|d|^{4 s } } \mathop{\sum\sum}_{(m), (n) \subset \CaloO} \frac {\vlambda_f (m n)   } {|m|^{2s+2\mu} |n|^{2s-2\mu} } ,  
\end{align*} 
and hence
\begin{align}\label{8eq: L(v,f) L(w,f)}
	\frac{L (s+\mu,  f   ) L (s-\mu, f)} { \zeta_{\RF} (2s) } =  \sum_{(n) \subset \CaloO } \frac {\vlambda_f (n) \vtau_{\mu} (n)} {|n|^{2s}} . 
\end{align}
Similarly, we have the Ramanujan identity:
\begin{align}\label{8eq: L(v,f) L(w,f), 2}
	\frac {\zeta_\RF(s+\mu+ir) \zeta_\RF(s+\mu-ir) \zeta_\RF(s-\mu+ir) \zeta_\RF(s-\mu-ir)} {\zeta_{\RF} (2s)} =  \sum_{(n) \subset \CaloO  }   \frac {\vtau_{ir} (n) \vtau_{\mu} (n)} {|n|^{2s}} . 
\end{align}

We have the (harmonic weighted) Weyl law:
\begin{align}\label{8eq: Weyl law}
	 \sum_{ f \in \Pi_{c}^0 : |r_f| \leqslant T} \frac 1 {L (1, \mathrm{Sym}^2 f) } \sim \frac {1} {3 \pi^3} T^3. 
\end{align}
This may be easily proven by the analysis in \cite{Qi-Liu-Moments}.  

\subsection{Setup} 

Let $m \in \CaloO \smallsetminus \{0\}$.  Let $h \in \mathscr{H}$ be as in Theorem \ref{thm: Kuz-Moto}. 
Let $s, \mu \in \BC$, with  $\Re (\mu) = 0$.   Define
\begin{align}\label{8eq: defn C}
	\SC_{\mu}^{s} ( m; h)  = \sum_{f \in \Pi_{c}^0}   \frac { L (s+\mu, f) L (s-\mu, f)       } {L(1, \mathrm{Sym^2} f )} \vlambda_f ( m )    h  ( r_f ). 
\end{align} 
Define $ \SE_{\mu}^{s} ( m; h) $  to be
\begin{equation}\label{8eq: defn E}
	\begin{split}  \frac 1 {\pi} &  \int_{-\infty}^{\infty} \hskip -2pt   
		\frac {\zeta_\RF(s+\mu+ir) \zeta_\RF(s+\mu-ir) \zeta_\RF(s-\mu+ir) \zeta_\RF(s-\mu-ir)   } {|\zeta_\RF(1+2i r)|^2} \vtau_{i r} ( m ) h ( r )  \shskip   d  r ,
	\end{split}
\end{equation} 
if $\Re (s) > 1$, and, similar to \cite[(2.30)]{Motohashi-JNT-Mean} (see also \cite[\S 3.3]{Motohashi-Riemann}), 
\begin{equation}\label{8eq: defn E, 2}
	\begin{split}     \frac 1 {\pi}    \int_{-\infty}^{\infty} \hskip -2pt   
		\frac {\zeta_\RF(s+\mu+ir) \zeta_\RF(s+\mu-ir) \zeta_\RF(s-\mu+ir) \zeta_\RF(s-\mu-ir)   } {|\zeta_\RF(1+2i r)|^2} \vtau_{i r} ( m ) h ( r )  \shskip   d  r & \\
		  +   4 \frac { \zeta_{\RF} (2s-1) \zeta_{\RF} (1-2\mu)   } {\zeta_{\RF} (3-2s-2\mu) } \vtau_{s+\mu-1} (m) h (i(s+\mu-1)) &  \\ 
		  +    4 \frac {\zeta_{\RF} (2s-1)  \zeta_{\RF} (1+2\mu)   } {\zeta_{\RF} (3-2s+2\mu) } \vtau_{s-\mu-1} (m) h (i(s-\mu-1))  & , 
	\end{split}
\end{equation} 
if $\Re (s) < 1$.

For $\Re (s )  > 1$,   it follows from \eqref{8eq: L-functions} and \eqref{8eq: L(v,f) L(w,f), 2}  that 
\begin{align}\label{8eq: L(1)} 
	 \frac {\SC_{\mu}^{s} ( m; h)} {\zeta_{\RF} (2s)} =   \sum_{(n) \subset \CaloO} \frac { \vtau_{\mu} (n)} {|n|^{2 s}}  \sum_{f \in \Pi_{c}^0}   \frac {  \vlambda_f ( m )  \vlambda_f ( n )   } {L(1, \mathrm{Sym^2} f )} \    h  ( r_f )  ,
\end{align} 
\begin{align}\label{8eq: E(1)} 
\frac {\SE_{\mu}^{s} ( m; h)} {\zeta_{\RF} (2s)} =    \sum_{(n) \subset \CaloO} \frac { \vtau_{\mu} (n)} {|n|^{2s}} \cdot  \frac 1 8 \int_{-\infty}^{\infty} \hskip -2pt   
	\frac {\vtau_{i r} (m )  \vtau_{i r} ( n )} {|\zeta_\RF(1+2i r)|^2} h ( r )  \shskip   d  r  . 
\end{align}
Note that the order of summations (or integration) has been changed, since, for example, the absolute convergence of the sum in \eqref{8eq: L(1)}  is clearly guaranteed by \eqref{8eq: Rankin-Selberg}, \eqref{8eq: Weyl law}, and the rapid decay of $h (r)$. 

\subsection{Application of Kuznetsov}
By applying the Kuznetsov formula in Lemma \ref{lem: Kuznetsov}, we have 
\begin{align}\label{8eq: after Kuznetsov}
\SC_{\mu}^{s} ( m; h) +	 \SE_{\mu}^{s} ( m; h) = \zeta_{\RF} (2s) \big( \SD_{\mu}^{s} ( m; h) + \ScrO_{\mu}^{s} ( m; h)   \big)  , 
\end{align}
with the diagonal term
\begin{align}\label{8eq: D(m;h)}
	\SD_{\mu}^{s} ( m; h) =     \frac { \vtau_{\mu} (m)} {2\pi^3 |m|^{2 s}} \int_{-\infty}^{\infty} h (r) r^2 d r   ,
\end{align}
and the off-diagonal sum 
\begin{align}
	\ScrO_{\mu}^{s} ( m; h) = \hskip -2pt  \sum_{(n) \subset \CaloO} \sum_{ \epsilon  \shskip \in   \CaloO^{\times} \hskip -1pt / \CaloO^{\times 2} } \sum_{c \shskip  \in   \CaloO \smallsetminus \{0\} }   \hskip -2pt  \frac {   S_{\RF} (  m ,  \epsilon n  ; c  ) \vtau_{\mu} (n) } {4 \pi |c|^2 |n|^{2s}} \hskip -2pt  \int_{-\infty}^{\infty} \hskip -2pt h (r)  \boldJ_{2ir} \bigg( \hskip -2pt \frac { {\sqrt{\epsilon m n}} } {  2 c     }  \bigg) r^2 d r   . 
\end{align}
Now we combine the $\epsilon$- and $(n)$-sums into an $n$-sum, and fold the $c$-sum into a $(c)$-sum, so that 
\begin{align}
	\ScrO_{\mu}^{s} ( m; h) = \hskip -1pt  \sum_{n \shskip  \in \CaloO \smallsetminus \{0\} }   \sum_{ (c) \subset \CaloO }     \frac {   S_{\RF} (   m ,   n  ; c  ) \vtau_{\mu} (n) } {2 \pi |c|^2 |n|^{2s}}   \int_{-\infty}^{\infty} h (r)  \boldJ_{2ir} \bigg( \frac { {\sqrt{  m n}} } {  2 c     }  \bigg) r^2 d r   . 
\end{align}
Note that a factor $2$ arises in the process. 

Next, we wish to apply the Vorono\"i--Oppenheim summation formula to the $n$-sum, so we need to move it inside the $r$-integral and the $(c)$-sum.  To this end, let us assume for the moment  \begin{align*}
	\Re (s  )  > \frac 5 4  
\end{align*}  to ensure absolute convergence. According as $ |c| \Gt \hskip -1.5pt \sqrt{|mn|} $ or $ |c| \Lt \hskip -1.5pt \sqrt{|mn|} $, the convergence may be easily verified by \eqref{7eq: bound for J} or \eqref{7eq: bound for J, 2},  along with the contour shift  as in \S \ref{sec: uniform bounds} in the former case.    By moving the $n$-sum inside, we obtain
\begin{align}\label{8eq: O (m;h) = O (m;c)}
	 \ScrO_{\mu}^{s} ( m; h) = \sum_{ (c) \subset \CaloO } \frac { 1 } {2 \pi  } \int_{-\infty}^{\infty}  \sum_{n \shskip  \in \CaloO \smallsetminus \{0\} } \frac {   S_{\RF} (   m ,   n  ; c  ) \vtau_{\mu} (n) } {|c|^2 |n|^{2 s}} \boldJ_{2 i r} \bigg( \frac { {\sqrt{  m n}} } {  2 c     }  \bigg)  h (r) r^2 d r. 
\end{align} 
Since the $(c)$-sum is outermost, it is legitimate to   shift the contour to $ \Re (ir) = 0 $ again. However, in view of \eqref{2eq: bound 1/z}, for the inner  $n$-sum  to be convergent, it only requires 
\begin{align*}
	\Re (s) > \frac 3 4;
\end{align*}
indeed, in our later analysis, we shall mainly work in the region: 
\begin{align}\label{8eq: Re(s), 3}
	\frac 3 4 < \Re(s) < 1. 
\end{align}

Kuznetsov's original idea is to apply Vorono\"i--Oppenheim to the inner $n$-sum, and subsequently evaluate the Bessel integrals and the resulting $(c)$-sum. However, this can be done only in the {\it formal} manner due to the convergence issue after  Vorono\"i--Oppenheim, even though it will eventually yield  the right formula. Our idea of rectifying Kuznetsov's argument is to introduce a regularization prior to the application of Vorono\"i--Oppenheim. 

\subsection{Regularization} \label{sec: regularization}
Define 
\begin{equation}\label{8def: R (z)}
	A_{\nu } (z) = \frac {(|z|/2)^{2\nu} } {\Gamma (\nu+1)^2} + \frac {(|z|/2)^{2\nu} ((z/2)^2 + (\widebar{z}/2)^2)} {\Gamma (\nu+1) \Gamma (\nu+2) } ,
\end{equation}
and
\begin{equation}\label{8eq: defn of R}
	\boldsymbol{R}_{ \nu  } (z)  =   \frac {1} {\sin (\pi \nu)} \lp A_{-\nu } (4 \pi   z) - A_{\nu } (4 \pi   z)  \rp. 
\end{equation}
Let $\txw (x) \in C^\infty(\BR_+)$ be such that  $\txw (x) \equiv 1$ for $x$ near $0$ and $\txw (x) \equiv  0$ for $x$ large. Define 
\begin{align}
	\boldsymbol{M}_{\nu} (z) = \boldJ_{ \nu } (z) - \txw (|2 \pi z|^2 )  \boldsymbol {R}_{ \nu } (z). 
\end{align} 
Note that if $\Re (\nu) = 0$ and $\nu \neq 0$, then 
\begin{align}\label{8eq: bound for M}
	 \boldsymbol{M}_{\nu} (z) = O (|z|^{4}), 
\end{align}
for $|z| \Lt 1$. 
 
According to 
\begin{align}
	\boldJ_{ 2ir  } (\hskip -1.5pt \sqrt{z}) =   \txw ( (2\pi)^2 |z|)  \boldsymbol {R}_{ 2ir } (\hskip -1.5pt \sqrt{z}) + \boldsymbol{M}_{2ir} (\hskip -1.5pt \sqrt{z})  , 
\end{align}
we split $ \ScrO_{\mu}^{s} ( m; h)  $ into
\begin{align}
	 \ScrO_{\mu}^{s} ( m; h) =  \SR_{\mu}^{s} ( m; h) +  \SM_{\mu}^{s} ( m; h) , 
\end{align}
with 
\begin{equation}\label{8eq: R (m;h) }
	\SR_{\mu}^{s} ( m; h) = \sum_{ (c) \subset \CaloO } 	\SR^s_{  \mu } (    m; c; h)  , 
\end{equation} 
\begin{equation}\label{8eq: R (m;c)}
	\SR^s_{  \mu } (    m; c; h)  \hskip -1pt   =  \hskip -4pt \sum_{n \shskip  \in \CaloO \smallsetminus \{0\} }  \hskip -4pt \frac { S_{\RF} (   m ,   n  ; c  ) \vtau_{\mu} (n) } {|c|^2 |n|^{2 s}} \txw \hskip -1pt \left(   \left|\frac { \pi^2 m n } {    c^2     } \right| \right) \hskip -1pt \cdot \frac { 1 } {2 \pi  } \hskip -1pt  \int_{-\infty}^{\infty}  \hskip -1pt  \boldsymbol{R}_{2ir} \bigg( \hskip -1.5pt \frac { {\sqrt{  m n}} } {  2 c     }  \bigg) h (r) r^2 d r ,
\end{equation}
and
\begin{equation}\label{8eq: M (m;h) }
	\SM_{\mu}^{s} ( m; h) = \sum_{ (c) \subset \CaloO } \frac { 1 } {2 \pi  } \int_{-\infty}^{\infty}  	\SM^s_{ir, \shskip \mu } (    m; c) h (r) r^2 d r, 
\end{equation} 
\begin{equation}\label{8eq: M (m;c)}
		\SM^s_{\nu, \shskip \mu}  ( m; c) =  \sum_{n \shskip  \in \CaloO \smallsetminus \{0\} } \frac {  S_{\RF} (   m ,   n  ; c  ) \vtau_{\mu} (n)} {|c|^2 |n|^{2 s}} \boldsymbol{M}_{2\nu} \bigg( \frac { {\sqrt{  m n}} } {  2 c     }  \bigg) . 
\end{equation}

Next, we need to treat  $\SR^s_{  \mu } (    m; c; h) $   and $\SM^s_{\nu, \shskip \mu}  ( m; c)$ as in   \eqref{8eq: R (m;c)} and   \eqref{8eq: M (m;c)}. For the former, we use  the standard Mellin technique and the functional equation for Estermann function as in Lemma \ref{lem: Estermann}. For the latter, we apply   Vorono\"i--Oppenheim as in Lemma \ref{lem: Voronoi} and Corollary \ref{cor: V-O}. By \eqref{0def: H mu m (z)}--\eqref{1eq: derivatives of E} and \eqref{8eq: bound for M},    this is legitimate since $ |z|^{-2s}  \boldsymbol{M}_{2 i r}  ( \hskip -1.5pt  { {\sqrt{  m z}} } / {  2       }   ) $ is in the space $ \mathscr{P}^{\infty}_{\valpha, \shskip \beta}  (\BC^{\times}; \hskip -1.5pt  { {\sqrt{  m  }} } / {  2       }  ) $ as long as   
\begin{align}\label{8eq: alpha, beta}
	  \frac 3 2 - 2 \Re(s) < \beta  < 0 < \valpha < 2-2\Re(s) .  
\end{align}
 Note that the inequalities above are guaranteed by \eqref{8eq: Re(s), 3}. At the end, the $(c)$-sum will become separated and can be explicitly evaluated by the Ramanujan identity \eqref{7eq: sum of (c)}--\eqref{7eq: sigma(0)}.

\subsection{Application of the functional equation for Estermann functions}

On opening the Kloosterman sum $S_{\RF} (   m ,   n   ;  \allowbreak c  )$ as in \eqref{2eq: defn Kloosterman KS} and applying the Mellin inversion (on $\BR_+$)
\begin{align*}
	\txw (x) = \frac 1 {2\pi i} \int_{\shskip \sigma -i\infty}^{\sigma + i \infty}  \widetilde{\txw} (\rho) x^{-\rho} d \rho ,  \qquad \text{($\sigma > 0$)},
\end{align*}
with 
\begin{align*}
	 \widetilde{\txw} (\rho) = \int_{ \BR_+} \txw (x) x^{\shskip \rho-1} d x, 
\end{align*} 
we may transform $ \SR^s_{  \mu } (    m; c; h) $ (see \eqref{8eq: R (m;c)}) into an expression that involves the Estermann functions $E (s   , \mu, 0;  {a}/c)$ and  $E  (s   , \mu, \pm 1/2;  {a}/c)$ as defined by \eqref{7eq: defn of Estermann}. 
For simplicity, we consider   only the contribution $ \SP^s_{  \mu } (    m; c; h) $ from $\boldsymbol{P}_{ 2ir  } (z)$,  defined as in \eqref{0def: P mu m (z)} and \eqref{0eq: defn of P} by 
\begin{align}\label{7eq: defn of P}
	 	\boldsymbol{P}_{ \nu  } (z) =  \frac {1} {\sin (\pi \nu)} \lp D_{-\nu } (4 \pi   z) - D_{\nu } (4 \pi   z)  \rp , \qquad D_{\nu } (z) = \frac {(|z|/2)^{2\nu} } {\Gamma (\nu+1)^2}; 
\end{align}   that from the two terms of higher order (see \eqref{8def: R (z)}) can be calculated in the same way. 
More explicitly,  for $\sigma/2 > 1 - \Re (s)$, we have 
\begin{equation}
  \begin{split}
\SP^s_{  \mu } (    m; c; h) =   &	\sum_{a \, \in (\CaloO /c)^{\times } }    \frac {e   (   \Re  (   {\widebar{a} m     }/ {c}  )   )} {|c|^2}   \frac { 1 } {  \pi^2  }     \int_{-\infty}^{\infty} \frac { \Gamma (2ir) } {   \Gamma (1-2ir) }  h (r) r^2   \\
 \cdot  &       \frac 1 {2\pi i} \int_{\shskip \sigma -i\infty}^{\sigma + i \infty}  E (s + i r + \rho/2    , \mu;  {a}/c) 
     \cdot  \widetilde{\txw}  (\rho)     \big| c^2 / \pi^2 m   \big|^{  2 ir + \rho  }  d \rho \,  d r    . 
  \end{split}
\end{equation}  
Note that we have used the expression of $\boldsymbol{P}_{ 2ir  } (z)$ as in \eqref{7eq: defn of P} and the fact that $h (r)$ is an even function. 
Now we shift the  contour of the $\rho$-integral to the left and 
collect the residues at $ \rho/2   = 1   -s - ir \pm \mu $ as in Lemma \ref{lem: Estermann}, obtaining
\begin{align}\label{8eq: residue}
	 \frac {2 S (m, 0; c) } {\pi |c|^{4s}} \sum_{\pm}   & \frac { \zeta_{\RF} (1 \pm 2 \mu) } {|\pi^2m|^{2-2s \pm 2\mu}}  \int_{-\infty}^{\infty}   \widetilde{\txw} (2(1-s - ir \pm \mu))  \frac { \Gamma (2ir) } {   \Gamma (1-2ir) }  h (r) r^2 d r  . 
\end{align}
Next, we shift the   contour of the $r$-integral up to $\Im (r) =   \vlambda$, with 
  \begin{align*}
	\Re (s) + \frac {\sigma} 2 < \vlambda < 1, 
\end{align*}  and then apply the functional equation \eqref{7eq: FE Estermann} in Lemma   \ref{lem: Estermann}. Note that it is safe to pass through $r = i/2$ by the condition $h (i/2) = 0$ as in \eqref{0eq: h (i/2)=0}.  For the resulting expression with $E (1-s - ir - \rho/2 , \mu; \widebar{a}/c)$, 
we replace it by the absolutely convergent Dirichlet series as in \eqref{7eq: defn of Estermann}. In this way, we arrive at    
\begin{equation}\label{8eq: integral}
	\begin{split}
		    	\sum_{ n \shskip  \in    \CaloO \smallsetminus \{0\} } \frac { S_{\RF} (m-n, 0; c) \vtau_{\mu} (n) } {|c|^{4s} |\pi^2 n|^{2- 2s}}       &     \int_{\, i\vlambda -\infty}^{i\vlambda + \infty} \frac { \Gamma (2ir) } {   \Gamma (1-2ir) }  h (r) r^2   \\
		\cdot        \frac 1 {2\pi i} & \int_{\shskip \sigma -i\infty}^{\sigma + i \infty}  \Gamma ( s + ir + \rho/2, \mu)  
		\cdot  \widetilde{\txw}  (\rho)      |  n/ m  |^{   2 ir + \rho  }  d \rho \,  d r    . 
	\end{split}
\end{equation}  
Finally, we insert   \eqref{8eq: residue} and \eqref{8eq: integral} into \eqref{8eq: R (m;h) } (denote this by $\SP_{\mu}^{s} ( m; h) $), evaluate the $(c)$-sum by the  the Ramanujan identity \eqref{7eq: sum of (c)}--\eqref{7eq: sigma(0)}, and multiply the $ \zeta_{\RF} (2s) $ as in \eqref{8eq: after Kuznetsov}. In conclusion, we have turned $ \zeta_{\RF} (2s) \SP_{\mu}^{s} ( m; h) $ into the sum of  the residual term
\begin{equation}\label{8eq: Mellin, Z}
\breve{\SZ}{}_{\mu}^{s} ( m; h)	=  
\frac {2 \sigma_{1-2s} (m) } {\pi  } \sum_{\pm}     \frac { \zeta_{\RF} (1 \pm 2 \mu) } {|\pi^2m|^{2-2s \pm 2\mu}}  \int_{-\infty}^{\infty}   \widetilde{\txw} (2(1-s - ir \pm \mu))  \frac { \Gamma (2ir) } {   \Gamma (1-2ir) }  h (r) r^2 d r  , 
\end{equation}
and the dual  sum 
\begin{equation}\label{8eq: Mellin, P}
	\begin{split}
		\widetilde{\hskip -1pt \SP}{}_{\mu}^{s} ( m; h)	= 
	\sum_{ n \shskip  \in    \CaloO \smallsetminus \{0\} }  &  \frac {\vtau_{\mu} (n) \sigma_{1-2s} (m-n)  } {  |\pi^2 n|^{2- 2s}}          \int_{\, i\vlambda -\infty}^{i\vlambda + \infty} \frac { \Gamma (2ir) } {   \Gamma (1-2ir) }  h (r) r^2   \\
 & \quad 	\cdot        \frac 1 {2\pi i}     \int_{\shskip \sigma -i\infty}^{\sigma + i \infty}  \Gamma ( s + ir + \rho/2, \mu)  
	\cdot  \widetilde{\txw}  (\rho)      |  n/ m  |^{   2 ir + \rho  }  d \rho \,  d r; 
	\end{split}
\end{equation}
the higher order terms in $	\boldsymbol{R}_{ 2ir   } (z)$ contribute to $   \zeta_{\RF} (2s) \SR_{\mu}^{s} ( m; h) $ a similar dual sum (no residual term!).  Let $ \, \widetilde{\hskip -1.5pt \SR}{}_{\mu}^{s} ( m; h)  $ denote the whole dual sum, so that 
\begin{align}\label{8eq: R = Z + R}
	  \zeta_{\RF} (2s) \SR_{\mu}^{s} ( m; h) = \breve{\SZ}{}_{\mu}^{s} ( m; h) + \, \widetilde{\hskip -1.5pt \SR}{}_{\mu}^{s} ( m; h) . 
\end{align} 

\begin{rem}
	Note that the diagonal term with $n = m$ requires an extra argument by analytic continuation, as the $(c)$-sum is convergent only for $\Re (s) > 1$ {\rm(}see {\rm\eqref{7eq: sum of c, 2}}{\rm)}. 
\end{rem}


\subsection{Application of Vorono\"i--Oppenheim} 

Recall that  $\SM^s_{\nu, \shskip \mu}  ( m; c) $ is defined as in \eqref{8eq: M (m;c)}.  It follows from \eqref{2eq: defn Kloosterman KS} that
\begin{align*}
	\SM^s_{\nu, \shskip \mu}  ( m; c) =  \sum_{a \, \in (\CaloO /c)^{\times } } e  \bigg(   \Re \bigg( \frac {\widebar{a} m     } {c} \bigg)  \bigg)      \sum_{n \shskip  \in \CaloO \smallsetminus \{0\} } e  \bigg(   \Re \bigg( \frac {  {a} n  } {c} \bigg)  \bigg)  \frac {\vtau_{\mu} (n)  } {|n|^{2 s}} \boldsymbol{M}_{2\nu} \bigg( \frac { {\sqrt{  m n}} } {  2 c     }  \bigg) , 
\end{align*} 
As discussed in \S \ref{sec: regularization}, it is legitimate to apply the Vorono\"i--Oppenheim summation formula (see Lemma \ref{lem: Voronoi} and Corollary \ref{cor: V-O}) to the $n$-sum.  Consequently, we have 
\begin{align}\label{8eq: O = Z + D + S}
\SM^s_{\nu, \shskip \mu}  ( m; c) =  	\check{\SZ}{}^s_{\nu, \shskip \mu} (m; c) 
+  \, \widetilde{\hskip -2pt \SM}{}^s_{\nu, \shskip \mu} (m; c) , 
\end{align}
where $   \check{\SZ}{}^s_{\nu, \shskip \mu} (m; c) $ is the zero frequency
\begin{align}\label{8eq: defn Z}
\check{\SZ}{}^s_{\nu, \shskip \mu} (m; c) = \frac {S_{\RF} (m, 0; c) } {   | c|^{4 s} } \sum_{\pm}  \frac { \zeta_{\RF} (1 \pm 2 \mu) } {2^{4s   \mp 4 \mu - 3} }   \viint  \boldsymbol{M}_{2\nu}  (  \hskip -1.5pt { {\sqrt{  m z}} }  )   \frac { i d      z \vwedge d      \widebar{z}} {|z|^{2s \mp 2\mu} },  
\end{align}
and $ \, \widetilde{\hskip -2pt \SM}{}^s_{\nu, \shskip \mu} (m; c) $  is the dual sum 
\begin{align}\label{8eq: defn S}
    \widetilde{\hskip -2pt \SM}{}^s_{\nu, \shskip \mu} (m; c)  =	\hskip -2pt  
\sum_{n \shskip  \in \CaloO \smallsetminus \{0   \} } \hskip -2pt \frac  {4 S_{\RF} (m-n, 0; c) \vtau_{\mu} (n  ) } { |2c|^{4s}}     \cdot 2\pi^2 \hskip -1pt \viint   \boldsymbol{M}_{2\nu}   (  \hskip -1.5pt { {\sqrt{  m z}} }  ) \boldJ_{ 2 \mu }  (  \hskip -1.5pt { {\sqrt{  n z}} }  ) \frac  {i d      z \vwedge d      \widebar{z}} {|z|^{2s}}. 
\end{align}
In view of Lemma \ref{lem: bound for H, 2} and Remark \ref{rem: implied const} (see also \eqref{8eq: alpha, beta}), after inserting  \eqref{8eq: defn S} into \eqref{8eq: M (m;h) }, we are free to rearrange the order of summation and integration. Now we move the $n$-sum  out so that the $(c)$-sum may be evaluated by the Ramanujan identity \eqref{7eq: sum of (c)}--\eqref{7eq: sum of c, 2}. 

\subsection{Cancellation} 
Recall that 
\begin{align*}
	\boldsymbol{M}_{\nu} (z) = \boldJ_{ \nu } (z) - \txw (|2 \pi z|^2 )  \boldsymbol {R}_{ \nu } (z). 
\end{align*} 
By definition or by the Mellin integral formula  in Lemma \ref{lem: Mellin} (applied to \eqref{8eq: defn Z} or \eqref{8eq: defn S}), one may easily prove that the contribution to $ \zeta_{\RF} (2s) \SM_{\mu}^{s} ( m; h) $ from $ \txw (|2 \pi z|^2 )  \boldsymbol {R}_{ 2ir  } (z)  $ is   canceled exactly with $ \zeta_{\RF} (2s) \SR_{\mu}^{s} ( m; h) $ (see \eqref{8eq: Mellin, Z}, \eqref{8eq: Mellin, P}, and \eqref{8eq: R = Z + R}). This is expected as in general the Vorono\"i summation is an incarnation of the functional equations (see \cite{Miller-Schmid-2004-1}).

After the cancellation, we arrive at 
\begin{align}\label{8eq: O = Z + D + O}
\zeta_{\RF} (2s)	 \ScrO_{\mu}^{s} ( m; h) =  \SZ_{\mu}^{s} ( m; h) + \widetilde{\SD}{}_{\mu}^{s} ( m; h) + \widetilde{\ScrO}{}_{\mu}^{s} ( m; h),   
\end{align}
where 
\begin{align}\label{8eq: defn of Z, 2}
	\SZ_{\mu}^{s} ( m; h) = \frac {   \sigma_{1-2s} (m) } {  \pi  }  \sum_{\pm} \frac {  \zeta_{\RF} (1 \pm 2 \mu) }   {   2^{4 s \mp 4  \mu - 2} } \int_{-\infty}^{\infty}       \viint  \boldJ_{2 i r}  (  \hskip -1.5pt { {\sqrt{  m z}} }  )   \frac { i d      z \vwedge d      \widebar{z}} {|z|^{2s \mp 2\mu} }  	\cdot   h (r) r^2 d r , 
\end{align}
\begin{align}\label{8eq: defn of D, 2}
	\widetilde{\SD}{}_{\mu}^{s} ( m; h) = \frac { \pi \zeta_{\RF} (2s-1) \vtau_{\mu} (m) } {2^{4s-2}} \int_{-\infty}^{\infty} \viint  \boldJ_{2 ir}   (  \hskip -1.5pt { {\sqrt{  m z}} }  ) \boldJ_{ 2 \mu }  (  \hskip -1.5pt { {\sqrt{  m z}} }  ) \frac  {i d      z \vwedge d      \widebar{z}} {|z|^{2s}} \cdot h (r) r^2 d r , 
\end{align}
\begin{equation}\label{8eq: defn of O, 2}
	\begin{split}
		 \widetilde{\ScrO}{}_{\mu}^{s} ( m; h) = \hskip -2pt  \sum_{n \shskip  \in \CaloO \smallsetminus \{0, \shskip m  \} } \hskip -2pt \frac { \pi \sigma_{1-2s} (n-m) \vtau_{\mu} (n) } {2^{4s-2}} \hskip -1pt \int_{-\infty}^{\infty} \hskip -1pt \viint  \boldJ_{2 ir}   (  \hskip -1.5pt { {\sqrt{  m z}} }  ) \boldJ_{ 2 \mu }  (  \hskip -1.5pt { {\sqrt{  n z}} }  ) \frac  {i d      z \vwedge d      \widebar{z}} {|z|^{2s}} \cdot h (r) r^2 d r . 
	\end{split}
\end{equation}
Note that we have extracted the diagonal term $ \widetilde{\SD}{}_{\mu}^{s} ( m; h) $ to pair with $\zeta_{\RF} (2s)  \SD_{\mu}^{s} ( m; h)$---the diagonal term from Kuznetsov (see \eqref{8eq: D(m;h)}).  

Our final step is to evaluate the integrals in \eqref{8eq: defn of Z, 2}--\eqref{8eq: defn of O, 2} by Corollaries \ref{cor: Hankel}, \ref{cor: Gauss}, and \ref{cor: Mellin, spherical}; 
the validity of these integral formulae is again guaranteed by \eqref{8eq: Re(s), 3}.

For simplicity, we introduce 
 \begin{equation}\label{8eq: Gamma}
 	\begin{split}
 		\Gamma ({1-\gamma},{\nu})  =   \frac { 	 (    \cos (2\pi\nu) -	  \cos (2\pi \gamma)    )    \Gamma \lp   {\gamma+\nu}    \rp^2   \Gamma \lp   {\gamma-\nu}    \rp^2 } {2 \pi^2}   , 
 	\end{split}
 \end{equation}
 as occurred in Corollaries \ref{cor: Gauss} and \ref{cor: Mellin, spherical}. By the Euler reflection formula, 
 \begin{equation}\label{8eq: Gamma, 2}
 	\begin{split}
 		\Gamma ({\gamma}, {\nu})  =  \frac {  \Gamma (1-\gamma+\nu) \Gamma (1-\gamma-\nu) } {\Gamma (\gamma-\nu) \Gamma (\gamma+\nu)} , 
 	\end{split}
 \end{equation}
 as the $ \Gamma ({\gamma}, {\nu}, 0) $ defined in Lemma \ref{lem: Estermann}.  

\subsection{The zero frequency} By applying Corollary \ref{cor: Mellin, spherical} to the integral in  \eqref{8eq: defn of Z, 2}, we have  
\begin{equation}\label{8eq: Z, final}
\SZ_{\mu}^{s} ( m; h) =	\frac { \sigma_{1-2s} (m)} {2\pi  |\pi^2 m|^{2-2s}} \sum_{\pm}  \frac {\zeta_{\RF} (1 \pm 2 \mu) } {  |\pi^2 m|^{ \pm 2\mu} }  \int_{-\infty}^{\infty}      \Gamma ({s \mp \mu}, ir)  h (r) r^2 d r              . 
\end{equation}

\delete{
\begin{align*}
	\SZ^s_{\nu, \shskip \mu} (m; c) = 	\frac 1 {|c|^{4 s}} \bigg(	\frac {\zeta_{\RF} (1 + 2 \mu) } {  |\pi^2 m|^{2-2s+2\mu} }        \Gamma ( {s-\mu}, {\nu} )    + \frac {\zeta_{\RF} (1 - 2 \mu) } {  |\pi^2 m|^{2-2s-2\mu}   }   \Gamma ( {s+\mu}, {\nu} )   \bigg) . 
\end{align*}   
The $(c)$-sum inside  $\SZ_{\mu}^{s} ( m; h)$ is  separated as 
\begin{align*}
	\sum_{(c)   \subset   \CaloO  }  \frac { S (m , 0; c)  } { |c  |^{4 s} } = \frac { \sigma_{1-2s} (m) } {  \zeta_{\RF} (2s) },  
\end{align*}
by \eqref{7eq: sum of (c)}. Therefore $\zeta_{\RF} (2s)	\SZ_{\mu}^{s} ( m; h)$ is equal to }


\subsection{The diagonal terms} \label{sec: diag terms}

By applying Corollary \ref{cor: Gauss} to the integral in  \eqref{8eq: defn of D, 2}, we have
\begin{align*}
\widetilde{\SD}{}_{\mu}^{s} ( m; h) \hskip -1pt  =	-  \frac {     \sin (2\pi s) \Gamma (2s-1)^2 \zeta_{\RF} (2s-1)  } {  2   \pi^{5-4s}     }  \frac {  \vtau_{\mu} (m)} {| m|^{2-2s}   }   \hskip -2pt    \int_{-\infty}^{\infty} \hskip -2pt   \Gamma (s+\mu, ir ) \Gamma (s - \mu, ir)   h (r)    r^2 d r . 
\end{align*}
Next we use the Euler reflection formula for $\Gamma (s)$ and the functional equation for $\zeta_{\RF} (s)$, 
\begin{align*}
	\Gamma (s) \zeta_{\RF} (s) = \pi^{2s-1} \Gamma (1-s) \zeta_{\RF} (1-s), 
\end{align*}
to simplify this into  
\begin{align}\label{8eq: D(m;h), 2}
	\widetilde{\SD}{}_{\mu}^{s} ( m; h) =         \frac {      \zeta_{\RF} (2-2s)  } {  2   \pi^{7-8s}    \zeta_{\RF} (2s)   }  \frac {  \vtau_{\mu} (m)} {| m|^{2-2s}   }       \int_{-\infty}^{\infty}   \Gamma (s+\mu, ir ) \Gamma (s - \mu, ir)  h (r)    r^2 d r.  
\end{align}
By combining \eqref{8eq: D(m;h)} and \eqref{8eq: D(m;h), 2}, we infer that $\zeta_{\RF}   (2s)    \SD_{\mu}^{s} ( m; h)   + \widetilde{\SD}{}_{\mu}^{s} ( m; h)     $ is equal to
\begin{equation}\label{8eq: D, final}
	\begin{split}
		\frac {  \vtau_{\mu} (m)} {2\pi^3 | m|   }   \int_{-\infty}^{\infty} \hskip -2pt \bigg( \frac {\zeta_{\RF}(2s)} {|m|^{2s-1}}  + \frac {\zeta_{\RF} (2-2s)} { |\pi^4 m|^{1-2s} }    \Gamma (s+\mu, ir ) \Gamma (s - \mu, ir) \bigg) h (r) r^2 d r     . 
	\end{split}
\end{equation}

\delete{ 
\begin{align*}
	\SD^s_{\nu, \shskip \mu} (m; c) = -  \frac {     \sin (2\pi s) \Gamma (2s-1)^2   } {     \pi^{4-4s}   }   \frac {  \vtau_{\mu} (m)} {| m|^{2-2s}  }  \frac 1 {|c|^{4s}}    \Gamma (s+\nu, \mu) \Gamma (s-\nu, \mu) . 
\end{align*}
Thus the $(c)$-sum inside  $\widetilde{\SD}{}_{\mu}^{s} ( m; h)$ is again separated as 
\begin{align*}
	\sum_{ (c) \subset \CaloO } \frac {\varphi (c)} {|c|^{4s}} = \frac {\zeta_{\RF} (2s-1)} {\zeta_{\RF} (2s)}, 
\end{align*}
by \eqref{7eq: sum of c, 2}.   Therefore }


\subsection{The dual sum} \label{sec: dual sum}

By applying  Corollary \ref{cor: Hankel}  to the integral in  \eqref{8eq: defn of O, 2}, we have 
\delete{\begin{align*}
\ScrS^s_{\nu, \shskip \mu} (m; c) = \sum_{n \shskip  \in \CaloO \smallsetminus \{0, \shskip m  \} }    \frac {S (m-n, 0; c)  } {|c|^{4s} }  	\frac {   \vtau_{\mu} (n  )   } { 2    | \pi^2 n|^{2-2s} } {_{\phantom{\nu}}^{2-2s} \hskip -1pt \boldF^{2\mu }_{2\nu }}  \Big(  \frac {m} {n} \Big). 
\end{align*}
Consequently, }
\begin{align}\label{8eq: O, final}
\widetilde{\ScrO}{}_{\mu}^{s} ( m; h) = \sum_{n \shskip  \in \CaloO \smallsetminus \{0, \shskip m  \} }   \frac {  \sigma_{1-2s} (n-m) \vtau_{\mu} (n  )   } {    4 \pi   | \pi^2 n|^{2-2s} }   \int_{-\infty}^{\infty}   	 {_{\phantom{\nu}}^{2-2s} \hskip -1pt \boldF^{2\mu }_{2 i r }}  \Big(  \frac {m} {n} \Big)  h (r)  r^2 d r . 
\end{align}

\subsection{Conclusion} By \eqref{8eq: defn C}, \eqref{8eq: defn E, 2}, \eqref{8eq: after Kuznetsov}, \eqref{8eq: O = Z + D + O}, and \eqref{8eq: Z, final}--\eqref{8eq: O, final}, we have established the analogue over $\BQ (i)$ of Kuznetsov's formulae in \cite[Theorems 3.1, 3.3]{Kuznetsov-Motohashi-formula}.\footnote{Kuznetsov missed the Eisenstein residual term (see \cite[(2.30)]{Motohashi-JNT-Mean}). }

Finally, by the passage to the limit as $ \mu \ra 0 $ and $s \ra   1 / 2$, it follows that 
\begin{equation}\label{8eq: C, final, 2}
	\SC_{0}^{\frac 1 2} ( m; h)  = \sum_{f \in \Pi_{c}^0}   \frac { L \big(\frac 1 2, f \big)^2       } {L(1, \mathrm{Sym^2} f )} \vlambda_f ( m )    h  ( r_f ), 
\end{equation} 
\begin{equation}\label{8eq: E, final, 2}
	\begin{split}  \SE_{0}^{\frac 1 2 } ( m; h) =   \frac 1 {\pi}    \int_{-\infty}^{\infty} \hskip -2pt   
		\frac {\left|\zeta_\RF\big(\frac 1 2 +ir\big)\right|^4  } {|\zeta_\RF(1+2i r)|^2} \vtau_{i r} ( m ) h ( r )  \shskip   d  r  - \frac {i \pi       } {4 \zeta_{\RF} (2 ) } \vtau_{-\frac 1 2} (m) h' (-i/2), 
	\end{split}
\end{equation} 
\begin{equation}\label{8eq: Z, final, 2}
	\begin{split}
		\SZ_{0}^{\frac 1 2 } ( m; h) & = \lim_{s \ra \frac 1 2} \big( \zeta_{\RF}   (2s)    \SD_{0}^{s} ( m; h)   + \widetilde{\SD}{}_{0}^{s} ( m; h)   \big)  \\ 
		& =	\frac {\vtau  (m)} { \pi^3  | m| }   \int_{-\infty}^{\infty}   \lp \gamma_{\RF} -  \frac {\pi} 4 \log |\pi^2m| + \frac {\pi} 2 \frac{\Gamma'} {\Gamma}  \lp \frac 1 2 + ir \rp \rp  h (r) r^2 d r      ,   
	\end{split}    
\end{equation} 
\begin{align}\label{8eq: O, final, 2}
	\widetilde{\ScrO}{}_{0}^{\frac 1 2 } ( m; h) = \sum_{n \shskip  \in \CaloO \smallsetminus \{0, \shskip m  \} }   \frac {  \vtau (n-m) \vtau  (n  )   } {    4 \pi^3   |  n|  }   \int_{-\infty}^{\infty}   	 {_{\phantom{\nu}}^{1} \hskip -1pt \boldF^{ 0 }_{2 i r }}  \Big(  \frac {m} {n} \Big)  h (r)  r^2 d r ,
\end{align}
where for \eqref{8eq: E, final, 2}   we have used 
\begin{align*}
  \mathop{\mathrm{Res}}_{s=1}   \zeta_{\RF} (s) = \frac {\pi} 4, \qquad \zeta_{\RF} (0) = -\frac 1 4. 
\end{align*}
Now the proof of Theorem \ref{thm: Kuz-Moto} is completed by rearrangement.

	
	\def\cprime{$'$}

\end{document}